\documentclass[ba]{imsart}

\pdfoutput=1

\RequirePackage{amsthm,amsmath,amsfonts,amssymb}
\RequirePackage[authoryear]{natbib}
\RequirePackage[colorlinks,citecolor=blue,urlcolor=blue,backref=page,backref=page]{hyperref}
\RequirePackage{graphicx}

\RequirePackage{enumitem}
\RequirePackage{bm}
\RequirePackage{tabularx}
\RequirePackage{array}
\RequirePackage{verbatim}

\pubyear{2024}
\arxiv{2010.00000}
\volume{TBA}
\issue{TBA}
\firstpage{1}
\lastpage{1}

\startlocaldefs
\numberwithin{equation}{section}
\theoremstyle{plain}

\newtheorem{theorem}{Theorem}[section]
\newtheorem{lemma}[theorem]{Lemma}

\theoremstyle{definition}
\newtheorem{definition}[theorem]{Definition}

\theoremstyle{remark}

\newcommand{\argmax}{\operatornamewithlimits{argmax}}
\newcommand{\argmin}{\operatornamewithlimits{argmin}}
\newcommand{\defeq}{\stackrel{\text{def}}{=}}
\newcolumntype{?}{!{\vrule width 1.5pt}}
\endlocaldefs

\begin{document}

\begin{frontmatter}
\title{An Axiomatisation of Error Intolerant Estimation}
\runtitle{Error Intolerant Estimation}

\begin{aug}
\author[A]{\fnms{Michael}~\snm{Brand}\ead[label=e1]{michael.brand@rmit.edu.au}\orcid{0000-0001-9447-4933}}
\address[A]{School of Computing Technologies,
RMIT University\printead[presep={,\ }]{e1}}
\runauthor{M. Brand}
\end{aug}

\begin{abstract}
Point estimation is a fundamental statistical task. Given the wide selection of
available point estimators, it is unclear, however, what, if any, would be
universally-agreed theoretical reasons to generally prefer one such estimator
over another. In this paper, we define a class of estimation scenarios which
includes commonly-encountered problem situations such as both ``high stakes’’
estimation and scientific inference, and introduce a new class of estimators,
Error Intolerance Candidates (EIC) estimators, which we prove is optimal for it.

EIC estimators are parameterised by an externally-given loss function.
We prove, however, that even without such a loss function if one accepts a
small number of incontrovertible-seeming assumptions regarding what constitutes
a reasonable loss function, the optimal EIC estimator can be characterised
uniquely.

The optimal estimator derived in this second case is a previously-studied
combination of maximum a posteriori (MAP) estimation and
Wallace-Freeman (WF) estimation which has long been
advocated among Minimum Message Length (MML) researchers, where it is derived
as an approximation to the
information-theoretic Strict MML estimator. Our results provide a novel
justification for it that is purely Bayesian and requires neither
approximations nor coding, placing both MAP and WF as special cases in
the larger class of EIC estimators.
\end{abstract}



\end{frontmatter}

\section{Introduction}\label{S:introduction}

One of the fundamental problems in statistics is point estimation:
even when the distribution of the observed data given any potential hypothesis
is fully known, one typically still needs to select one hypothesis as one's
ultimate ``decision'' given the data.

The reason for the need to select a single hypothesis is often for simplicity
of presentation or for pragmatic reasons such as to streamline downstream
decision-making, but it can also have deeper justifications.
For example, scientific inference explicitly aims to decide on a
single theory that best fits the known set of experimental results
\citep{gauch2012scientific}.

While much of statistical learning theory deals with optimisation algorithms
that select the learnt hypothesis, there is no universal agreement on the
question of what these algorithms should optimise for.

R.A.~Fisher famously advocated for likelihood maximisation
\citep{Aldrich:Fisher},
\[
\hat{\theta}_{\text{MLE}}(x) \defeq \argmax_{\theta\in\Theta} f(x|\theta),
\]
but other optimisation criteria have since been
in wide use, such as the wider class of extremum estimators
\citep{Hayashi:econometrics}. Even among users of
maximum likelihood estimation (MLE), most, today, will
use Penalised MLE (PMLE) \citep{cole2014maximum},
\[
\hat{\theta}_{\text{PMLE}(g)}(x) \defeq \argmax_{\theta\in\Theta} [f(x|\theta) g(\theta)],
\]
in order to provide regularisation. However, different researchers will choose
distinct penalty functions, $g$, leading to algorithms such as
LASSO regression \citep{tibshirani1996regression} when using $L_1$
regularisation, ridge (Tikhonov) regression \citep{hoerl1970ridge} when using
$L_2$ regularisation, or elastic net \citep{zou2005regularization} when using
both. These are examples of regularisation
techniques that have additional parameters that require tuning. Other
methods, such as the
Akaike information criterion (AIC)
\citep{akaike1998information, bozdogan1987model} and
the Bayesian information criterion (BIC) \citep{schwarz1978estimating},
do not require further
parameters, but do require a level of semantic understanding of the
estimation problem in order to determine, e.g., the number of observations
and the number of parameters estimated by the model, which are not
necessarily straightforward to determine in the general case.

Moreover, the methods cited above are frequentist methods. In a Bayesian
setting, one is provided with the additional information of the prior
distribution over the hypothesis space, opening the door to a much wider
array of possibilities. Maximum a posteriori (MAP) \citep{degroot2005optimal},
for example,
maximises the posterior probability (or probability density) of the chosen
hypothesis, whereas Bayes estimation \citep{berger2013statistical} minimises an
expected loss function.

Other methods are defined in ways that do not bring a goodness-of-fit function
to an extremum. Two such examples are
posterior expectation \citep{carlin2010bayes} and
minimum message length (MML) \citep{Wallace2005}. Notably, while neither is
defined via a goodness-of-fit extremum, posterior expectation
can at least be re-defined equivalently as a Bayes estimator over the
quadratic loss function. MML, by contrast,
optimises an information criterion that cannot be formulated as an
extremum estimator at all, even in a Bayesian sense.

When faced with this wide array of potential methods to choose a single
point estimate given exactly the same data, exactly the same hypotheses and
exactly the same priors, a natural question to ask is whether there is any
disciplined, systematic method that one can use in order to determine a
single ``best'' point estimate.

For one commonly-encountered type of estimation
scenarios, which we refer to as \emph{error intolerant estimation},
we show that such a method exists and develop it in this paper.
For this, we use the Axiomatic Method \citep{thomson2001axiomatic}, a
powerful mathematical technique that has proved successful in determining a
``best'' solution in many problems that otherwise appear to require
qualitative judgement. (For examples, see \citet{vonneumann1947theory,
nash1950bargaining,shapley1953value}.)

The essence of the Axiomatic Method is that one states an explicit set of
assumptions that appear to be incontrovertible, in the sense that each
appears to be an essential minimal criterion that any solution that purports
to be at all ``good'' must satisfy. One then proves that there exists only a
single, unique solution that satisfies all desiderata, making it, therefore,
the only possible ``best'' solution.

Axiomatisation has already proved to be a powerful tool in providing an
underlying theory for multiple aspects of Bayesian analysis
\citep{dupre2009new, majumdar2004axiomatic}, and specifically for understanding
inference \citep{myung2005bayesian, geiger1991axioms, wolfowitz1962bayesian,
marin2022axioms}. Here, we use it to create a unified framework, applicable to
both discrete (inference) and continuous (estimation) problems.

\subsection{Notation, assumptions and basic definitions}

\subsubsection{Estimation problems}

In a Bayesian setting, an \emph{estimation problem} is
a pair of random variables $(\boldsymbol{x},\boldsymbol{\theta})$
with a known joint distribution that has some $X \times \Theta$ as its
support. I.e., it assigns positive probability,
$\mathbf{P}(x,\theta)=\mathbf{P}(\boldsymbol{x}=x,\boldsymbol{\theta}=\theta)$,
or otherwise positive probability density,
$f(x,\theta)=f^{(\boldsymbol{x},\boldsymbol{\theta})}(x,\theta)$,
to any $(x,\theta)\in X\times\Theta$ (and analogously also for other cases,
such as where $\boldsymbol{\theta}$ is continuously distributed but
$\boldsymbol{x}$ only takes values in a discrete set).
In addition to sharing the common support $X$, the
conditional data distributions
$P_\theta=P_{\boldsymbol{x}|\boldsymbol{\theta}=\theta}$ (sometimes
referred to simply as ``data distributions'') in an estimation
problem are assumed to all be distinct and absolutely continuous with respect
to each other.
We refer to $x$ as the \emph{observation}, $X\subseteq\mathbb{R}^N$ as
\emph{observation space}, $\theta$ as the \emph{parameter} and
$\Theta\subseteq \mathbb{R}^M$ as \emph{parameter space}.

Throughout, we use the following notation. Boldface symbols
(``$\boldsymbol{x}$'') indicate random variables. Regular symbols (``$x$'')
indicate variable instantiations.
For continuous $\boldsymbol{x}$,
$f_\theta=f^{(\boldsymbol{x},\boldsymbol{\theta})}_\theta$ is the probability
density function (pdf) of the data distribution $P_\theta$ (i.e.,
$f^{(\boldsymbol{x},\boldsymbol{\theta})}_\theta(x)\defeq f^{(\boldsymbol{x},\boldsymbol{\theta})}(x|\theta)$).
$\Lambda=\Lambda^{(\boldsymbol{x},\boldsymbol{\theta})}$ is the prior. (For
discrete $\boldsymbol{\theta}$, $\Lambda(\theta)$ conveys
$\mathbf{P}(\boldsymbol{\theta}=\theta)$, for continuous --- $f(\theta)$.)

Where available, we use pdfs interchangeably with
the distributions they represent (e.g., ``$\boldsymbol{x}\sim f_{\theta}$''
is equivalent to ``$\boldsymbol{x}\sim P_{\theta}$'').

In this paper, we analyse three classes of estimation problems:
\begin{description}
\item[Discrete:] For every $x\in X$ there exists a $\theta\in\Theta$, such that $\mathbf{P}(\boldsymbol{\theta}=\theta|\boldsymbol{x}=x)>0$.
\item[Continuous:] Both $\boldsymbol{x}$ and $\boldsymbol{\theta}$ are continuous random variables, and $f(x,\theta)$ is well-defined.
\item[Semi-continuous:] The random variable $\boldsymbol{\theta}$ is continuous with a well-defined $f(\theta|x)$ for every $x\in X$,
but the random variable $\boldsymbol{x}$ is discrete.
\end{description}
We refer to the latter two, collectively, as
$\boldsymbol{\theta}$-continuous estimation problems.

A problem's \emph{type} will be a description of
\begin{enumerate}
\item The problem's class (discrete, continuous or semi-continuous), and
\item The problem's dimensions, $M$ and $N$.
\end{enumerate}

For discrete problems we require no further restriction, but for
$\boldsymbol{\theta}$-continuous estimation problems our analysis,
for reasons of mathematical convenience and simplicity of presentation,
will be restricted to only problems satisfying the following criteria:
\begin{enumerate}
\item For all $x$, $f(\theta|x)$ is three-times continuously differentiable in $\theta$, and
\item For continuous problems, $f(x,\theta)$ is piecewise continuous in $x$.
\end{enumerate}
We refer to such problems as \emph{elementary} estimation problems.

We note that by ``differentiable'', we mean that the derivative is finite and
well-defined as a standard derivative in an open neighbourhood of every
$\theta\in\Theta$.

\subsubsection{Point estimators}

A \emph{point estimator} is a function $\hat{\theta}:X\to\mathbb{R}^M$.
For example:
\begin{itemize}
\item $\hat{\theta}^L_{\text{Bayes}}(x)\defeq\argmin_{\theta\in\Theta} \mathbf{E}[L(\boldsymbol{\theta},\theta)|\boldsymbol{x}=x]$
is the \emph{Bayes estimator} defined over the \emph{loss function}
$L=L_{(\boldsymbol{x},\boldsymbol{\theta})}:\Theta\times\Theta\to\mathbb{R}^{\ge 0}$.
\item $\hat{\theta}_{\text{c-MAP}}(x)\defeq\argmax_{\theta\in\Theta}f^{(\boldsymbol{x},\boldsymbol{\theta})}(\theta|x)$
and
$\hat{\theta}_{\text{d-MAP}}(x)\defeq\argmax_{\theta\in\Theta}\mathbf{P}(\theta|x)$
are the continuous and discrete maximum a posteriori (MAP) estimators,
respectively. For our purposes, these are treated as entirely separate
estimators.
\end{itemize}

Note that all standard point estimators are defined, like the ones
above, by means of an $\argmin$ or an $\argmax$. Such functions intrinsically
allow estimation results to be a subset of $\mathbb{R}^M$, rather than an
element of $\mathbb{R}^M$. Thus, they are technically \emph{set estimators}
rather than point estimators.

Though our error intolerant estimators will be explicitly
defined as point estimators rather than as set estimators, much of our
analysis involves estimators that, in the general case, may sometimes return
more than a single $\theta$ value as their estimates. In those cases, the use of
point-estimator notation should be considered solely a notational convenience.
Our theorems will be worded most generally, however, to account also for
set estimation.

In order to express the practical usability of a set estimator, we define the
following.
We say that an estimator is a \emph{well-defined point estimator} for
$(\boldsymbol{x},\boldsymbol{\theta})$ if it returns a single-element set for
every $x\in X$, in which case we take this element to be its estimate.
A general set estimator is considered
\emph{well-defined} on $(\boldsymbol{x},\boldsymbol{\theta})$ if it
returns a finite-sized, nonempty set as its estimate for every $x\in X$.

\subsubsection{Loss functions}

A \emph{loss function} is a function
$L=L_{(\boldsymbol{x},\boldsymbol{\theta})}:\Theta\times\Theta\to\mathbb{R}^{\ge 0}$,
where $L(\theta_1,\theta_2)$ represents the cost of choosing $\theta_2$ when
the true value of $\boldsymbol{\theta}$ is $\theta_1$.
We assume for all loss functions
\begin{equation}\label{Eq:distinct}
L(\theta_1,\theta_2)=0 \Leftrightarrow \theta_1=\theta_2 \Leftrightarrow P_{\theta_1}=P_{\theta_2}.
\end{equation}
Bayes estimators may also be defined, equivalently, over a
\emph{utility function}, $G$, where $G$ is the result of
a monotone strictly decreasing affine transform on a loss
function.\footnote{Note that loss functions and utility functions are always
indexed by the estimation problem, even when we omit
this subscript as shorthand. A single ``loss
function'' represents a separate function for each estimation problem.
We do not require every loss function to necessarily be applicable
to every possible estimation problem. Rather, when discussing point estimators
defined over a specific loss function, we narrow our discussion to a particular
class of estimation problems. The loss function is then expected to be
defined over that particular class.}

We say that a loss function $L$ is \emph{discriminative} for an estimation problem
$(\boldsymbol{x},\boldsymbol{\theta})$ if for every $\theta\in\Theta$ and
every open neighbourhood $\Phi$ of $\theta$ such that
$\Theta\setminus \Phi\ne\emptyset$,
$\inf_{\theta'\in\Theta\setminus \Phi} L_{(\boldsymbol{x},\boldsymbol{\theta})}(\theta',\theta)>0$.

We say a loss function $L$ is \emph{smooth} if for every elementary
$\boldsymbol{\theta}$-continuous estimation problem,
$(\boldsymbol{x},\boldsymbol{\theta})$,
the function
$L_{(\boldsymbol{x},\boldsymbol{\theta})}(\theta_1,\theta_2)$ is three times
continuously differentiable in $\theta_1$.

We say a loss function $L$ is \emph{sensitive} on a particular type of
$\theta$-continuous estimation problems if there is at least one
elementary estimation problem
$(\boldsymbol{x},\boldsymbol{\theta})$ of that type
and at least one choice of $\theta_0$, $i$ and $j$ such that
\[
\left.\frac{\partial^2 L_{(\boldsymbol{x},\boldsymbol{\theta})}(\theta,\theta_0)}{\partial\theta(i)\partial\theta(j)}\right\rvert_{\theta=\theta_0}\ne 0.
\]

Lastly, a loss function $L$ is said to be \emph{problem-con\-tin\-u\-ous}
(or ``$\mathcal{M}$-continuous'') if for every sequence of elementary continuous
estimation problems
$\left((\boldsymbol{x}_i,\boldsymbol{\theta})\right)_{i\in\mathbb{N}}$
and elementary continuous estimation problem $(\boldsymbol{x},\boldsymbol{\theta})$
such that at every $\theta\in\Theta$ the sequence satisfies
$\left(f^{(\boldsymbol{x}_i,\boldsymbol{\theta})}_\theta\right)\xrightarrow{\mathcal{M}} f^{(\boldsymbol{x},\boldsymbol{\theta})}_{\theta}$,
it is true that for every $\theta_1,\theta_2\in\Theta$,
\[
\lim_{i\to\infty} L_{(\boldsymbol{x}_i,\boldsymbol{\theta})}(\theta_1,\theta_2)=L_{(\boldsymbol{x},\boldsymbol{\theta})}(\theta_1,\theta_2).
\]

The symbol ``$\xrightarrow{\mathcal{M}}$'', used above, indicates
convergence in measure~\citep{halmos2013measure}. This is defined as follows.
Let $\mathcal{M}$ be the space of normalisable, non-atomic measures over some
$\mathbb{R}^s$, let $f$ be a function $f:\mathbb{R}^s\to\mathbb{R}^{\ge 0}$ and
let $\left(f_i\right)_{i\in\mathbb{N}}$ be a sequence of such
functions. Then $\left(f_i\right)\xrightarrow{\mathcal{M}} f$ if
\begin{equation}\label{Eq:measure0}
\forall\epsilon>0,\lim_{i\to\infty}\mu(\{x\in\mathbb{R}^s:|f(x)-f_i(x)|\ge \epsilon\})=0,
\end{equation}
where $\mu$ can be, equivalently, any measure in $\mathcal{M}$ whose support
is at least the union of the support of $f$ and all $f_i$ (recalling that
all measures in $\mathcal{M}$ are normalisable and non-atomic).

We will usually take $f$ and all $f_i$ to be pdfs.
When this is the case, $\mu$'s support only needs to equal the support
of $f$. Furthermore,
because $\mu$ is normalisable, one can always choose a value $t$
such that $\mu(\{x:f(x)>t\})$ is
arbitrarily small, for which reason one can substitute the absolute
difference ``$\ge\epsilon$'' in \eqref{Eq:measure0} with a relative difference
``$\ge\epsilon f(x)$'', and reformulate it in the case that $f$ and all
$f_i$ are pdfs as
\begin{equation}\label{Eq:measure_ratio}
\forall\epsilon>0,\lim_{i\to\infty}\mathbf{P}_{\boldsymbol{x}\sim f}(|f(\boldsymbol{x})-f_i(\boldsymbol{x})|\ge \epsilon f(\boldsymbol{x}))=0.
\end{equation}

This reformulation makes it clear that convergence in measure over pdfs
is a condition independent of the representation of the random variable.
I.e., it is invariant to transformations of the kind defining the IRO axiom in
Section~\ref{S:axioms}.

Not all commonly used loss functions satisfy discriminativity, smoothness,
sensitivity and problem continuity, for which
reason these properties should not be taken as general ``well-behavedness''
criteria. Nevertheless, in specific claims we
allow ourselves to restrict our choice of loss function somewhat in favour of
ensuring the function's good behaviour.
These are assumptions made for mathematical convenience.

\subsection{The plan, and a detailed result overview}

This paper is divided into two parts.

In the first part (Section~\ref{P:estimation}) we define the error intolerant
scenario and introduce ``Axiom 0'' of error intolerant estimation, which we
refer to as AIA. This part of
the paper combines a definitional approach with an axiomatic approach.
The basic notion of error intolerance, which we describe using terminology
from econometric theory, is a definition, which we introduce in
Section~\ref{SS:background}. It may be objected to, e.g. on
the grounds that it is a Bayesian notion, or that error intolerance is defined
relative to a loss function. However, once one accepts this basic definition of
what we mean by an error intolerant situation, we deem AIA, introduced in
Section~\ref{SS:AIA}, to be an
incontrovertible requirement for any reasonable estimator to be applied
in such a situation.

In Section~\ref{S:EIC}, we introduce a new estimator, which we refer to as
the Error Intolerance Candidates (EIC) estimator, that is defined relative to
one's choice of loss function. We prove
that for both discrete and $\boldsymbol{\theta}$-continuous estimation
problems, AIA is enough to uniquely characterise EIC.
On discrete problems, EIC coincides with discrete MAP, but on
$\boldsymbol{\theta}$-continuous ones it is a novel estimator and depends on
one's choice of loss function.

In Section~\ref{S:PMLE}, we show that EIC, on its own, is equivalent to PMLE.

By itself, this is quite a weak characterisation, but EIC has one advantage
over PMLE, namely that the EIC estimate
is completely determined by one's choice of loss function. This is discussed
in Section~\ref{SS:discussion}.

In the second part of the paper (Section~\ref{P:loss}), which is entirely
axiomatic, we explore
what would be a reasonable loss function to use in an error intolerant scenario.
We begin this in Section~\ref{S:axioms} by
introducing $4$ axioms that describe incontrovertible-seeming minimal criteria
for such loss functions.
Two of these are arguably incontrovertible also in more general estimation
scenarios.

Notably, our axioms are not regarding the estimator itself, but rather
directly regarding the underlying loss functions.
The ability to do this is in itself an assumption of our construction.
We refer to it as \emph{the loss principle}, and discuss its own justifications
and incontrovertibility in Section~\ref{S:lossprinciple}.

In Section~\ref{S:reasonable}, we prove regarding our first three axioms on loss
that together they uniquely characterise a single EIC estimator, no longer
parameterised by a loss function, in the continuous case, and that all
four axioms on loss together uniquely characterise this estimator also
in the semi-continuous case. This EIC estimator is equivalent to
the Wallace-Freeman estimator (WF)~\citep{WallaceFreeman1987}, defined as
\[
\hat{\theta}_\text{WF}(x)\defeq\argmax_{\theta} \frac{f(\theta|x)}{\sqrt{|\mathcal{I}_\theta|}},
\]
where $\mathcal{I}_\theta$ is the Fisher information matrix~\citep{lehmann2006theory},
whose $(i,j)$ element is the conditional expectation
\[
\mathcal{I}_\theta(i,j)\defeq\mathbf{E}\left[\left(\frac{\partial \log f(\boldsymbol{x}|\theta)}{\partial \theta(i)}\right)\left(\frac{\partial \log f(\boldsymbol{x}|\theta)}{\partial \theta(j)}\right)\middle|\boldsymbol{\theta}=\theta\right].
\]

We conclude the paper in Section~\ref{S:feasibility},
where we follow the conventions of the axiomatic approach, and
prove for all our axioms necessity and feasibility.

Necessity is the property that no subset of the axioms suffices for our
results of unique characterisation. We show this by providing counterexamples
for each potentially-omit\-ted axiom. In particular, we show that the
loss principle, despite being merely a guiding principle and not a mathematical
axiom,
is fundamental to our results. We do this by proving that removing it opens the
door even to non-Bayesian estimators, demonstrating by this not only the
necessity of the loss principle to our derivation, but its encapsulation of
certain aspects of Bayesian thinking.

Feasibility is the property that the set of loss functions meeting our
requirements is non-empty. Interestingly, while our characterisation of the
optimal error intolerant estimator is unique, we show that the underlying
choice of a loss function is not. In fact, many
standard loss functions, such as all $f$-metrics satisfying certain mild
criteria, meet our requirements. One interpretation of our results is therefore
that while many different loss functions, including standard ones, are reasonable
choices in an error intolerant scenario, all ultimately lead to the same
estimator: discrete MAP for discrete problems, WF otherwise.

A short conclusion section follows.

Regarding empirical support for our error intolerant estimator, we note that
the particular combined usage we describe of discrete MAP and WF has been
advocated by MML
researchers since at least \citet{comley200511}, with recent examples of
its empirical success being
\citet{bregu2021online, bregu2021mixture,hlavavckova2020heterogeneous,bourouis2021color,sumanaweera2018bits,schmidt2016minimum,saikrishna2016statistical,jin2015two,kasarapu2014representing}.
We do not, in this paper, attempt to add more such examples. See, however, the
Supplementary Information (SI), Section~B.1 for an analysis of the similarities
and the fundamental differences between our error intolerant estimator and that
advocated in MML, as well as of the differences in their
justifications.

In brief, whereas MML views this
estimator as a computationally convenient approximation to the Strict MML
point estimator, with no independent justification of its own, and whereas the
Strict MML estimator itself is justified through information theoretical means,
our derivation is purely Bayesian, requires neither approximations
nor coding, and places the combined estimator as a point in the larger
continuum of all EIC estimators.

\section{Rational estimation given error intolerance}\label{P:estimation}

\subsection{Background}\label{SS:background}

In econometrics \citep{pratt1978risk}, an agent's utility from an outcome can
be described either in objective terms
(e.g., how many dollars the agent has gained) or subjectively
(e.g., how happy they are with their winnings).
The agent's \emph{risk attitude function} is the mapping from objective to
subjective utility,
and the agent is considered \emph{rational} if their choices
maximise their expected subjective utility. This definition for rationality
has been previously axiomatised in many different ways. For examples, see
\citet{werner200573} and references therein. For our purposes, we will take
it as given.

Despite the fact that risk attitude functions are not directly observable,
they are a critical tool in explaining why rational actors prefer
lotteries with a lower win variance over ones with a higher win variance in
situations where both lotteries have
the same win expectation, and are even willing to trade off win expectation for
a decrease in win variance.\footnote{In econometrics, any situation that has a
stochastic outcome is referred to as a lottery.}

A person willing to perform such trade offs
is described as \emph{risk averse}, and the rationality of their
behaviour is explained by positing that their risk attitude function,
unobservable though it is, is concave. Conversely, a person with a convex
risk attitude function is described as \emph{risk loving}: they will be willing
to pay a premium in order to participate in a lottery with a high win variance.

Switching to the language of loss functions, $L$, rather than utilities,
and taking the risk attitude function, $T$, from now on to be a mapping
converting objective losses to subjective losses,
we say that a rational estimation procedure should select for each observed
value $x$ the estimate $\hat{\theta}(x)$ that minimises the
expected subjective loss
$\mathbf{E}[T(L(\boldsymbol{\theta}, \hat{\theta}(x)))|x]$.
Formally, we require of such risk attitude functions the following.

\begin{definition}
A function, $T:\mathbb{R}^{\ge 0}\to\mathbb{R}$, will be
called a \emph{risk attitude function} if it is
continuously differentiable and monotone weakly increasing.
\end{definition}

Our aim is to characterise the optimal behaviour in the face of an
\emph{error intolerant scenario}. This we define as a situation in
which there is only one right answer, and all mistakes are considered equally
bad (subjectively). Within the confines of risk-aversion theory,
we can approximate such a risk attitude by saying that all
losses above some $t_0$ are equally bad. Formally, we define as follows.

\begin{definition}
A risk attitude function $T$ is said to have an \emph{error limit} of
$t_0$ if there exists a $V_{\text{max}}>T(0)$, such that for all $t\ge t_0$,
$T(t)=V_{\text{max}}$.

A risk attitude function's minimal error limit will be called its
\emph{error tolerance}.
\end{definition}

Informally, we can now say that
error \emph{intolerance}, the object of our study, is the scenario
characterised by a vanishingly low error tolerance.

To be sure,
much of the use of statistical estimation is not in an error intolerant context.
A machine-learning algorithm running hundreds of thousands of times to make
hundreds of thousands of essentially-interchangeable choices, such as, for
example, in bidding for online ads, should optimally choose the strategy that
optimises its on-average success, and hence the minimisation of expected loss
used within the standard Bayes-estimation framework would be ideal for it.

Such an approach explicitly chooses its estimate, $\theta$, so as to be an
optimal trade-off between how close it is to other candidate values $\theta'$
and how likely those other $\theta'$ values are to be correct. In other words,
a $\theta$ value can be chosen not because $f(\theta|x)$ is high, but rather
because there are sufficient other $\theta'$ for which $f(\theta'|x)$ is
high, and $L(\theta',\theta)$ is relatively small. We refer to such estimation
methods as \emph{trade-off based methods}.

Error intolerant estimation, where we deem all errors equally bad, aims
explicitly to avoid such trade-offs. We want each $\theta$ to be judged solely
on its own merits, because the aim is to commit to a single hypothesis.
Scientific inference is an example of such a scenario. Another example is
high-stakes ``bet the ranch''-type
situations where the costs of any wrong decision are prohibitively high.

For a concrete example, consider a machine learning algorithm whose purpose
is to recommend
a medication dosage to a hospital patient. A Bayes estimator would recommend
a dosage that would average over the recommended dosages for each of the
patient's possible diagnoses, with averaging weights determined by diagnosis
probability and the severity of alterations over the correct dosage given any
such possible diagnosis. By contrast, an error intolerant algorithm would choose
its dosage recommendation, and, indeed,
its recommendation for the entire treatment regime, based on an underlying
choice regarding what specific condition the patient is diagnosed as having,
not on a probability cloud of such diagnoses.

Such an algorithm would provide consistent recommendations regarding the
patient's care, and, because of its underlying diagnosis, will score better in
ease to interrogate (related to explainability), making it likely to be the
more trusted choice among medical staff and patients alike.

For more intuition regarding error intolerant estimation and how it explicitly
avoids trade-offs, see the discussion in the SI, Section~B.2.

\subsection{Axiom 0}\label{SS:AIA}

The purpose of this paper is to define estimators that are suitable for
error intolerant scenarios, i.e. for scenarios where
one's error tolerance is vanishingly low. Though it is initially unclear how
to choose these well, we present minimal criteria that avoid those decisions
that are clearly undesirable.

Consider a spectrum of risk attitudes that one might take, each mapped
according to its error limit (noting that for the purpose of investigating
the error intolerant scenario we can safely ignore any risk attitude that
does not have an error limit).

\begin{definition}
A set of risk attitude functions
$\mathcal{T}=\{T_\epsilon\}_{\epsilon\in\mathbb{R}^{> 0}}$ where
each index value $\epsilon$ is an error limit of its respective $T_\epsilon$
will be called a \emph{risk attitude spectrum}.
\end{definition}

For a given loss function $L$ and a given risk attitude spectrum
$\mathcal{T}$, here are some choices that should clearly be avoided. We avoid
them because they present no advantage: they are inferior at \emph{every}
low error tolerance.

\begin{definition}
An estimate $\theta_1$ is said to be \emph{error-intolerance inferior}
(EI-inferior, as well as ``EI-inferior relative to $\theta_2$'') given an
observation $x$, a loss function $L$ and a risk attitude spectrum $\mathcal{T}$,
if there exists a $\theta_2$ and an $\epsilon_0>0$,
such that for every $\epsilon\in(0,\epsilon_0]$,
\begin{equation}\label{Eq:inferior}
\mathbf{E}[T_\epsilon(L(\boldsymbol{\theta}, \theta_1))|x] > \mathbf{E}[T_\epsilon(L(\boldsymbol{\theta}, \theta_2))|x].
\end{equation}
\end{definition}

With this in mind, we formulate our first axiom, which is a minimal
admissibility criterion for error intolerant estimators.

\begin{description}[style=unboxed]
\item[\textbf{Axiom 0: Avoidance of Inferior Alternatives (AIA)}]\ \newline
An error intolerant estimator over a loss function $L$ should not choose
estimates that are always inferior, regardless of one's risk attitude spectrum.
Formally:\newline
\textit{An estimator $\hat{\theta}$ over a loss function $L$ is said to satisfy
AIA if for every $x\in X$ there exists at least one risk attitude spectrum
$\mathcal{T}$ such that $\hat{\theta}(x)$ is not EI-inferior at $x$ over
$L$ and $\mathcal{T}$.}
\end{description}

\subsection{The EIC estimator}\label{S:EIC}

In this section, we show that for a given loss function $L$, AIA uniquely
characterises the following estimator.

\begin{definition}
The \emph{Error Intolerance Candidates} (EIC) estimator is defined by
\begin{equation}\label{Eq:EIC}
\hat{\theta}^L_{\text{EIC}}(x)\defeq\begin{cases}
\argmax_{\theta\in\Theta}\mathbf{P}(\theta|x) & \text{if $\max_{\theta\in\Theta}\mathbf{P}(\theta|x)>0$} \\
\argmax_{\theta\in\Theta}\frac{f(\theta|x)}{\sqrt{|H_L^{\theta}|}} & \text{otherwise},
\end{cases}
\end{equation}
where $H_L^{\theta}$ is the Hessian matrix of $L$, whose $(i,j)$ element is
defined as
\[
H_L^{\theta}(i,j)\defeq\left.\frac{\partial^2 L(\theta', \theta)}{\partial \theta'(i) \partial\theta'(j)}\right\vert_{\theta'=\theta}.
\]
\end{definition}

This estimator equals discrete MAP where discrete MAP is well-defined,
irrespective of $L$,
but is a novel estimator, which does depend on the choice of $L$, for
$\boldsymbol{\theta}$-continuous problems.

The fact that error intolerance implies discrete MAP for problems where
discrete MAP is well-defined is highly intuitive: if the
purpose is to maximise the probability of the estimate being ``right'', and
if some possibilities have a positive probability of being right, it makes
sense that an error intolerant person would choose the option whose
probability of being right is greatest.

We prove this characterisation in Theorem~\ref{T:MAP}.

\begin{theorem}\label{T:MAP}
Let $(\boldsymbol{x},\boldsymbol{\theta})$ be a discrete estimation problem,
and let $\hat{\theta}_L$ be a point estimator satisfying AIA with respect to
a loss function $L$ discriminative for $(\boldsymbol{x},\boldsymbol{\theta})$.
At every $x\in X$,
\[
\hat{\theta}_L(x) \in \hat{\theta}^L_{\text{EIC}}(x) = \hat{\theta}_{\text{d-MAP}}(x).
\]
\end{theorem}

We remind the reader that even though we use the terminology of point
estimation, estimators defined by means of an
$\argmin$ or an $\argmax$ are most generally set estimators. This is
why Theorem~\ref{T:MAP}, and, indeed, all our claims, are worded so as to
admit sets as estimates, despite this requiring a slightly more complex
wording to the claims. By contrast, $\hat{\theta}_L$ is specifically
defined to be a point estimator.

\begin{proof}
For a choice of a risk attitude spectrum, $\mathcal{T}=\{T_\epsilon\}$, define
$\mathcal{A}=\{A_\epsilon = S_\epsilon \circ T_\epsilon\}$ (where
$S \circ T$ denotes function composition) such that each
$S_\epsilon$ is an affine function mapping $T_\epsilon(0)$ to $1$ and
$T_\epsilon(\epsilon)$, the
maximum of $T_\epsilon$, to $0$. Minimising the expected loss
$\mathbf{E}[T_\epsilon(L(\boldsymbol{\theta},\theta))|\boldsymbol{x}=x]$
for choosing $\hat{\theta}_L(x)=\theta$ at a given
$\epsilon$ and a given $x$ is equivalent to maximising the expected utility
$\mathbf{E}[A_\epsilon(L(\boldsymbol{\theta},\theta))|\boldsymbol{x}=x]$.

Fix $x$, and let $V_\epsilon(\theta)$ be this utility:
\[
V_\epsilon(\theta) = \mathbf{E}[A_\epsilon(L(\boldsymbol{\theta},\theta))|\boldsymbol{x}=x].
\]

Let $\Phi_\epsilon(\theta)$ be the set of $\theta'$ for which
$A_\epsilon(L(\theta',\theta))$ is positive.
The value of the expected utility is bounded from both sides by
\begin{equation}\label{Eq:Vbounds}
\mathbf{P}(\theta|x)\le V_\epsilon(\theta)\le \mathbf{P}(\boldsymbol{\theta}\in \Phi_\epsilon(\theta)|x).
\end{equation}
This is due to $A_\epsilon(0)=1$ and the range of $A_\epsilon$ being bounded in
$[0,1]$.

Because, by discriminativity of $L$, for any neighbourhood $\Phi$ of $\theta$
there is an $\epsilon_0$ value such that for every $\epsilon\in(0,\epsilon_0]$,
$\Phi_\epsilon(\theta)\subseteq \Phi$, both bounds
converge to $\mathbf{P}(\theta|x)$. So, this is the limit for
$V_\epsilon(\theta)$.

To prove the claim, suppose to the contrary that
$\hat{\theta}_L(x) \notin \hat{\theta}_{\text{d-MAP}}(x)$.
Choose $\theta^{*} \in \hat{\theta}_{\text{d-MAP}}(x)$, noting that
$\hat{\theta}_{\text{d-MAP}}(x)$ is never the empty set.

By construction,
\[
\lim_{\epsilon\to 0} V_\epsilon(\hat{\theta}_L(x)) - V_\epsilon(\theta^{*})
=\lim_{\epsilon\to 0} V_\epsilon(\hat{\theta}_L(x)) - \lim_{\epsilon\to 0} V_\epsilon(\theta^{*}) = \mathbf{P}(\hat{\theta}_L(x)|x) - \mathbf{P}(\theta^{*}|x) < 0,
\]
so for a small enough $\epsilon$,
$V_\epsilon(\hat{\theta}_L(x)) < V_\epsilon(\theta^{*})$, and
$\hat{\theta}_L(x)$ is therefore EI-inferior, contradicting AIA.
\end{proof}

The claim for $\boldsymbol{\theta}$-continuous problems is similar, with
minor additional caveats.

\begin{theorem}\label{T:tcontinuous}
Let $(\boldsymbol{x},\boldsymbol{\theta})$ be an elementary
$\boldsymbol{\theta}$-continuous estimation problem over an open set $\Theta$,
let $L$ be a smooth loss function discriminative for it,
such that $\hat{\theta}^L_{\text{EIC}}$ is a well-defined set estimator
on $(\boldsymbol{x},\boldsymbol{\theta})$,
and let $\hat{\theta}_L$ be a point estimator satisfying AIA with respect to
$L$. For all $x\in X$,
\[
\hat{\theta}_L(x)\in \hat{\theta}^L_{\text{EIC}}(x).
\]
\end{theorem}

For reasons of brevity, formal, rigorous proofs to Theorem~\ref{T:tcontinuous}
and most other claims in this paper are given in the SI, Section A, rather
than in the main paper. We present here, instead, only general proof ideas.

\begin{proof}[Proof idea]
For any $\theta$, one can use a Taylor approximation in order to find a
neighbourhood around $\theta$ where $L(\theta',\theta)$ is in the range
defined by
\begin{equation}\label{Eq:Taylor}
L(\theta',\theta)=\frac{1}{2}(\theta'-\theta)^T H_L^{\theta} (\theta'-\theta)\pm \frac{m}{6} |\theta'-\theta|^3,
\end{equation}
for some constant error term $m$.

Choose $\mathcal{A}=\{A_\epsilon\}$ as in the proof of Theorem~\ref{T:MAP}.

In computing the expected utility
$\mathbf{E}[A_\epsilon(L(\boldsymbol{\theta},\theta))|\boldsymbol{x}=x]$
for $\epsilon$ tending to zero, if one chooses a small enough
$\delta$ value such that $B(\theta,\delta)$, the ball of radius $\delta$ around
$\theta$, is entirely within the neighbourhood where \eqref{Eq:Taylor}
holds, and if one further chooses an $\epsilon$ value small enough that
outside this neighbourhood both $L(\theta',\theta)$ and
$\frac{1}{2}(\theta'-\theta)^T H_L^{\theta} (\theta'-\theta)$ are greater
than $\epsilon$ (which is always possible to do because $L$ is discriminative
for the problem), then
\[
\begin{split}
\mathbf{E}[A_\epsilon(L(\boldsymbol{\theta}, \theta))|x]
& = \int_\Theta f(\theta'|x) A_\epsilon(L(\theta',\theta)) \text{d}\theta'
= \int_{B(\theta,\delta)} f(\theta'|x) A_\epsilon(L(\theta',\theta)) \text{d}\theta' \\
& \approx \int_{B(\theta,\delta)} f(\theta'|x) A_\epsilon\left(\frac{1}{2}(\theta'-\theta)^T H_L^{\theta} (\theta'-\theta)\right) \text{d}\theta'.
\end{split}
\]
For a sufficiently small $\delta$ (and a correspondingly small $\epsilon$)
this can then be further approximated as
\[
\begin{split}
f(\theta|x)& \int_{B(\theta,\delta)} A_\epsilon\left(\frac{1}{2}(\theta'-\theta)^T H_L^{\theta} (\theta'-\theta)\right) \text{d}\theta' \\
&=f(\theta|x) \int_{\mathbb{R}^M} A_\epsilon\left(\frac{1}{2}(\theta'-\theta)^T H_L^{\theta} (\theta'-\theta)\right) \text{d}\theta',
\end{split}
\]
and this, in turn, can be computed by a Jacobian transformation,
ultimately yielding the following approximation:
\begin{equation}\label{Eq:approx}
\mathbf{E}[A_\epsilon(L(\boldsymbol{\theta}, \theta))|x]
\approx \frac{f(\theta|x)}{\sqrt{|H^\theta_L|}}\int_{\mathbb{R}^M} A_\epsilon\left(\frac{1}{2}|\omega|^2\right)\text{d}\omega.
\end{equation}

The integral in \eqref{Eq:approx} is a multiplicative factor
independent of $\theta$, so gets dropped when computing the ratio
\[
\frac{\mathbf{E}[A_\epsilon(L(\boldsymbol{\theta}, \theta_1))|x]}{\mathbf{E}[A_\epsilon(L(\boldsymbol{\theta}, \theta_2))|x]} \approx
\frac{f(\theta_1|x)/\sqrt{|H^{\theta_1}_L|}}{f(\theta_2|x)/\sqrt{|H^{\theta_2}_L|}},
\]
for which the approximation becomes equality in the limit as $\epsilon$
goes to zero. If, contrary to the claim,
$\theta_1=\hat{\theta}_L(x)\notin\hat{\theta}^L_{\text{EIC}}$, we can choose
$\theta_2\in\hat{\theta}^L_{\text{EIC}}$, for which this limit is smaller than
$1$. This proves that $\theta_1$ is EI-inferior, and therefore that
$\hat{\theta}_L$ does not satisfy AIA, contradicting our assumption.
\end{proof}

\subsection{EIC and PMLE}\label{S:PMLE}

We have shown that AIA implies EIC, but EIC depends on one's choice of a
loss function. In this section, we show that if one is free to choose any
loss function then, for $\boldsymbol{\theta}$-continuous problems,
the span of EIC estimators equals the span of PMLE estimators.

\begin{theorem}\label{T:PMLE}
Let $(\boldsymbol{x},\boldsymbol{\theta})$ be a
$\boldsymbol{\theta}$-continuous estimation problem.

For every loss function $L$
such that $\hat{\theta}^L_{\text{EIC}}$ is well-defined on
$(\boldsymbol{x},\boldsymbol{\theta})$,
there is a function
$g:\Theta\to\mathbb{R}^{>0}$, for which $\hat{\theta}_{\text{PMLE}(g)}$ is
well-defined on $(\boldsymbol{x},\boldsymbol{\theta})$ and
\begin{equation}\label{Eq:PMLEEIC}
\hat{\theta}_{\text{PMLE}(g)} = \hat{\theta}^L_{\text{EIC}}.
\end{equation}

Conversely, for every function
$g:\Theta\to\mathbb{R}^{>0}$ such that $\hat{\theta}_{\text{PMLE}(g)}$ is
well-defined on $(\boldsymbol{x},\boldsymbol{\theta})$
there is a smooth loss function $L$
that is discriminative for $(\boldsymbol{x},\boldsymbol{\theta})$
such that $\hat{\theta}^L_{\text{EIC}}$ is well-defined on
$(\boldsymbol{x},\boldsymbol{\theta})$ and
$\hat{\theta}_{\text{PMLE}(g)} = \hat{\theta}^L_{\text{EIC}}$.
\end{theorem}

\begin{proof}
From \eqref{Eq:EIC}, we know that given $L$ we can set
\[
g(\theta) = \frac{f(\theta)}{\sqrt{\left|H^{\theta}_L\right|}}
\]
to satisfy \eqref{Eq:PMLEEIC}. By definition of an estimation problem
we know that $f(\theta)$ is always positive,
and because EIC is well-defined, we know that $\left|H^{\theta}_L\right|$ is
always defined and positive. Altogether, these imply that $g$ is also
always positive, as required.

Given $g$, on the other hand, we can choose
\begin{equation}\label{Eq:Lg}
L(\theta_1, \theta_2) = \left(\frac{f(\theta_2)}{g(\theta_2)}\right)^{2/M} L_2(\theta_1, \theta_2),
\end{equation}
where $L_2$ is the quadratic loss function,
\begin{equation}\label{Eq:qloss}
L_2(\theta_1,\theta_2)\defeq|\theta_1-\theta_2|^2,
\end{equation}
and $M$ is (as always) the dimension of $\Theta$.

This works because only the $L_2$ portion depends on $\theta_1$. The rest
acts as a constant in the calculation of $H^\theta_L$. As a result:
\[
H^\theta_L = \left(\frac{f(\theta)}{g(\theta)}\right)^{2/M} \times H^\theta_{L_2} =  \left(\frac{f(\theta)}{g(\theta)}\right)^{2/M} \times 2I_{M}
\]
and
\[
\left|H^\theta_L\right| = \left(\frac{f(\theta)}{g(\theta)}\right)^2 \times 2^M \times \left|I_{M}\right| = \left(\frac{f(\theta)}{g(\theta)}\right)^2 \times 2^M,
\]
where $I_M$ is the $M\times M$ identity matrix.

Thus,
\[
\argmax_{\theta\in\Theta}\frac{f(\theta|x)}{\sqrt{|H^\theta_L|}}
=\argmax_{\theta\in\Theta}\frac{f(\theta|x)}{f(\theta)}g(\theta)
=\argmax_{\theta\in\Theta}f(x|\theta)g(\theta).
\]

This proves that
$\hat{\theta}_{\text{PMLE}(g)} = \hat{\theta}^L_{\text{EIC}}$,
and because we know that $\hat{\theta}_{\text{PMLE}(g)}$ is well-defined, so
is EIC.

Because $L(\theta_1,\theta_2)$ is for every $\theta_2$ a constant positive
multiple of $L_2(\theta_1,\theta_2)$, we also know that it is smooth and
discriminative for every problem.
\end{proof}

\subsection{Discussion}\label{SS:discussion}

Before concluding the first part of this paper, let us discuss why the
results so far, characterising EIC as the class of estimators satisfying AIA
and essentially equating between the range of EIC estimators
and the range of PMLE estimators, are important conclusions.
This is for four reasons.

\begin{enumerate}
\item Without specifying a loss function, we see that EIC can essentially be
any member of the class of PMLE estimators. This class is very broad,
making EIC without a specified loss function a very weak characterisation.
This fact is important, when using the axiomatic method, in order to demonstrate
that our single AIA axiom is not an overreach: it does not overly constrain
the list of potential solutions.
\item Conversely, it is equally important to note that AIA whittled away
exactly those estimators that we wished to discard. If we compare EIC/PMLE
against all estimators listed in Section~\ref{S:introduction}, we see that,
with the single exception of the general class of extremum estimators,
all estimators mentioned are either special cases of PMLE, or are
trade-off based methods, which is exactly the class of methods that AIA was
specifically designed to carve out because they are unsuitable for the
error intolerant scenario.
\item EIC is not only defined on $\boldsymbol{\theta}$-continuous problems.
It is defined equally on all types of estimation problems, and by this
constructs a bridge between the different types of problems.
If we choose $L_2$ as our loss function, for
example, the EIC estimator over this loss will be discrete MAP in the
discrete case and continuous MAP in the $\boldsymbol{\theta}$-continuous
case (as can be verified by substituting $L=L_2$ into \eqref{Eq:EIC}).
Thus, EIC provides a formal justification
for the use of continuous MAP as a continuous analogue of discrete MAP.
This connection was previously investigated by \citet{bassett2018maximum}
in the context of the limiting behaviour of Bayes estimators, but our work
places this relationship between discrete MAP and
continuous MAP inside a much larger context, where the continuous analogue of
discrete MAP may be $\hat{\theta}_{\text{PMLE}}(g)$ for essentially any
function $g$ simply by use of an appropriate loss function. The choice of
``equating'' discrete MAP specifically with continuous MAP is therefore akin
to accepting $L_2$ loss (or an equivalent) as one's loss function of choice.
\item Generalising the previous point: there is one feature that
differentiates the class of
$\hat{\theta}^L_{\text{EIC}}$ estimators introduced here from the well-known
class of $\hat{\theta}_{\text{PMLE}}(g)$ estimators: whereas there is no
intrinsic reason to choose one $g$ over another in PMLE (and, indeed, we see
many researchers using many different $g$ functions), each EIC estimator is
indexed by a specific loss function, $L$, and this enables us to tackle the
question of what estimator is most suitable for use in a given scenario
(such as in the case of the error intolerant scenario) by reformulating this
question as ``Which loss functions would be appropriate for use in this
situation?''. It is this question that will be explored in
Section~\ref{P:loss}.
\end{enumerate}

\section{Choosing a reasonable loss function}\label{P:loss}

\subsection{Axioms on loss}\label{S:axioms}

In standard econometrics risk-aversion theory, the objective loss (or utility)
is part of the problem definition. In such a case where $L$ is known,
Section~\ref{P:estimation} proves that $\hat{\theta}^L_{\text{EIC}}$ is the
best choice of estimator for the error intolerant scenario, if one
accepts the AIA axiom.

In estimation, however, $L$ is typically chosen by the statistician, and is
not simply given. In our case, this allows us to choose judiciously which
of the many EIC estimators to select, by transforming this choice into one
about deciding which loss function would be most appropriate to use in our
scenario of interest.
We now present the axioms that we claim to be incontrovertible for any
``reasonable'' loss function to be used in the error intolerant problem
scenario.

In all axioms, our requirement is that the loss function
$L$ (parameterised by its estimation problem as
$L_{(\boldsymbol{x},\boldsymbol{\theta})}$) satisfies the specified conditions
for every elementary estimation problem
$(\boldsymbol{x},\boldsymbol{\theta})$ and
every pair of parameters $\theta_1$ and $\theta_2$ in parameter space.

As always, we take the parameter space to be $\Theta\subseteq\mathbb{R}^M$ and
the observation space to be $X\subseteq\mathbb{R}^N$.

We present the axioms here in their formal formulation. Readers looking for a
deeper analysis of the rationale and the claim of incontrovertibility of
each axiom can find this discussion in the SI, Section~B.3. We claim that
two of the axioms (IRP and IIA) are incontrovertible in error intolerant scenarios,
while the remaining axioms (IRO and ISI) are incontrovertible even in the
general estimation scenario, beyond error intolerant estimation.

\begin{description}[style=unboxed]
\item[\textbf{Axiom 1: Invariance to Representation of Parameter Space (IRP)}]\ \newline
Loss (and the subsequent estimate) should not depend on how estimates are
presented (e.g., in metric or imperial measurements). Formally:\newline
\textit{A loss function $L$ is said to satisfy IRP if for every diffeomorphism
$F:\mathbb{R}^M \to \mathbb{R}^M$ for which
$(\boldsymbol{x},F(\boldsymbol{\theta}))$ is an elementary estimation problem,
\[
L_{(\boldsymbol{x},\boldsymbol{\theta})}(\theta_1,\theta_2)
=L_{(\boldsymbol{x},F(\boldsymbol{\theta}))}(F(\theta_1),F(\theta_2)).
\]}
\end{description}

\begin{description}[style=unboxed]
\item[\textbf{Axiom 2: Invariance to Representation of Observation Space (IRO)}]\ \newline
Loss (and the subsequent estimate) should not depend on how measurements are
presented (e.g., in polar or Cartesian coordinates). Formally:\newline
\textit{A loss function $L$ is said to satisfy IRO if
for every invertible, piecewise-diffeo\-morphic function
$G:\mathbb{R}^N \to \mathbb{R}^N$ for which
$(G(\boldsymbol{x}),\boldsymbol{\theta})$ is an elementary estimation problem,
\[
L_{(\boldsymbol{x},\boldsymbol{\theta})}(\theta_1,\theta_2)
=L_{(G(\boldsymbol{x}),\boldsymbol{\theta})}(\theta_1,\theta_2).
\]}
\end{description}

\begin{description}[style=unboxed]
\item[\textbf{Axiom 3: Invariance to Irrelevant Alternatives (IIA)}]\ \newline
Loss from $\theta_1$ to $\theta_2$ should not depend on details of the problem
at alternatives other than $\theta_1$ and $\theta_2$ (and consequently, the
estimate should not base its preference between these two on such
irrelevancies). Formally:\newline
\textit{A loss function $L$ is said to satisfy IIA if for every
elementary estimation problem $(\boldsymbol{x},\boldsymbol{\theta'})$,
where  $\boldsymbol{\theta'}$ is distributed over
$\Theta'\subseteq\mathbb{R}^M$, and every
$\theta_1, \theta_2 \in \Theta\cap\Theta'$, if at
$\theta\in\{\theta_1,\theta_2\}$
$\Lambda^{(\boldsymbol{x},\boldsymbol{\theta})}(\theta)=\Lambda^{(\boldsymbol{x},\boldsymbol{\theta'})}(\theta)$ and
$P^{(\boldsymbol{x},\boldsymbol{\theta})}_\theta=P^{(\boldsymbol{x},\boldsymbol{\theta'})}_\theta$, then
\[
L_{(\boldsymbol{x},\boldsymbol{\theta})}(\theta_1,\theta_2)=
L_{(\boldsymbol{x},\boldsymbol{\theta'})}(\theta_1,\theta_2).
\]}
\end{description}

\begin{description}[style=unboxed]
\item[\textbf{Axiom 4: Invariance to Superfluous Information (ISI)}]\ \newline
The addition of independent noise should not influence loss (nor the
subsequent estimate).
Formally:\newline
\textit{A loss function $L$ is said to satisfy ISI if for any random variable
$\boldsymbol{y}$ such that $\boldsymbol{y}$ is independent of
$\boldsymbol{\theta}$ given $\boldsymbol{x}$,
\[
L_{((\boldsymbol{x},\boldsymbol{y}),\boldsymbol{\theta})}(\theta_1,\theta_2)
=L_{(\boldsymbol{x},\boldsymbol{\theta})}(\theta_1,\theta_2),
\]
assuming $((\boldsymbol{x},\boldsymbol{y}),\boldsymbol{\theta})$ is an
elementary estimation problem.}
\end{description}

Table~\ref{Table:axioms} provides a graphical representation of the four
axioms.

\begin{table}[h]
\begin{tabular}{?l||c|c?}
\noalign{\hrule height 1.5pt}
& Invariance to Representation & Invariance to Irrelevancies \\
\hline \hline
In Parameter Space & IRP & IIA \\
In Observation Space & IRO & ISI \\
\noalign{\hrule height 1.5pt}
\end{tabular}%
\caption{The axioms, represented graphically. The table demonstrates the
symmetry of the construction, with two axioms requiring invariance to
representation (one in parameter space, one in observation space) and two
requiring invariance to irrelevancies (one in parameter space, and one in
observation space).}
\label{Table:axioms}
\end{table}

We remark that the choice of these particular four axioms, both those
universally applicable and those specific to the error intolerant scenario, is
not arbitrary. It is heavily implied (and, technically, constrained) by the
structure of the EIC estimator (which, in turn, is characterised entirely
by AIA, and therefore by our understanding of error intolerance).

To explain: on discrete problems, EIC coincides with discrete MAP, which is an
estimator that is well known to exhibit exactly these types of invariance as
we are now requiring of the loss functions: it is agnostic both to
representation and to irrelevancies. EIC is an estimator that extends MAP into
the continuous domain, and ideally we would have liked it to exhibit MAP's
good properties also in its extensions. All four properties listed above
are ones that, if they are exhibited by EIC's underlying loss function, $L$,
then the corresponding estimator-level invariances will also be manifested
by $\hat{\theta}^L_{\text{EIC}}$:
\begin{itemize}
\item Regarding IRO and ISI, $\hat{\theta}^L_{\text{EIC}}(x)$ can only depend
on the representation of $x$ or superfluous information therein if this
dependency comes from $L$.
\item Regarding IIA, if $L$ satisfies its definition above, then
$\hat{\theta}^L_{\text{EIC}}(x)$ satisfies IIA in the sense originally defined
for this term by \citet{nash1950bargaining}: if, for an estimation problem
$(\boldsymbol{x},\boldsymbol{\theta})$ over some domain $X\times\Theta$
we have $\hat{\theta}^L_{\text{EIC}}(x)=\theta$, IIA in the Nash sense
stipulates that for any $\Theta'\subseteq\Theta$ such that $\theta\in\Theta'$,
if $\boldsymbol{\theta'}$ is
$\boldsymbol{\theta}$ restricted to $\Theta'$ then for the estimation problem
$(\boldsymbol{x},\boldsymbol{\theta'})$ the equality
$\hat{\theta'}^L_{\text{EIC}}(x)=\theta$ should also be satisfied.\footnote{This
actually requires a slightly stronger IIA condition than what we have
assumed in Axiom 3, but the stronger axiom is not required for our derivation.
(The estimator property described is, however, nevertheless implied by the
combination of our IIA and IRP axioms.)}
\item Lastly, regarding IRP, EIC is invariant to parameter representation in
the sense that if for some estimation problem
$(\boldsymbol{x}, \boldsymbol{\theta})$
we define a new estimation problem $(\boldsymbol{x}, \boldsymbol{\phi})$
by choosing $\boldsymbol{\phi}=F(\boldsymbol{\theta})$ for some
diffeomorphism $F:\mathbb{R}^M\to\mathbb{R}^M$ such that for all $\theta_1$ and
$\theta_2$,
$L_{(\boldsymbol{x}, \boldsymbol{\theta})}(\theta_1,\theta_2)
= L_{(\boldsymbol{x}, \boldsymbol{\phi})}(F(\theta_1),F(\theta_2))$,
then EIC will satisfy for all $x$,
$\hat{\phi}^L_{\text{EIC}}(x) = F\left(\hat{\theta}^L_{\text{EIC}}(x)\right)$.
This is because
\[
\begin{split}
\hat{\phi}^L_{\text{EIC}}(x) &= \argmax_{\phi} \frac{f^{(\boldsymbol{x},\boldsymbol{\phi})}(\phi|x)}{\sqrt{\left|H^\phi_{L_{(\boldsymbol{x},\boldsymbol{\phi})}}\right|}}
=F\left(\argmax_{\theta} \frac{f^{(\boldsymbol{x},\boldsymbol{\phi})}(F(\theta)|x)}{\sqrt{\left|H^{F(\theta)}_{L_{(\boldsymbol{x},\boldsymbol{\phi})}}\right|}}\right) \\
&=F\left(\argmax_{\theta} \frac{f^{(\boldsymbol{x},\boldsymbol{\theta})}(\theta|x)/|J_F(\theta)|}{\sqrt{\left|H^\theta_{L_{(\boldsymbol{x},\boldsymbol{\theta})}}\right|}/|J_F(\theta)|}\right)
=F\left(\hat{\theta}^L_{\text{EIC}}(x)\right),
\end{split}
\]
where $J_F$ is the Jacobian of $F$.
\end{itemize}
Thus, the choice of axioms reflects good properties already present at the
core structure of EIC, which we wish the final estimator,
$\hat{\theta}^L_{\text{EIC}}$ to also manifest.

\subsection{The loss principle}\label{S:lossprinciple}

So far, we have seen multiple justifications for all five of our axioms
(AIA, IRP, IRO, IIA and ISI). Another assumption that needs justification,
however, albeit remaining implicit so far, is that it is legitimate for us
to word our axioms as axioms about the loss functions, rather than as axioms
about the estimators themselves, axioms that do not assume any underlying
structure to the estimators.
This is a departure from the standard approach of the axiomatic method,
but is necessary, as we demonstrate in Section~\ref{SS:ax0}.

In terms of this assumption's incontrovertibility,
we submit to the reader than if an estimator is reasonable, it should be
formulatable as an estimator over a reasonable loss function. We refer to this
assertion as \emph{the loss principle}. This primacy of the loss function
over the estimator is used here in the context of error intolerant estimators, but
it is an underlying principle in much of Bayesian thought: Bayesian researchers
often justify their choice of estimator by the reasonableness of the underlying
loss function \citep{williamson2022bayesianism} or, indeed, consider the
loss functions themselves as objects of study in their own right
\citep{hasan2013comparative}.

We deem the loss principle incontrovertible because in Bayesian thought it is
almost definitional: not only does the loss function completely determine the
estimator, it is the loss function that defines the estimator to begin with.
Challenging the loss principle is challenging this Bayesian approach, and,
indeed, what is demonstrated in Section~\ref{SS:ax0} is that omitting it,
with all else remaining equal, opens the door even to decidedly non-Bayesian
solutions.

\subsection{EIC estimation over reasonable loss}\label{S:reasonable}

Our main theorem for continuous problems is as follows.
It requires only 3 of our 4 axioms.

\begin{theorem}\label{T:main}
If $(\boldsymbol{x},\boldsymbol{\theta})$ is an elementary continuous
estimation problem
for which $\hat{\theta}_{\text{WF}}$ is a well-defined
set estimator, and if $L$ is a smooth and problem continuous
loss function, sensitive on $(\boldsymbol{x},\boldsymbol{\theta})$'s type,
that satisfies all of IIA, IRP and IRO, then
\[
\hat{\theta}^L_{\text{EIC}}=\hat{\theta}_{\text{WF}}.
\]
\end{theorem}

We prove the theorem through a progression of lemmas.
As before, $M$ and $N$ will always represent the dimensions of the
parameter space and the observation space, respectively.

\begin{lemma}\label{L:likelihood}
For $\boldsymbol{\theta}$-continuous estimation problems
$(\boldsymbol{x},\boldsymbol{\theta})$,
if $L$ satisfies both IIA and IRP then
$L_{(\boldsymbol{x},\boldsymbol{\theta})}(\theta_1,\theta_2)$ is a function
only of the data distributions $P_{\theta_1}$ and
$P_{\theta_2}$ and of $M$, the dimension of the parameter space.
\end{lemma}

\begin{proof}[Proof idea]
The IIA axiom is tantamount to stating that
$L_{(\boldsymbol{x},\boldsymbol{\theta})}(\theta_1,\theta_2)$ is dependent
only on the following:
\begin{enumerate}
\item The function's inputs $\theta_1$ and $\theta_2$,
\item The data distributions $P_{\theta_1}$ and $P_{\theta_2}$,
\item The priors $f(\theta_1)$ and $f(\theta_2)$, and
\item The problem's parameter space dimension $M$.
\end{enumerate}

We can assume without loss of generality that $\theta_1\ne\theta_2$, or else
the value of $L(\theta_1,\theta_2)$ can be determined to be zero by
\eqref{Eq:distinct}.

IRP excludes dependence on the priors as well as on the values of $\theta_1$
and $\theta_2$ themselves, because the parameter space can be distorted
diffeomorphically to change these to any desired values, without impacting the
loss.
\end{proof}

In light of Lemma~\ref{L:likelihood}, we will henceforth use the notation
$L(P_{\theta_1},P_{\theta_2})$ (or, when $\boldsymbol{x}$ is known to be
continuous, $L(f_{\theta_1},f_{\theta_2})$) instead of
$L_{(\boldsymbol{x},\boldsymbol{\theta})}(\theta_1,\theta_2)$.
A loss function that can be expressed in this way is referred to
as a \emph{conditional-distribution-based} loss function.

We now introduce the function $c$, which will form the backbone for much
of the rest of our construction.

For distributions $P$ and $Q$ with a common support $X\subseteq\mathbb{R}^N$,
absolutely continuous with respect to each other,
let $r_{P,Q}(x)$ be the Radon-Nikodym derivative
$\frac{\text{d}P}{\text{d}Q}(x)$~\citep{royden1988real}.
For a value of $x$ that has positive probability in both $P$ and $Q$, this is
simply $\mathbf{P}_{\boldsymbol{x}\sim P}(\boldsymbol{x}=x)/\mathbf{P}_{\boldsymbol{x}\sim Q}(\boldsymbol{x}=x)$,
whereas for pdfs $f$ and $g$ it is $f(x)/g(x)$ within the common support.

The function $c[P,Q]:[0,1]\to\mathbb{R}^{\ge 0}$ is defined by
\[
c[P,Q](t)\defeq\inf\left(\{r\in\mathbb{R}^{\ge 0}:t\le \mathbf{P}_{\boldsymbol{x}\sim Q}(r_{P,Q}(\boldsymbol{x})\le r)\}\right).
\]
Up to how it resolves certain edge cases, this function is simply
the inverse function to the cumulative distribution function
of the random variable $r_{P,Q}(\boldsymbol{x})$, where $\boldsymbol{x}$ is
distributed according to $Q$.
The next two lemmas demonstrate the usefulness of $c$ to our analysis.

\begin{lemma}\label{L:ccont}
The function $c$ is $\mathcal{M}$-continuous for continuous distributions,
in the sense that if both
$\left(f_i\right)_{i\in\mathbb{N}}\xrightarrow{\mathcal{M}} f$ and
$\left(g_i\right)_{i\in\mathbb{N}}\xrightarrow{\mathcal{M}} g$,
where $f$ and $g$ are pdfs and $\left(f_i\right)_{i\in\mathbb{N}}$ and
$\left(g_i\right)_{i\in\mathbb{N}}$ are pdf sequences,
then $\left(c[f_i,g_i]\right)_{i\in\mathbb{N}}\xrightarrow{\mathcal{M}} c[f,g]$.
\end{lemma}

\begin{proof}[Proof idea]
By definition of convergence in measure.
\end{proof}

\begin{lemma}\label{L:r}
If $L$ is a
problem continuous con\-di\-tion\-al-distribution-based loss function that
satisfies IRO, and
$p$ and $q$ are piecewise-continuous probability density functions
that are data distributions
in an estimation problem $(\boldsymbol{x}, \boldsymbol{\theta})$, then
$L_{(\boldsymbol{x},\boldsymbol{\theta})}(p,q)$ depends only on $c[p,q]$ and on
the type of $(\boldsymbol{x}, \boldsymbol{\theta})$.
\end{lemma}

\begin{proof}[Proof idea]
Fix the problem type.

Because $L$ is conditional-distribution-based, it is solely a function of $p$
and $q$ (in addition to the now fixed problem type).

Consider the case $N=1$, which is the case where demonstrating the idea of the
proof is easiest. We will show that given any $p$ and $q$, we can alter the
observation space using a transformation that preserves the loss so as to map
the data distributions $p$ and $q$ into a canonical form, dependent only on
$c[p,q]$.
This being the case, starting from any $p'$ and $q'$ (also data distributions
in an estimation problem of the same type) for which $c[p',q']=c[p,q]$
will reach the same canonical form, and must therefore have $L(p',q')=L(p,q)$,
proving the claim.

To do this, let us first transform the observation $x$ into
$\text{cdf}(x)\defeq \mathbf{P}_{\boldsymbol{x}\sim q}(\boldsymbol{x}\le x)$.
By IRO, if an
estimation problem has $p=f^{(\boldsymbol{x},\boldsymbol{\theta})}_{\theta_1}$
and $q=f^{(\boldsymbol{x},\boldsymbol{\theta})}_{\theta_2}$, then the
estimation problem $(\text{cdf}(\boldsymbol{x}),\boldsymbol{\theta})$ will
satisfy $L_{(\boldsymbol{x},\boldsymbol{\theta})}(\theta_1,\theta_2)
=L_{(\text{cdf}(\boldsymbol{x}),\boldsymbol{\theta})}(\theta_1,\theta_2)$.

However, this transform maps
$f^{(\text{cdf}(\boldsymbol{x}),\boldsymbol{\theta})}_{\theta_2}$ to the
uniform distribution on $[0,1]$ and maps
$f^{(\text{cdf}(\boldsymbol{x}),\boldsymbol{\theta})}_{\theta_1}(\text{cdf}(x))$
to $r_{p,q}(x)$.

We next take the new problem,
$(\text{cdf}(\boldsymbol{x}),\boldsymbol{\theta})$, and transform it again.
This time, we partition the range $[0,1]$ into segments of length $\delta$,
and then rearrange the parts so that they will be ordered by increasing
$f^{(\text{cdf}(\boldsymbol{x}),\boldsymbol{\theta})}_{\theta_1}$.
Again, by IRO, the loss cannot change.

We can do this for arbitrarily small $\delta$. In the limit,
the resulting $f^{(\boldsymbol{x'},\boldsymbol{\theta})}_{\theta_1}$ is
simply $c[p,q]$. The loss in the limit must also remain the same, namely due to
Lemma~\ref{L:ccont} and the problem-continuity assumption.

This brings us, as desired, to a canonical form with the same loss, which
is determined only by $c[p,q]$ and must therefore be shared also by any other
$p'$ and $q'$ for which $c[p,q]=c[p',q']$.

Similar can also be done for a general $N$, in which case in the canonical
form of the problem the domain of $\boldsymbol{x}$ is the $N$-dimensional
unit cube, and in all dimensions except the first all marginal distributions are
uniform.
\end{proof}

\begin{lemma}\label{L:Fisher}
Let $L$ be a smooth conditional-distribution-based
loss function satisfying that $L(P,Q)$ is a function only of $c[P,Q]$
and of the problem type, and let $(\boldsymbol{x},\boldsymbol{\theta})$
be an elementary $\boldsymbol{\theta}$-continuous estimation problem of a type
on which $L$ is sensitive.

If one of the following conditions holds true:
\begin{enumerate}
\item The problem $(\boldsymbol{x},\boldsymbol{\theta})$ is a
continuous estimation problem, or,
\item The problem $(\boldsymbol{x},\boldsymbol{\theta})$ is a semi-continuous
estimation problem and $L$ satisfies ISI,
\end{enumerate}
then there exists a nonzero constant $\gamma$, dependent only on
the choice of $L$ and the problem type,
such that for every $\theta\in\Theta$
the Hessian matrix $H_L^{\theta}$ equals
$\gamma$ times the Fisher information matrix $\mathcal{I}_{\theta}$.
\end{lemma}

\begin{proof}[Proof idea]
Fix the problem type.

By assumption, $L(P,Q)$ can be rewritten as
$\tilde{L}(r_{P,Q})$. To determine the derivatives of $L(P,Q)$, we can use
this compositional form. This assumption also poses multiple constraints on the
form of these derivatives (which we derive in the SI), from which we ultimately
conclude the existence of functions
$\Delta_{\text{x}}$, $\Delta_{\text{xy}}$ and $\Delta_{\text{xx}}$, dependent
only on the distribution of $r_{f_{\theta_1},f_{\theta_2}}(\boldsymbol{x})$
with $\boldsymbol{x}\sim f_{\theta_2}$, such that
the second derivative of the loss function can be written as
\begin{equation}\label{Eq:second_deriv}
\begin{aligned}
\frac{\partial^2  L(\theta_1,\theta_2)}{\partial \theta_1(i)\partial \theta_1(j)}=
&\int_{X} \Delta_{\text{x}}(r_{f_{\theta_1},f_{\theta_2}}(x)) \frac{\partial^2 r_{f_{\theta_1},f_{\theta_2}}(x)}{\partial \theta_1(i)\partial \theta_1(j)}f_{\theta_2}(x)\text{d}x \\
&\quad+\int_{X}\int_{X} \Delta_{\text{xy}}(r_{f_{\theta_1},f_{\theta_2}}(x_1),r_{f_{\theta_1},f_{\theta_2}}(x_2)) \\
&\quad\quad\quad\quad\quad\quad\frac{\partial r_{f_{\theta_1},f_{\theta_2}}(x_1)}{\partial \theta_1(i)}\frac{\partial r_{f_{\theta_1},f_{\theta_2}}(x_2)}{\partial \theta_1(j)}
f_{\theta_2}(x_2)\text{d}x_2 f_{\theta_2}(x_1)\text{d}x_1 \\
&\quad+\int_{X} \Delta_{\text{xx}}(r_{f_{\theta_1},f_{\theta_2}}(x))
\frac{\partial r_{f_{\theta_1},f_{\theta_2}}(x)}{\partial \theta_1(i)}\frac{\partial r_{f_{\theta_1},f_{\theta_2}}(x)}{\partial \theta_1(j)}f_{\theta_2}(x)\text{d}x.
\end{aligned}
\end{equation}

The Hessian matrix element $H_L^{\theta}(i,j)$ equals
${\partial^2  L(\theta_1,\theta_2)}/{\partial \theta_1(i)\partial \theta_1(j)}$
at $\theta_1=\theta_2=\theta$. When this is the case,
$r_{f_{\theta_1},f_{\theta_2}}$ equals $1$ everywhere in $X$.

The value of \eqref{Eq:second_deriv} in this case becomes
\begin{equation}\label{Eq:second_deriv_at_1}
\begin{aligned}
&\Delta_{\text{x}}\left(1\right)\int_{X} \left.\frac{\partial^2 r_{f_{\theta_1},f_{\theta}}(x)}{\partial \theta_1(i)\partial \theta_1(j)}\right\rvert_{\theta_1=\theta}f_{\theta}(x)\text{d}x \\
&\quad +\Delta_{\text{xy}}\left(1,1\right)\left(\int_{X} \left.\frac{\partial r_{f_{\theta_1},f_{\theta}}(x)}{\partial \theta_1(i)}\right\rvert_{\theta_1=\theta}f_{\theta}(x)\text{d} x\right)
\left(\int_{X} \left.\frac{\partial r_{f_{\theta_1},f_{\theta}}(x)}{\partial \theta_1(j)}\right\rvert_{\theta_1=\theta}f_{\theta}(x)\text{d} x\right)\\
&\quad +\Delta_{\text{xx}}\left(1\right)\int_{X} \left(\left.\frac{\partial r_{f_{\theta_1},f_{\theta}}(x)}{\partial \theta_1(i)}\right\rvert_{\theta_1=\theta}\right)
\left(\left.\frac{\partial r_{f_{\theta_1},f_{\theta}}(x)}{\partial \theta_1(j)}\right\rvert_{\theta_1=\theta}\right)f_{\theta}(x)\text{d}x.
\end{aligned}
\end{equation}

Note, however, that because $(\boldsymbol{x},\boldsymbol{\theta})$ is an
estimation problem, i.e.\ all its data distributions are probability measures,
not general measures, it is the case that
\[
\int_{X} r_{f_{\theta_1},f_{\theta}}(x)f_{\theta}(x)\text{d}x
=\int_{X} f_{\theta_1}(x)\text{d}x=1,
\]
and is therefore a constant
independent of either $\theta_1$ or $\theta$. Its various derivatives in
$\theta_1$ are accordingly all zero. This makes the first two summands in
\eqref{Eq:second_deriv_at_1} zero.

Moreover, because $r$ is, again, simply $1$ throughout
all of $X$, and does not depend on the estimation problem, the remaining
function $\Delta_{\text{xx}}$ also does not depend on the estimation problem,
and in particular $\Delta_{\text{xx}}(1)$ is simply a constant. We therefore
define $\gamma=\Delta_{\text{xx}}(1)$, and rewrite the equation as
\begin{align*}
H_L^{\theta}(i,j)
&=\gamma \int_{X} \left(\left.\frac{\partial r_{f_{\theta_1},f_{\theta}}(x)}{\partial \theta_1(i)}\right\rvert_{\theta_1=\theta}\right)
\left(\left.\frac{\partial r_{f_{\theta_1},f_{\theta}}(x)}{\partial \theta_1(j)}\right\rvert_{\theta_1=\theta}\right)f_{\theta}(x)\text{d}x\\
&=\gamma \int_X \left(\left.\frac{\partial f_{\theta_1}(x)/f_{\theta}(x)}{\partial \theta_1(i)}\right\rvert_{\theta_1=\theta}\right)
\left(\left.\frac{\partial f_{\theta_1}(x)/f_{\theta}(x)}{\partial \theta_1(j)}\right\rvert_{\theta_1=\theta}\right) f_{\theta}(x)\text{d}x\\
&=\gamma \int_X \left(\frac{\partial \log f_{\theta}(x)}{\partial \theta(i)}\right)\left(\frac{\partial \log f_{\theta}(x)}{\partial \theta(j)}\right) f_{\theta}(x)\text{d}x\\
&=\gamma \mathbf{E}_{\boldsymbol{x}\sim f_{\theta}}\left[\left(\frac{\partial \log f_{\theta}(\boldsymbol{x})}{\partial \theta(i)}\right)\left(\frac{\partial \log f_{\theta}(\boldsymbol{x})}{\partial \theta(j)}\right)\right]\\
&=\gamma \mathcal{I}_{\theta}(i,j).
\end{align*}

Hence, $H_L^{\theta}=\gamma \mathcal{I}_{\theta}$.

The only difference in this derivation for the case where $\boldsymbol{x}$ is
discrete, is that in this case the final result would have used probabilities
rather than
probability densities. This is consistent, however, with the way Fisher
information is defined in this more general case. In fact, some sources
(e.g.,~\citealp{bobkov2014fisher}) define the Fisher information directly from
Radon-Nikodym derivatives.

As a final point in the proof, we remark that $\gamma$
must be nonzero, because had it been zero, $H_L^{\theta}$
would have been zero for every $\theta$ in every elementary
$\boldsymbol{\theta}$-continuous estimation problem of the same type,
contrary to our sensitivity assumption on $L$.
\end{proof}

We now turn to prove our main claim.

\begin{proof}[Proof of Theorem~\ref{T:main}]
By Lemma~\ref{L:likelihood} we know $L$ to be a conditional-distribution-based
loss function, and by Lemma~\ref{L:r} we know its value at $L(P,Q)$ depends only
on the value of $c[P,Q]$ and the problem type. With these prerequisites, we can
use Lemma~\ref{L:Fisher} to conclude that up to a nonzero constant
multiple $\gamma$ its $H_L^{\theta}$ is the Fisher information matrix
$\mathcal{I}_{\theta}$. This, in turn, we've assumed to be positive definite by
requiring $\hat{\theta}_\text{WF}$ to be well-defined.

Furthermore, $\gamma$ cannot be negative, as in combination with a positive
definite Fisher information matrix this would indicate that $H_L^{\theta}$
is not positive semidefinite, causing $L$ to attain negative values in the
neighbourhood of $\theta$.

It is therefore the case that $H_L^{\theta}$ must be positive definite, and
therefore $\hat{\theta}^L_{\text{EIC}}$ is well-defined and equal to
$\hat{\theta}_{\text{WF}}$.
\end{proof}

The final case remaining is that of semi-continuous estimation problems.
Unlike in the continuous case, here we use all 4 of our axioms on loss.

\begin{theorem}\label{T:semicontinuous}
If $(\boldsymbol{x},\boldsymbol{\theta})$ is an elementary, semi-continuous
estimation problem
for which $\hat{\theta}_{\text{WF}}$ is a well-defined
set estimator, and if $L$ is a smooth loss function,
sensitive on $(\boldsymbol{x},\boldsymbol{\theta})$'s type,
that satisfies all of IIA, IRP, IRO and ISI, then
\[
\hat{\theta}^L_{\text{EIC}}=\hat{\theta}_{\text{WF}}.
\]
\end{theorem}

\begin{proof}[Proof idea]
The proof is essentially identical to that of Theorem~\ref{T:main}. The only
change is that we can no longer apply Lemma~\ref{L:r}. However, because of ISI
we can nevertheless show that
$L(P,Q)$ still only depends on $c[P,Q]$ and $M$, for which reason we can still
apply Lemma~\ref{L:Fisher}, as before, to complete the proof.
\end{proof}

\section{Feasibility and necessity}\label{S:feasibility}

Customarily, when using the axiomatic method, one proves two additional claims:
\begin{description}
\item[Feasibility:] It is possible to construct a loss function and an
estimator that meet all criteria, and
\item[Necessity:] No criterion is superfluous for our main uniqueness results.
\end{description}

We prove these here. Note that,
as above, some technical proofs have been moved to the SI, Section~A, and
only their basic ideas are provided here.

\subsection{AIA}

\begin{theorem}
The AIA axiom is always feasible, in the sense that both
for Theorem~\ref{T:MAP} and for Theorem~\ref{T:tcontinuous},
if $(\boldsymbol{x}, \boldsymbol{\theta})$ and $L$ satisfy the conditions of
the theorem, then there exists a point estimator $\hat{\theta}_L$ that
satisfies AIA on $(\boldsymbol{x}, \boldsymbol{\theta})$ with respect to $L$.
\end{theorem}

\begin{proof}
The AIA axiom requires an estimator not to choose estimates that are always
EI-inferior regardless of the choice of one's risk attitude spectrum,
$\mathcal{T}$. We will prove that
even given an arbitrary choice of $\mathcal{T}$, at every $x$ one can always
find at least one $\theta\in\Theta$ that is not EI-inferior.

For an estimation problem $(\boldsymbol{x},\boldsymbol{\theta})$
that matches the conditions of Theorem~\ref{T:MAP}, the theorem proves that
all other $\theta\in\Theta$ are EI-inferior to those in
$\hat{\theta}_{\text{d-MAP}}(x)$.
For an estimation problem that matches the conditions of
Theorem~\ref{T:tcontinuous}, the theorem proves that
all other $\theta\in\Theta$ are EI-inferior to those in
$\hat{\theta}^L_{\text{EIC}}(x)$.

In both cases, the theorem describes a set, $\Theta'$, of $\theta$ values not
EI-inferior to any $\theta'\in\Theta\setminus\Theta'$, so these $\theta$
values, if they are EI-inferior, can only be EI-inferior to each other.
Furthermore, in both cases $\Theta'$ is nonempty and finite.

In the case of Theorem~\ref{T:MAP}, it
is nonempty and finite by the assumption that there exists a $\theta\in\Theta$
with a positive posterior probability $p$ given $x$, from which we can conclude
that at least one and no more than $\lfloor 1/p \rfloor$ $\theta$ values
exist with a posterior probability of $p$ or more. In the
case of Theorem~\ref{T:tcontinuous}, $\Theta'$ is nonempty and finite because
$\hat{\theta}^L_{\text{EIC}}$ was assumed to be a well-defined set estimator.

Given any choice of $\mathcal{T}$ and $x$, EI-inferiority induces a
partial ordering
of $\Theta'$. However, any partial ordering among a finite, nonempty set of
elements admits at least one maximal element. Hence, one can always define an
estimator satisfying AIA by choosing a maximal element for each $x$.
\end{proof}

\begin{theorem}
All axioms used in Theorem~\ref{T:MAP} and Theorem~\ref{T:tcontinuous}
are necessary.
\end{theorem}

\begin{proof}
This is trivial, because only one axiom, AIA, was used in these theorems.
Omitting this axiom means that the choice of estimator is not constrained
at all.
\end{proof}

\subsection{Axioms on loss}

\begin{theorem}\label{T:feasibility}
Our system of axioms is feasible, in the sense that there exist
smooth and problem continuous
loss functions, $L$, sensitive on all $\boldsymbol{\theta}$-continuous
problem types, that satisfy all of IRP, IRO, IIA and ISI.
\end{theorem}

\begin{proof}[Proof idea]
We prove that any $f$-divergence \citep{ali1966general} meets the requirements,
if its $F$-function satisfies the following:
\begin{enumerate}
\item $F$ has 3 continuous derivatives,
\item $F''(1)>0$, and
\item $\lim_{x\to 0} F(x)<\infty$ and
$\lim_{x\to\infty} F'(x)<\infty$,
\end{enumerate}
where $F'$ and $F''$ are the first two derivatives of $F$,
which we assumed exist.

One example of such a loss function is
squared Hellinger distance~\citep{pollard2002user}.
\end{proof}

The last of the above criteria is only needed to ensure problem continuity,
which is a requirement in Theorem~\ref{T:main} but not in
Theorem~\ref{T:semicontinuous}.
An example of an $f$-divergence that meets all requirements except this last,
and is therefore suitable for semi-continuous problems but not necessarily for
continuous ones, is Kullback-Leibler divergence~\citep{kullback1997information}.

Notably, there are many other loss functions, not $f$-divergences, that also
meet all requirements. One example is Bhattacharyya distance~\citep{bhattacharyya1946measure}.
\[
D_\text{B}(P,Q)\defeq-\ln(1-H^2(P,Q)),
\]
where $H^2$ is the squared Hellinger distance.

It can easily be shown to satisfy all 4 axioms, as well as problem continuity,
because it is a function of the squared Hellinger distance. Also, the ratio
of the Bhattacharyya distance to the squared Hellinger distance
tends to $1$ for small Hellinger distances, so it also satisfies smoothness
and sensitivity. We leave it for follow-up research to provide a more complete
explicit characterisation of the possible loss functions satisfying our
desiderata.

\begin{theorem}\label{T:nec_disc}
All axioms used in Theorem~\ref{T:main} are necessary.
\end{theorem}

\begin{theorem}\label{T:nec_disc2}
All axioms used in Theorem~\ref{T:semicontinuous} are necessary.
\end{theorem}

\begin{proof}[Proof idea for Theorem~\ref{T:nec_disc} and
Theorem~\ref{T:nec_disc2}]
We demonstrate both claims by const\-ruct\-ing counterexamples for each case.

For example, IRP can be shown to be necessary by considering quadratic loss,
which leads to continuous MAP rather than to $\hat{\theta}_{\text{WF}}$. It
satisfies all of IRO, IIA and ISI, but not IRP.
\end{proof}

\subsection{The loss principle}\label{SS:ax0}

Now that we have shown that all of our formal axioms are necessary, we return to
the underlying loss principle, which allowed us, from the start, to describe
the axioms in terms of $L$ rather than in term of the estimator itself.

To demonstrate that this is also necessary, we need to be able to apply
IRP, IRO, IIA and ISI to estimators rather than to loss functions.
For this, we could have used the descriptions of IRP, IRO, IIA and ISI at the
estimator level as given at the end of Section~\ref{S:axioms}. However,
in order to define estimator level properties in a way that clearly does not
alter the meaning of the original axioms, we define the following.

\begin{definition}
Let P be a property,
$\text{P}\in\{\text{IRP}, \text{IRO}, \text{IIA}, \text{ISI}\}$. We say that
estimator $\hat{\theta}$ satisfies property P if it equals
$\hat{\theta}^L_{\text{EIC}}$ for a loss function $L$ that satisfies P.
\end{definition}

With this definition, we can now demonstrate that the loss principle
is not only necessary, but also that it encapsulates, in fact,
important elements of Bayesianism. We do this by providing an example of another
estimator, which is neither discrete MAP nor WF, and is not even a Bayesian
estimator at all, that also meets all four of our axioms on loss.

\begin{theorem}\label{T:axiom0}
MLE satisfies IRP, IRO, IIA and ISI.
\end{theorem}

\begin{proof}[Proof idea]
We demonstrate that MLE can be described as a $\hat{\theta}^L_{\text{EIC}}$
using two separate loss functions, one satisfying IRP and the other satisfying
IRO, IIA and ISI.
\end{proof}

We remark that MLE also satisfies all of IRP, IRO, IIA and ISI if taking the
definitions of these at the estimator level as per their descriptions in
Section~\ref{S:axioms}.

\section{Conclusions and future work}\label{S:conclusion}

In this paper, we have rigorously defined ``error intolerant estimation'', a
common
and important scenario for point estimation that merits its own specialised
solutions, and defined a new framework for creating Bayesian estimators,
akin to the framework of Bayes estimation, specifically tailored for it.

We defined an axiom, AIA, that describes a minimal criterion for any good
error intolerant estimator, and this sufficed to characterise a class of
estimators which we call Error Intolerance Candidates (EIC) estimators,
parameterised by their underlying loss function.

We defined four axioms that appear incontrovertible within the context of
choosing a reasonable loss function for error intolerant estimation (of which
two seem highly compelling beyond error intolerant estimation as well) and
have proved that the combination of discrete MAP (for estimation
problems where at least one hypothesis has positive probability) and
WF (otherwise) is the only error intolerant estimator to satisfy our axioms,
and that the axioms are all necessary for the uniqueness of this
characterisation, as well as being jointly feasible.

By this, we have provided a first theoretical justification
for the use of WF estimation that does not rely on approximations, and the
first justification at all for the use of the combination of WF together with
discrete MAP, even though the empirical evidence for the power of this
combination has been accumulating within the context of the MML literature for
the past three decades.

Our justification for WF is fully Bayesian,
does not incorporate information theory, and is unrelated to MML.
Furthermore, the class of EIC estimators now includes
both discrete MAP and WF as special cases.

The paper opens up any number of potential avenues for future research,
but perhaps
the most intriguing follow-up question is whether such an axiomatic framework
can be built without the loss principle, i.e.\ whether similarly
incontrovertible axioms can be introduced for error intolerant estimation which
will refer directly to the estimators investigated, rather than to their
underlying loss functions. We have shown that in our approach,
the loss principle is a necessary assumption.

\begin{acks}[Acknowledgements]
The author would like to extend his sincerest thanks to the anonymous reviewer
for their in-depth review of the paper and for their insightful remarks. These
greatly improved the design of this paper. I would also like to thank
Mark Steel, the editor-in-chief, for allowing the paper its much-needed
extra space.
\end{acks}

\begin{funding}
Portions of this paper were conceived while the author was with the
Faculty of IT of Monash University.
\end{funding}

\begin{supplement}
\stitle{Supplementary Information for ``An Axiomatisation of Error Intolerant Estimation''}
\sdescription{A.~Supplementary proofs: The full text of the proofs that were
too lengthy and too technical to be included in the main paper.
B.~Supplementary discussion: Supplementary discussion on the connection between error intolerance and trade-off base methods, on the rationale and putative
incontrovertibility of the axioms presented for error intolerant estimation,
and also
on the similarities and the fundamental differences between the error intolerant
estimators developed here and estimators advocated for in the MML literature,
as well as on how analysing the empirical successes of MML
estimators through the lens of error intolerant estimation differs from their
MML analysis.}
\end{supplement}

\bibliographystyle{ba}
\bibliography{bibwf}

\end{document}



\begin{frontmatter}
\title{Supplementary Information for ``An Axiomatisation of Error Intolerant Estimation''}
\runtitle{Error Intolerant Estimation (SI)}

\begin{aug}
\author[A]{\fnms{Michael}~\snm{Brand}\ead[label=e1]{michael.brand@rmit.edu.au}\orcid{0000-0001-9447-4933}}
\address[A]{School of Computing Technologies,
RMIT University\printead[presep={,\ }]{e1}}
\runauthor{M. Brand}
\end{aug}


\end{frontmatter}

\appendix

\section{Supplementary proofs}\label{S:proofs}

We provide in this section complete proofs for the claims of the main paper.

\subsection{The EIC estimator}

\begin{theorem}[Theorem 2.7 of the main text]\label{T:tcontinuous}
Let $(\boldsymbol{x},\boldsymbol{\theta})$ be an elementary
$\boldsymbol{\theta}$-con\-tin\-u\-ous estimation problem over an open set $\Theta$,
let $L$ be a smooth loss function discriminative for it,
such that $\hat{\theta}^L_{\text{EIC}}$ is a well-defined set estimator
on $(\boldsymbol{x},\boldsymbol{\theta})$,
and let $\hat{\theta}_L$ be a point estimator satisfying AIA with respect to
$L$. For all $x\in X$,
\[
\hat{\theta}_L(x)\in \hat{\theta}^L_{\text{EIC}}(x).
\]
\end{theorem}

\begin{proof}
For a choice of a risk attitude spectrum, $\mathcal{T}=\{T_\epsilon\}$, define
$\mathcal{A}=\{A_\epsilon = S_\epsilon \circ T_\epsilon\}$ such that each
$S_\epsilon$ is an affine function mapping $T_\epsilon(0)$ to $1$ and the
maximum of $T_\epsilon$ to $0$. Minimising the expected loss
$\mathbf{E}[T_\epsilon(L(\boldsymbol{\theta},\theta))|\boldsymbol{x}=x]$
for choosing $\hat{\theta}_L(x)=\theta$ at a given
$\epsilon$ and a given $x$ is equivalent to maximising the expected utility
$\mathbf{E}[A_\epsilon(L(\boldsymbol{\theta},\theta))|\boldsymbol{x}=x]$.

When computing this expected utility,
\begin{equation}\label{Eq:AkL}
\mathbf{E}[A_\epsilon(L(\boldsymbol{\theta}, \theta))|x] = \int_\Theta f(\theta'|x) A_\epsilon(L(\theta',\theta)) \text{d}\theta',
\end{equation}
for $\epsilon$ values tending to zero,
one only needs to consider the integral over $B(\theta,\delta)$, the
(closed set) ball of
radius $\delta$ around $\theta$, for any $\delta>0$, as for a sufficiently
small $\epsilon$, the rest of the integral values will be zero by the
discriminativity assumption. We will, in particular, assume that any chosen
$\delta$ is small enough so that $B(\theta,\delta)\subseteq\Theta$.

Because $L$ is smooth,
the second derivative (i.e., the Hessian) $H_L^{\theta}$ is defined and
finite everywhere, and its own derivative (i.e., the third derivative) is
continuous. Thus, this third derivative is bounded inside the closed set
$B(\theta,\delta)$. Let $m$ be its maximum absolute value in any direction.

Using a second-order Taylor approximation,
we can therefore write, for $\theta'\in B(\theta,\delta)$,
\begin{equation}\label{Eq:Taylor}
L(\theta',\theta)=\frac{1}{2}(\theta'-\theta)^T H_L^{\theta} (\theta'-\theta)\pm \frac{m}{6} |\theta'-\theta|^3.
\end{equation}

Let $H^\theta_\text{min} = H^\theta_L - \frac{m\delta}{3} I_M$ and
$H^\theta_\text{max} = H^\theta_L + \frac{m\delta}{3} I_M$, where $I_M$ is the
$M\times M$ identity matrix. For
all $\theta'$ values in $B(\theta, \delta)$,
\begin{equation}\label{Eq:Hbounds}
\frac{1}{2}(\theta'-\theta)^T H^{\theta}_\text{min} (\theta'-\theta) \le L(\theta',\theta) \le \frac{1}{2}(\theta'-\theta)^T H^{\theta}_\text{max} (\theta'-\theta).
\end{equation}

Because we also assumed that the posterior distribution $f(\theta|x)$ is
continuous, inside $B(\theta,\delta)$ it is also bounded from above and
below. Let its bounds be $f_{\text{max}}$ and $f_{\text{min}}$.

From \eqref{Eq:Hbounds}, we know that the value of \eqref{Eq:AkL}
can be bounded by
\begin{equation}\label{Eq:Ibounds}
\begin{split}
f_{\text{min}} & \int_{B(\theta,\delta)} A_\epsilon\left(\frac{1}{2}(\theta'-\theta)^T H^\theta_{\text{max}} (\theta'-\theta)\right)\text{d}\theta'
\le \int_\Theta f(\theta'|x) A_\epsilon(L(\theta',\theta)) \text{d}\theta' \\
& \le f_{\text{max}}\int_{B(\theta,\delta)} A_\epsilon\left(\frac{1}{2}(\theta'-\theta)^T H^\theta_{\text{min}} (\theta'-\theta)\right)\text{d}\theta'.
\end{split}
\end{equation}

Because we assumed that $\hat{\theta}^L_{\text{EIC}}$ is well-defined, we
know that $H^\theta_L$ is positive definite, so for a small enough $\delta$
$H^\theta_{\text{min}}$ and $H^\theta_{\text{max}}$ will be, too. Thus,
for $H$ being either $H^\theta_{\text{min}}$ or $H^\theta_{\text{max}}$
there is a positive infimum to the value of
$\frac{1}{2}(\theta'-\theta)^T H (\theta'-\theta)$ outside $B(\theta,\delta)$.

When $\epsilon$ is smaller than this infimum, the integration bounds in
\eqref{Eq:Ibounds} can be switched from $B(\theta,\delta)$ to $\mathbb{R}^M$,
without this impacting the truth of the inequality, because all $\theta'$
values outside $B(\theta,\delta)$ will contribute zero to the integral.

Values of the form
\[
f \int_{\mathbb{R}^M} A_\epsilon\left(\frac{1}{2}(\theta'-\theta)^T H (\theta'-\theta)\right)\text{d}\theta'
\]
can be computed via a Jacobian transformation as
\begin{equation}\label{Eq:prop}
\frac{f}{\sqrt{|H|}}\int_{\mathbb{R}^M} A_\epsilon\left(\frac{1}{2}|\omega|^2\right)\text{d}\omega,
\end{equation}
where the integral is a multiplicative factor independent of $\theta$.

As $\epsilon$, and therefore also $\delta$, tends to zero,
both $|H^\theta_\text{min}|$ and
$|H^\theta_\text{max}|$ tend to $|H^\theta_L|$, and both $f_{\text{min}}$ and
$f_{\text{max}}$ converge to $f(\theta|x)$.

Noting that $f(\theta|x)$ is positive by our assumption on estimation problems,
and that $|H^\theta_L|$ is positive by our assumption that
$\hat{\theta}^L_{\text{EIC}}$ is well-defined, we reach
\[
\lim_{\epsilon\to 0} \frac{\mathbf{E}[A_\epsilon(L(\boldsymbol{\theta}, \theta))|x]}{\int_{\mathbb{R}^M} A_\epsilon\left(\frac{1}{2}|\omega|^2\right)\text{d}\omega} = \frac{f(\theta|x)}{\sqrt{|H^\theta_L|}} > 0.
\]

Suppose, contrary to the claim, that
$\hat{\theta}_L(x)\notin \hat{\theta}^L_{\text{EIC}}(x)$. Because
$\hat{\theta}^L_{\text{EIC}}$ was assumed to be a well-defined set estimator,
we know that $\hat{\theta}^L_{\text{EIC}}(x)$ is not empty. Choose
$\theta^{*}\in\hat{\theta}^L_{\text{EIC}}(x)$.

By its definition, EIC maximises the metric
$f(\theta|x)/\sqrt{|H^\theta_L|}$, so
\[
\lim_{\epsilon\to 0} \frac{\mathbf{E}[A_\epsilon(L(\boldsymbol{\theta}, \hat{\theta}_L(x)))|x]}{\mathbf{E}[A_\epsilon(L(\boldsymbol{\theta}, \theta^{*}))|x]}
=\frac{f(\hat{\theta}_L(x)|x)/\sqrt{\left|H^{\hat{\theta}_L(x)}_L\right|}}{f(\theta^{*}|x)/\sqrt{\left|H^{\theta^{*}}_L\right|}} < 1.
\]
This proves that for every low enough $\epsilon$,
\[
\mathbf{E}[A_\epsilon(L(\boldsymbol{\theta}, \hat{\theta}_L(x)))|x] < \mathbf{E}[A_\epsilon(L(\boldsymbol{\theta}, \theta^{*}))|x],
\]
contradicting our assumption that $\hat{\theta}_L$ satisfies AIA.
\end{proof}

\subsection{EIC estimation over reasonable loss}

Our main theorem for continuous problems is as follows.

\begin{theorem}[Theorem~3.1 of the main text]\label{T:main}
If $(\boldsymbol{x},\boldsymbol{\theta})$ is an elementary continuous
estimation problem
for which $\hat{\theta}_{\text{WF}}$ is a well-defined
set estimator, and if $L$ is a smooth and problem continuous
loss function, sensitive on $(\boldsymbol{x},\boldsymbol{\theta})$'s type,
that satisfies all of IIA, IRP and IRO, then
\[
\hat{\theta}^L_{\text{EIC}}=\hat{\theta}_{\text{WF}}.
\]
\end{theorem}

In the main text, we prove the theorem through a progression of lemmas.
While the proof of the main theorem exists in the main text, we supplement it
here with the complete proofs of its underlying lemmas.

Recall that, throughout, $M$ and $N$ always represent the dimensions of the
parameter space and the observation space, respectively.

\begin{lemma}[Lemma~3.2 of the main text]\label{L:likelihood}
For $\boldsymbol{\theta}$-continuous estimation problems
$(\boldsymbol{x},\boldsymbol{\theta})$,
if $L$ satisfies both IIA and IRP then
$L_{(\boldsymbol{x},\boldsymbol{\theta})}(\theta_1,\theta_2)$ is a function
only of the data distributions $P_{\theta_1}$ and
$P_{\theta_2}$ and of $M$, the dimension of the parameter space.
\end{lemma}

\begin{proof}
The IIA axiom is tantamount to stating that
$L_{(\boldsymbol{x},\boldsymbol{\theta})}(\theta_1,\theta_2)$ is dependent
only on the following:
\begin{enumerate}
\item The function's inputs $\theta_1$ and $\theta_2$,
\item The data distributions $P_{\theta_1}$ and $P_{\theta_2}$,
\item The priors $f(\theta_1)$ and $f(\theta_2)$, and
\item The problem's parameter space dimension $M$ (noting that all other
elements of the problem type are already fixed by specifying
the data distributions).
\end{enumerate}

We can assume without loss of generality that $\theta_1\ne\theta_2$, or else
the value of $L(\theta_1,\theta_2)$ can be determined to be zero by
(1.1).

Our first claim is that, due to IRP, $L$ can also not depend on the problem's
prior probability densities $f(\theta_1)$ and $f(\theta_2)$. To show this,
construct a diffeomorphism
$F:\mathbb{R}^M \to \mathbb{R}^M$
in the following way.

Let $(\vec{v}_1,\ldots,\vec{v}_M)$ be an orthogonal basis
for $\mathbb{R}^M$ wherein $\vec{v}_1=\theta_2-\theta_1$.
We design $F$ as
\[
F\left(\theta_1+\sum_{i=1}^M b_i\vec{v}_i\right)=\theta_1+F_1(b_1)\vec{v}_1+\sum_{i=2}^M b_i\vec{v}_i,
\]
where $F_1:\mathbb{R}\to\mathbb{R}$ is a continuous, differentiable function
onto $\mathbb{R}$, with a derivative that is positive everywhere, satisfying
\begin{enumerate}
\item $F_1(0)=0$ and $F_1(1)=1$, and
\item $F_1'(0)=d_0$ and $F_1'(1)=d_1$, for some arbitrary
positive values $d_0$ and $d_1$ to be chosen later.
\end{enumerate}

Such a function is straightforward to construct for any values of $d_0$ and
$d_1$, and by an appropriate choice of these values, it is possible to map
$(\boldsymbol{x},\boldsymbol{\theta})$ into
$(\boldsymbol{x},F(\boldsymbol{\theta}))$ in a way that does not change
$P_{\theta_1}$ or $P_{\theta_2}$, but adjusts $f(\theta_1)$ and $f(\theta_2)$
to any desired positive values.

Lastly, we show that $L(\theta_1,\theta_2)$ can also not depend on the values
of $\theta_1$ and $\theta_2$ other than through $P_{\theta_1}$ and
$P_{\theta_2}$. For this
we once again invoke IRP: by applying a similarity transform on $\Theta$, we
can map any $\theta_1$ and $\theta_2$ values into arbitrary new values, again
without this affecting their respective conditional data distributions.
\end{proof}

\begin{lemma}[Lemma~3.3 of the main text]\label{L:ccont}
The function $c$ is $\mathcal{M}$-continuous for continuous distributions,
in the sense that if both
$\left(f_i\right)_{i\in\mathbb{N}}\xrightarrow{\mathcal{M}} f$ and
$\left(g_i\right)_{i\in\mathbb{N}}\xrightarrow{\mathcal{M}} g$,
where $f$ and $g$ are pdfs and $\left(f_i\right)_{i\in\mathbb{N}}$ and
$\left(g_i\right)_{i\in\mathbb{N}}$ are pdf sequences,
then $\left(c[f_i,g_i]\right)_{i\in\mathbb{N}}\xrightarrow{\mathcal{M}} c[f,g]$.
\end{lemma}

\begin{proof}
For any $t_0\in [0,1]$, let $r_0=c[f,g](t_0)$, and let
$t_{\text{min}}$ and $t_{\text{max}}$ be the infimum $t$ and the
supremum $t$, respectively, for which $c[f,g](t)=r_0$.

Because $c[f,g]$ is a monotone increasing function,
$t_{\text{min}}$ is also the supremum $t$ for which
$c[f,g](t)< r_0$ (unless no such $t$ exists, in which case $t_{\text{min}}=0$),
so by definition $t_{\text{min}}$ is the supremum of
$\mathbf{P}_{\boldsymbol{x}\sim g}(r_{f,g}(\boldsymbol{x})\le r)$,
for all $r<r_0$, from which we conclude that
$t_{\text{min}}=\mathbf{P}_{\boldsymbol{x}\sim g}(r_{f,g}(\boldsymbol{x})<r_0)$.

Because $\left(f_i\right)\xrightarrow{\mathcal{M}} f$ and
$\left(g_i\right)\xrightarrow{\mathcal{M}} g$, we can use
(1.3) to determine that using a large enough $i$ both
$f_i(x)/f(x)$ and $g_i(x)/g(x)$ are arbitrarily close to $1$ in all
but a diminishing measure of $X$. Hence,
\begin{equation*}
\lim_{i\to\infty} \mathbf{P}_{\boldsymbol{x}\sim g_i}(r_{f_i,g_i}(\boldsymbol{x})<r_0)
=\lim_{i\to\infty} \mathbf{P}_{\boldsymbol{x}\sim g}(r_{f_i,g_i}(\boldsymbol{x})<r_0)
=\mathbf{P}_{\boldsymbol{x}\sim g}(r_{f,g}(\boldsymbol{x})<r_0)
=t_{\text{min}}.
\end{equation*}

We conclude that for any $t^{+}>t_{\text{min}}$ a large enough $i$ will satisfy
$\mathbf{P}_{\boldsymbol{x}\sim g_i}(r_{f_i,g_i}(\boldsymbol{x})<r_0)< t^{+}$,
and hence
$c[f_i,g_i](t^{+})\ge r_0$. For all such $t^{+}$, and in particular for all
$t^{+}>t_0$,
\begin{equation}\label{Eq:liminf}
\liminf_{i\to\infty} c[f_i,g_i](t^{+})\ge r_0 = c[f,g](t_0).
\end{equation}

A symmetrical analysis on $t_{\text{max}}$ yields that for all $t^{-}<t_0$,
\begin{equation}\label{Eq:limsup}
\limsup_{i\to\infty} c[f_i,g_i](t^{-})\le c[f,g](t_0).
\end{equation}

Consider, now, the functions
\[
c_{\text{sup}}(t)\defeq\limsup_{i\to\infty} c[f_i,g_i](t)
\]
and
\[
c_{\text{inf}}(t)\defeq\liminf_{i\to\infty} c[f_i,g_i](t).
\]
Because each
$c[f_i,g_i]$ is monotone increasing, so are $c_{\text{sup}}$ and
$c_{\text{inf}}$. Monotone functions can only have countably many discontinuity
points (for a total of measure zero). For any $t_0$ that is not
a discontinuity point of either
function, we have from \eqref{Eq:liminf} and \eqref{Eq:limsup}
that $\lim_{i\to\infty} c[f_i,g_i](t_0)$ exists and equals $c[f,g](t_0)$,
so the conditions of convergence in measure hold.
\end{proof}

\begin{lemma}[Lemma~3.4 of the main text]\label{L:r}
If $L$ is a
problem continuous con\-di\-tion\-al-distribution-based loss function that
satisfies IRO, and
$p$ and $q$ are piecewise-continuous probability density functions
that are data distributions
in an estimation problem $(\boldsymbol{x}, \boldsymbol{\theta})$, then
$L_{(\boldsymbol{x},\boldsymbol{\theta})}(p,q)$ depends only on $c[p,q]$ and on
the type of $(\boldsymbol{x}, \boldsymbol{\theta})$.
\end{lemma}

\begin{proof}
Fix the problem type.

Let us first note that the following conditions are equivalent.
\begin{enumerate}
\item $c[p,q]$ equals the indicator function on $(0,1]$ in all but a measure
zero of values,
\item $p$ equals $q$ in all but a measure zero of $X$, the joint support of
$p$ and $q$,
\item $p$ and $q$ are $\mathcal{M}$-equivalent, in the sense that a sequence
of elements all equal to $p$ nevertheless satisfies the condition of
$\mathcal{M}$-convergence to $q$, and
\item $L(p,q)=0$.
\end{enumerate}

The equivalence of the first and second conditions stems from the definition
of $c[p,q]$, the equivalence of the second and third conditions stems from the
definition of $\mathcal{M}$-convergence, and the equivalence of the last
two conditions follows from problem continuity, together with
(1.1).

Hence, if $c[p,q]$ equals the indicator function on $(0,1]$, this uniquely
determines the value of $L(p,q)$ to be $0$, in accordance with the lemma.

It remains to be proved that such a functional relationship from $c[p,q]$ to
$L(p,q)$ holds for all other $c[p,q]$ as well, but for this, in accordance with
the second condition stated, we can safely assume that $p$ and $q$ differ in a
positive measure of $X$. Because both $p$ and $q$ integrate to $1$, we can
consequently also assume that the two are not linearly dependent.
The remainder of this proof follows under this assumption.

Note first that because $L$ is known to be conditional-distribution-based, the value of
$L(p,q)$ is not dependent on the full details of the
estimation problem: it will be the same in any estimation problem of the
same type that contains the conditional data distributions $p$ and $q$.
Let us therefore design an estimation problem that is easy to analyse but
contains these two conditional data distributions.

Let $(\boldsymbol{x},\boldsymbol{\theta})$ be an estimation problem with
$\Theta=[0,1]^M$ and a uniform prior on $\boldsymbol{\theta}$. Its conditional data distribution
at $\theta_0=(0,\ldots,0)$ will be $p$, at $\theta_1=(1,0,\ldots,0)$
will be $q$,
and we will choose piecewise continuous conditional data distributions, $f_\theta$, over the rest
of the $\theta\in\{0,1\}^M$ so all $2^M$ are linearly independent, share the
same support, and differ
from each other over a positive measure of $X$, and so that their
respective $r_{f_\theta,q}$ values are all monotone weakly increasing
with $r_{p,q}$ and with each other.

If $M>1$, we further choose $f_{\theta_2}$ at $\theta_2=(0,1,0,\ldots,0)$ to
satisfy that $c[f_{\theta_2},q]$ is
monotone strictly increasing.
If $M=1$, this is not necessary and we, instead, choose $\theta_2=\theta_0$. 

We then extend this description
of $f^{(\boldsymbol{x},\boldsymbol{\theta})}_\theta$ at $\{0,1\}^M$ into a full
characterisation of all the problem's conditional data distributions by setting these to be
multilinear functions of the coordinates of $\theta$.

We now create a sequence of estimation problems,
$(\boldsymbol{x}_i,\boldsymbol{\theta})$ to satisfy the conditions of
$L$'s problem-continuity assumption. We do this by constructing a sequence
$\left(S_i\right)_{i\in\mathbb{N}}$ of subsets of $\mathbb{R}^N$ such that
for all $\theta\in\{0,1\}^M$,
$\mathbf{P}(\boldsymbol{x}\in S_i|\boldsymbol{\theta}=\theta)$ tends to $1$,
and for every $x\in S_i$,
\begin{equation}\label{Eq:epsilon_i}
\left|f^{(\boldsymbol{x},\boldsymbol{\theta})}_\theta(x)-f^{(\boldsymbol{x}_i,\boldsymbol{\theta})}_\theta(x)\right|<\epsilon_i,
\end{equation}
for an arbitrarily-chosen sequence $\left(\epsilon_i\right)_{i\in\mathbb{N}}$
tending to zero.
By setting the remaining conditional data distribution values as multilinear functions of the
coordinates of $\theta$, as above, the sequence
$(\boldsymbol{x}_i,\boldsymbol{\theta})$ will satisfy the problem-continuity
condition and will guarantee
\[
\lim_{i\to\infty} L_{(\boldsymbol{x}_i,\boldsymbol{\theta})}(\theta_0,\theta_1)=L_{(\boldsymbol{x},\boldsymbol{\theta})}(\theta_0,\theta_1)=L(p,q).
\]

We now construct the sequence $\left(S_i\right)_{i\in\mathbb{N}}$ explicitly.

Each $S_i$ will be describable by the positive parameters $(a,b,d,r)$ as
follows.
Let $C^N_d=\{x\in\mathbb{R}^N:|x|_{\infty}\le d/2\}$, i.e.\ the axis-parallel,
origin-centred, $N$-dimensional cube of side length $d$.
$S_i$ will be chosen to contain all $x\in C^N_d$
such that for all $\theta\in\{0,1\}^M$,
$a\le f^{(\boldsymbol{x},\boldsymbol{\theta})}_\theta(x)\le b$ and $x$ is at
least a distance of $r$ away from the nearest discontinuity point of
$f^{(\boldsymbol{x},\boldsymbol{\theta})}_\theta$, as well as from the origin.
By choosing small enough
$a$ and $r$ and large enough $b$ and $d$, it is always possible to make
$\mathbf{P}(\boldsymbol{x}\in S_i|\boldsymbol{\theta}=\theta)$ arbitrarily
close to $1$, so the sequence can be made to satisfy its requirements.

We will choose $d$ to be a natural.

We now describe how to construct each
$f^{(\boldsymbol{x}_i,\boldsymbol{\theta})}_\theta$ from its respective
$f^{(\boldsymbol{x},\boldsymbol{\theta})}_\theta$.
We first describe for each $\theta\in\{0,1\}^M$
a new function $g^i_\theta:\mathbb{R}^N\to\mathbb{R}^{\ge 0}$
as follows. Begin by setting
$g^i_\theta(x)=f^{(\boldsymbol{x},\boldsymbol{\theta})}_\theta(x)$
for all $x\in S_i$.
If $x\notin C^N_d$, set $g^i_\theta(x)$
to zero. Otherwise, complete the $g^i_\theta$ functions so that all are
linearly independent and so that each
is positive and continuous inside $C^N_d$, and integrates to $1$.
Note that because a neighbourhood around the origin is known to not be in
$S_i$, it is never the case that $S_i=C^N_d$.
This allows enough degrees of
freedom in completing the functions $g$ in order to meet all
their requirements.

As all $g^i_\theta$ are continuous
functions over the compact domain $C^N_d$,
by the Heine-Cantor Theorem~\citep{rudin1964principles} they are
uniformly continuous. There must therefore exist a natural $n$, such that we can
tile $C^N_d$ into sub-cubes of side length $1/n$ such that by setting each
$f^{(\boldsymbol{x}_i,\boldsymbol{\theta})}_\theta$ value in each sub-cube to a
constant for the sub-cube equal to the mean over the entire sub-cube tile
of $g^i_\theta$, the result will satisfy for all $x\in C^N_d$ and all
$\theta\in\{0,1\}^M$,
$\left|f^{(\boldsymbol{x}_i,\boldsymbol{\theta})}_\theta(x)-g^i_\theta(x)\right|<\epsilon_i$.
Because $f^{(\boldsymbol{x}_i,\boldsymbol{\theta})}$ is by design multi-linear
in $\theta$, this implies that for all
$\theta\in [0,1]^M$ and all $x\in S_i$, condition
\eqref{Eq:epsilon_i} is attained.
Furthermore, by choosing a large enough $n$, we can always ensure, because
the $g$ functions are continuous and linearly independent, that also the
$f^{(\boldsymbol{x}_i,\boldsymbol{\theta})}_\theta$ functions, for
$\theta\in \{0,1\}^M$ are linearly independent and differ in more than a
measure zero of $\mathbb{R}^N$. Together, these properties ensure that the
new problems constructed are well-defined and elementary.

We have therefore constructed $(\boldsymbol{x}_i,\boldsymbol{\theta})$ as
a sequence of elementary estimation problems that $\mathcal{M}$-approximate
$(\boldsymbol{x},\boldsymbol{\theta})$ arbitrarily well, while being entirely
composed of $f^{(\boldsymbol{x}_i,\boldsymbol{\theta})}_\theta$ functions whose
support is $C^N_{d_i}$ for some natural
$d_i$ and whose values within their support are piecewise-constant inside cubic
tiles of side-length $1/n_i$, for some natural $n_i$.

We now use IRO to reshape the observation space of the estimation problems
in the constructed sequence by a piecewise-continuous transform.

Namely, we take each constant-valued cube of side length $1/n_i$ and
transform it using a scaling transformation in each coordinate,
as follows. Consider a single cubic tile, and let
the value of $f^{(\boldsymbol{x}_i,\boldsymbol{\theta})}_{\theta_1}(x)$
at points $x$ that are within it be $G_i$.
We scale the first coordinate of the tile to be of length
$G_i/n_i^N$, and all other coordinates to be of length $1$.
Notably, this transformation increases the volume of the cube by a factor
of $G_i$, so the probability density inside the cube, for each $f_\theta$,
will drop by a corresponding factor of $G_i$.

We now place the transformed cubes by stacking them along the first
coordinate, sorted by increasing value of
$f^{(\boldsymbol{x}_i,\boldsymbol{\theta})}_{\theta_2}(x)/f^{(\boldsymbol{x}_i,\boldsymbol{\theta})}_{\theta_1}(x)$.

Notably, because the probability density $f_{\theta_1}$
in all transformed cubes is $G_i/G_i=1$, it is
possible to arrange all transformed cubes in this way so that, together, they
fill exactly the unit cube in $\mathbb{R}^N$.
Let the new estimation problems created in this way be
$(\boldsymbol{x}'_i,\boldsymbol{\theta})$,
let $t_i:\mathbb{R}^N\to\mathbb{R}^N$ be the transformation,
$t_i(\boldsymbol{x}_i)=\boldsymbol{x}'_i$, applied on
the observation space and let $t^1_i(x)$ be the first coordinate value of
$t_i(x)$.

By IRO, $L(f^{(\boldsymbol{x}'_i,\boldsymbol{\theta})}_{\theta_0},f^{(\boldsymbol{x}'_i,\boldsymbol{\theta})}_{\theta_1})=L(f^{(\boldsymbol{x}_i,\boldsymbol{\theta})}_{\theta_0},f^{(\boldsymbol{x}_i,\boldsymbol{\theta})}_{\theta_1})$,
which we know tends to $L(p,q)$.

Consider the probability density of each
$f^{(\boldsymbol{x}'_i,\boldsymbol{\theta})}_\theta$ over its support
$[0,1]^N$. This is a probability density that is uniform along all axes except
the first, but has some marginal, $s=s^i_\theta$, along the first axis.
We denote such a distribution by $D_N(s)$. Specifically, for $\theta=\theta_2$,
because of our choice of sorting order, we have
$s^i_{\theta_2}=c[f^{(\boldsymbol{x}_i,\boldsymbol{\theta})}_{\theta_2},f^{(\boldsymbol{x}_i,\boldsymbol{\theta})}_{\theta_1}]$,
so by Lemma~\ref{L:ccont}, this is known to $\mathcal{M}$-converge to
$c[f^{(\boldsymbol{x},\boldsymbol{\theta})}_{\theta_2},f^{(\boldsymbol{x},\boldsymbol{\theta})}_{\theta_1}]$.

If $M=1$, the above is enough to show that the $\mathcal{M}$-limit problem
of $(\boldsymbol{x}'_i,\boldsymbol{\theta})$ exists. If $M>1$, consider
the following.

Let $t:X\to [0,1]$ be the transformation mapping each $x\in X$ to
the supremum of all $\tilde{t}$ for which
$c[f^{(\boldsymbol{x},\boldsymbol{\theta})}_{\theta_2},f^{(\boldsymbol{x},\boldsymbol{\theta})}_{\theta_1}](\tilde{t})\le f^{(\boldsymbol{x},\boldsymbol{\theta})}_{\theta_2}(x)/f^{(\boldsymbol{x},\boldsymbol{\theta})}_{\theta_1}(x)$.
By design,
$c[f^{(\boldsymbol{x},\boldsymbol{\theta})}_{\theta_2},f^{(\boldsymbol{x},\boldsymbol{\theta})}_{\theta_1}]$
is left continuous, so
$c[f^{(\boldsymbol{x},\boldsymbol{\theta})}_{\theta_2},f^{(\boldsymbol{x},\boldsymbol{\theta})}_{\theta_1}](t(x))= f^{(\boldsymbol{x},\boldsymbol{\theta})}_{\theta_2}(x)/f^{(\boldsymbol{x},\boldsymbol{\theta})}_{\theta_1}(x)$.

Summing up the entire construction,
in all but a diminishing measure of $x$ we have that
$f^{(\boldsymbol{x}_i,\boldsymbol{\theta})}_{\theta_2}(x)/f^{(\boldsymbol{x}_i,\boldsymbol{\theta})}_{\theta_1}(x)$ approaches $f^{(\boldsymbol{x},\boldsymbol{\theta})}_{\theta_2}(x)/f^{(\boldsymbol{x},\boldsymbol{\theta})}_{\theta_1}(x)$, which in turn equals the value $c[f^{(\boldsymbol{x},\boldsymbol{\theta})}_{\theta_2},f^{(\boldsymbol{x},\boldsymbol{\theta})}_{\theta_1}](t(x))$.
On the other hand, we have that
$s^i_{\theta_2}=c[f^{(\boldsymbol{x}_i,\boldsymbol{\theta})}_{\theta_2},f^{(\boldsymbol{x}_i,\boldsymbol{\theta})}_{\theta_1}]$ also $\mathcal{M}$-approaches
$c[f^{(\boldsymbol{x},\boldsymbol{\theta})}_{\theta_2},f^{(\boldsymbol{x},\boldsymbol{\theta})}_{\theta_1}]$
by Lemma~\ref{L:ccont}, and satisfies
$f^{(\boldsymbol{x}_i,\boldsymbol{\theta})}_{\theta_2}(x)/f^{(\boldsymbol{x}_i,\boldsymbol{\theta})}_{\theta_1}(x)=s^i_{\theta_2}(t^1_i(x))$.

Together, this indicates $\left(c[f^{(\boldsymbol{x},\boldsymbol{\theta})}_{\theta_2},f^{(\boldsymbol{x},\boldsymbol{\theta})}_{\theta_1}](t^1_i(x))\right) \xrightarrow{\mathcal{M}} c[f^{(\boldsymbol{x},\boldsymbol{\theta})}_{\theta_2},f^{(\boldsymbol{x},\boldsymbol{\theta})}_{\theta_1}](t(x))$.

Because
$c[f^{(\boldsymbol{x},\boldsymbol{\theta})}_{\theta_2},f^{(\boldsymbol{x},\boldsymbol{\theta})}_{\theta_1}]$
is monotone strictly increasing,
it follows that $t^1_i(x)$ $\mathcal{M}$-converges to $t(x)$.
For all other $\theta\in [0,1]^M$ this then implies that $s^i_\theta$
$\mathcal{M}$-converges to $c[f^{(\boldsymbol{x},\boldsymbol{\theta})}_{\theta},f^{(\boldsymbol{x},\boldsymbol{\theta})}_{\theta_1}]$,
because by construction all the problem's $r_{f_{\theta},q}$ are monotone
increasing with each other.

For all $\theta$,
the $f^{(\boldsymbol{x}'_i,\boldsymbol{\theta})}_{\theta}$ sequence therefore
has a limit, that limit being the distribution
$D_N(c[f^{(\boldsymbol{x},\boldsymbol{\theta})}_{\theta},f^{(\boldsymbol{x},\boldsymbol{\theta})}_{\theta_1}])$.

In particular, the limit at $\theta=\theta_0$ is $D_N(c[p,q])$ and the limit
at $\theta=\theta_1$ is $U([0,1]^N)$, the uniform distribution over the
unit cube.

By problem-continuity of $L$,
$L(D_N(c[p,q]),U([0,1]^N))=L(p,q)$.
Hence, $L(p,q)$ is a function of only $c[p,q]$.
\end{proof}

\begin{lemma}[Lemma~3.5 of the main text]\label{L:Fisher}
Let $L$ be a smooth conditional-distribution-based
loss function satisfying that $L(P,Q)$ is a function only of $c[P,Q]$
and of the problem type, and let $(\boldsymbol{x},\boldsymbol{\theta})$
be an elementary $\boldsymbol{\theta}$-continuous estimation problem of a type
on which $L$ is sensitive.

If one of the following conditions holds true:
\begin{enumerate}
\item The problem $(\boldsymbol{x},\boldsymbol{\theta})$ is a
continuous estimation problem, or,
\item The problem $(\boldsymbol{x},\boldsymbol{\theta})$ is a semi-continuous
estimation problem and $L$ satisfies ISI,
\end{enumerate}
then there exists a nonzero constant $\gamma$, dependent only on
the choice of $L$ and the problem type,
such that for every $\theta\in\Theta$
the Hessian matrix $H_L^{\theta}$ equals
$\gamma$ times the Fisher information matrix $\mathcal{I}_{\theta}$.
\end{lemma}

\begin{proof}
Fix the problem type.

We wish to calculate the derivatives of $L(P_{\theta_1},P_{\theta_2})$
according to $\theta_1$. For convenience, let us define a new function,
$\tilde{L}$, which describes $L$ in a one-parameter form, by
\[
\tilde{L}(r_{P,Q})=L(P,Q).
\]
This is possible by our assumption that $L(P,Q)$ is only a function of $c[P,Q]$.

We differentiate $\tilde{L}(r_{P,Q})$ as we would any composition of functions.
The derivatives of $r_{P_{\theta_1},P_{\theta_2}}$ in $\theta_1$ are
straightforward to compute, so we
concentrate on the question of how minute perturbations of $r$ affect
$\tilde{L}(r)$.

For this, we first extend the domain of $\tilde{L}(r)$. Natively, $\tilde{L}(r)$ is
only defined when $\mathbf{E}_{\boldsymbol{x}\sim Q}[r(\boldsymbol{x})]=1$.
However, to be able to perturb $r$ more freely, we define, for finite
expectation $r(\boldsymbol{x})$,
$\tilde{L}(r)=\tilde{L}(r/\mathbf{E}_{\boldsymbol{x}\sim Q}[r(\boldsymbol{x})])$.

Let $Y\subseteq X$ be a set with
$\mathbf{P}_{\boldsymbol{x}\sim Q}(\boldsymbol{x}\in Y)=\epsilon>0$ such that
for all $x\in Y$, $r(x)$ is a constant, $r_0$. The derivative of $\tilde{L}(r)$ in
$Y$ is defined as
\[
\left[\nabla_Y(\tilde{L})\right](r)=\lim_{\Delta\to 0} \frac{\tilde{L}(r+\Delta \cdot\chi_Y)-\tilde{L}(r)}{\Delta},
\]
where for any $S\subseteq\mathbb{R}^s$, we denote by $\chi_S$ the function
over $\mathbb{R}^s$ that yields $1$ when the input is in $S$ and $0$ otherwise.

By our smoothness assumption on $L$, this derivative is known to
exist,
because it is straightforward to construct an elementary, continuous
estimation problem for which these would be (up to normalisation) the
$r_{P_{\theta_0},P_{\theta_0+\Delta e_1}}$ values for
some $\theta_0$ and basis vector $e_1$.

Consider, now, what this derivative's value can be. By assumption that $L(P,Q)$
only depends on $c[P,Q]$,
we know that the derivative's value can depend on $\epsilon$, $r_0$ and
on the distribution of $r(\boldsymbol{x})$ for $\boldsymbol{x}\sim Q$, a
distribution we will name $D_Q(r)$, but it cannot depend
on any other properties of $Y$ or $r$, because any changes to $Y$ and $r$ that
do not affect $\epsilon$, $r_0$ and $D_Q(r)$
will yield the same $c[P,Q]$ values throughout the calculation of
$\left[\nabla_Y(\tilde{L})\right](r)$. Hence, we can describe the derivative as
$\Delta_{\text{x}}(\epsilon,r_0|D_Q(r))$.

As a next step, consider what would happen if we were to partition
$Y$ into $2$ sets, each of measure $\epsilon/2$ in $Q$. (Such a partition is
straightforward to produce in the continuous case. If $X$ is discrete, one can
create it by utilising ISI: translate $\boldsymbol{x}$ to
$(\boldsymbol{x},\boldsymbol{y})$, where $\boldsymbol{y}$ is a Bernoulli
random variable with $p=1/2$.)

The marginal impact of each set on the value of $L$ is
$\Delta_{\text{x}}(\epsilon/2,r_0|D_Q(r))$, but their total
impact is $\Delta_{\text{x}}(\epsilon,r_0|D_Q(r))$.
More generally, either using the problem's continuity or
using ISI and $L$'s smoothness, we can describe
$\Delta_{\text{x}}(\epsilon,r_0|D_Q(r))$ as
$\epsilon \cdot \Delta_{\text{x}}(r_0|D_Q(r))$.

Utilising $L(\theta_1,\theta_2)$'s representation as
$\tilde{L}(r)$, where $Q=P_{\theta_2}$ and
$r=r_{P_{\theta_1},P_{\theta_2}}$, and
noting that $\epsilon$ was, in the calculation above, the measure of $Y$ in
$Q=P_{\theta_2}$,
we can now write the first derivative of $L$ in some direction $i$ of
$\theta_1$ explicitly as an integral in this measure, i.e.\ in
``$\text{d}P_{\theta_2}$''.

For clarity of presentation, we will write this as an integral in
``$f_{\theta_2}(x)\text{d}x$'', using here and throughout the remainder of
the proof pdf notation, as would be appropriate when $\boldsymbol{x}$ is
known to be continuous. This change is meant merely to simplify the notation,
and in no way restricts the proof. Readers are welcome to verify that all
steps are equally valid for any $P_{\theta_1}$ and $P_{\theta_2}$
distributions.

Let $R_{\theta_1,\theta_2}$ be
$D_{f_{\theta_2}}(r_{f_{\theta_1},f_{\theta_2}})$.
If $c[f_{\theta_1},f_{\theta_2}]$ is a piecewise-constant
function, the derivative of $L$
can be written as follows.
\[
\frac{\partial L(\theta_1,\theta_2)}{\partial \theta_1(i)}=
\int_{X} \Delta_{\text{x}}(r_{f_{\theta_1},f_{\theta_2}}(x)|R_{{\theta_1},{\theta_2}}) \frac{\partial r_{f_{\theta_1},f_{\theta_2}}(x)}{\partial \theta_1(i)}f_{\theta_2}(x)\text{d}x.
\]

The same reasoning can be used to describe the second derivative of $L$
(this time in the directions $i$ and $j$ of $\theta_1$). The second derivative
of $\tilde{L}(r)$ when perturbing $r$ relative to two subsets $Y_1$ and $Y_2$ is
defined as
\[
\lim_{\Delta\to 0} \frac{\left[\nabla_{Y_1}(\tilde{L})\right](r+\Delta\cdot\chi_{Y_2})-\left[\nabla_{Y_1}(\tilde{L})\right](r)}{\Delta},
\]
and once again using the assumption that $L(P,Q)$ only
depends on $c[P,Q]$, we can see that if $Y_1$ and $Y_2$
are disjoint, if the measures of $Y_1$ and $Y_2$ in $Q$ are, respectively,
$\epsilon_1$ and $\epsilon_2$, both positive,
and if the value of $r_{P,Q}(x)$ for $x$ values in each subset
is a constant, respectively $r_1$ and $r_2$, then any such $Y_1$ and $Y_2$
will perturb $\tilde{L}(r)$ in exactly the same amount over any $r$ with the
same $D_Q(r)$.
We name the second derivative coefficient
in this case $\Delta_{\text{xy}}(r_1,r_2|D_Q(r))$.

The caveat that $Y_1$ and $Y_2$ must be disjoint is important, because if
$Y=Y_1=Y_2$ the symmetry no longer holds. This is a second case, and for it
we must define a different coefficient $\Delta_{\text{xx}}(r_0|D_Q(r))$, where
$r_0=r_1=r_2$.

In the case where $c[f_{\theta_1},f_{\theta_2}]$ is
a piecewise-constant function, the second derivative
can therefore be written as
\begin{equation}\label{Eq:second_deriv}
\begin{split}
\frac{\partial^2  L(\theta_1,\theta_2)}{\partial \theta_1(i)\partial \theta_1(j)}=
&\int_{X} \Delta_{\text{x}}(r_{f_{\theta_1},f_{\theta_2}}(x)|R_{{\theta_1},{\theta_2}}) \frac{\partial^2 r_{f_{\theta_1},f_{\theta_2}}(x)}{\partial \theta_1(i)\partial \theta_1(j)}f_{\theta_2}(x)\text{d}x \\
&\quad+\int_{X}\int_{X} \Delta_{\text{xy}}(r_{f_{\theta_1},f_{\theta_2}}(x_1),r_{f_{\theta_1},f_{\theta_2}}(x_2)|R_{{\theta_1},{\theta_2}}) \\
&\quad\quad\quad\quad\quad
\frac{\partial r_{f_{\theta_1},f_{\theta_2}}(x_1)}{\partial \theta_1(i)}\frac{\partial r_{f_{\theta_1},f_{\theta_2}}(x_2)}{\partial \theta_1(j)}
f_{\theta_2}(x_2)\text{d}x_2 f_{\theta_2}(x_1)\text{d}x_1 \\
&\quad+\int_{X} \Delta_{\text{xx}}(r_{f_{\theta_1},f_{\theta_2}}(x)|R_{{\theta_1},{\theta_2}})
\frac{\partial r_{f_{\theta_1},f_{\theta_2}}(x)}{\partial \theta_1(i)}\frac{\partial r_{f_{\theta_1},f_{\theta_2}}(x)}{\partial \theta_1(j)}f_{\theta_2}(x)\text{d}x.
\end{split}
\end{equation}

In order to compute the Hessian matrix $H_L^{\theta}$, we consider each
matrix element $H_L^{\theta}(i,j)$ in turn.

This equals
${\partial^2  L(\theta_1,\theta_2)}/{\partial \theta_1(i)\partial \theta_1(j)}$
where $\theta_1=\theta_2=\theta$. In particular, $c[f_{\theta_1},f_{\theta_2}]$
is $\chi_{(0,1]}$ and $R_{{\theta_1},{\theta_2}}$ is the distribution
$\mathbf{1}$, which yields the constant value $1$.

The value of \eqref{Eq:second_deriv} in this case becomes
\begin{equation}\label{Eq:second_deriv_at_1}
\begin{aligned}
&\Delta_{\text{x}}\left(1\middle|\mathbf{1}\right)\int_{X} \left.\frac{\partial^2 r_{f_{\theta_1},f_{\theta}}(x)}{\partial \theta_1(i)\partial \theta_1(j)}\right\rvert_{\theta_1=\theta}f_{\theta}(x)\text{d}x \\
&\quad +\Delta_{\text{xy}}\left(1,1\middle|\mathbf{1}\right)\left(\int_{X} \left.\frac{\partial r_{f_{\theta_1},f_{\theta}}(x)}{\partial \theta_1(i)}\right\rvert_{\theta_1=\theta}f_{\theta}(x)\text{d} x\right)
\left(\int_{X} \left.\frac{\partial r_{f_{\theta_1},f_{\theta}}(x)}{\partial \theta_1(j)}\right\rvert_{\theta_1=\theta}f_{\theta}(x)\text{d} x\right)\\
&\quad +\Delta_{\text{xx}}\left(1\middle|\mathbf{1}\right)\int_{X} \left(\left.\frac{\partial r_{f_{\theta_1},f_{\theta}}(x)}{\partial \theta_1(i)}\right\rvert_{\theta_1=\theta}\right)
\left(\left.\frac{\partial r_{f_{\theta_1},f_{\theta}}(x)}{\partial \theta_1(j)}\right\rvert_{\theta_1=\theta}\right)f_{\theta}(x)\text{d}x.
\end{aligned}
\end{equation}

Note, however, that because $(\boldsymbol{x},\boldsymbol{\theta})$ is an
estimation problem, i.e.\ all its conditional data distributions are probability measures, not
general measures, it is the case that
\[
\int_{X} r_{f_{\theta_1},f_{\theta}}(x)f_{\theta}(x)\text{d}x
=\int_{X} f_{\theta_1}(x)\text{d}x=1,
\]
and is therefore a constant
independent of either $\theta_1$ or $\theta$. Its various derivatives in
$\theta_1$ are accordingly all zero. This makes the first two summands in
\eqref{Eq:second_deriv_at_1} zero. What is left, when setting
$\gamma=\Delta_{\text{xx}}(1|\mathbf{1})$, is
\begin{align*}
H_L^{\theta}(i,j)
&=\gamma \int_{X} \left(\left.\frac{\partial r_{f_{\theta_1},f_{\theta}}(x)}{\partial \theta_1(i)}\right\rvert_{\theta_1=\theta}\right)
\left(\left.\frac{\partial r_{f_{\theta_1},f_{\theta}}(x)}{\partial \theta_1(j)}\right\rvert_{\theta_1=\theta}\right)f_{\theta}(x)\text{d}x\\
&=\gamma \int_X \left(\left.\frac{\partial f_{\theta_1}(x)/f_{\theta}(x)}{\partial \theta_1(i)}\right\rvert_{\theta_1=\theta}\right)
\left(\left.\frac{\partial f_{\theta_1}(x)/f_{\theta}(x)}{\partial \theta_1(j)}\right\rvert_{\theta_1=\theta}\right) f_{\theta}(x)\text{d}x\\
&=\gamma \int_X \left(\frac{\partial \log f_{\theta}(x)}{\partial \theta(i)}\right)\left(\frac{\partial \log f_{\theta}(x)}{\partial \theta(j)}\right) f_{\theta}(x)\text{d}x\\
&=\gamma \mathbf{E}_{\boldsymbol{x}\sim f_{\theta}}\left[\left(\frac{\partial \log f_{\theta}(\boldsymbol{x})}{\partial \theta(i)}\right)\left(\frac{\partial \log f_{\theta}(\boldsymbol{x})}{\partial \theta(j)}\right)\right]\\
&=\gamma \mathcal{I}_{\theta}(i,j).
\end{align*}

Hence, $H_L^{\theta}=\gamma \mathcal{I}_{\theta}$.

The only difference in this derivation for the case where $\boldsymbol{x}$ is
discrete, is that in this case the final result would have used probabilities
rather than
probability densities. This is consistent, however, with the way Fisher
information is defined in this more general case. In fact, some sources
(e.g.,~\citealp{bobkov2014fisher}) define the Fisher information directly from
Radon-Nikodym derivatives.

As a final point in the proof, we remark that $\gamma$
must be nonzero, because had it been zero, $H_L^{\theta}$
would have been zero for every $\theta$ in every elementary
$\boldsymbol{\theta}$-continuous estimation problem of the same type,
contrary to our sensitivity assumption on the loss function.
\end{proof}

\begin{theorem}[Theorem~3.6 of the main text]\label{T:semicontinuous}
If $(\boldsymbol{x},\boldsymbol{\theta})$ is an elementary, semi-continuous
estimation problem
for which $\hat{\theta}_{\text{WF}}$ is a well-defined
set estimator, and if $L$ is a smooth loss function,
sensitive on $(\boldsymbol{x},\boldsymbol{\theta})$'s type,
that satisfies all of IIA, IRP, IRO and ISI, then
\[
\hat{\theta}^L_{\text{EIC}}=\hat{\theta}_{\text{WF}}.
\]
\end{theorem}

\begin{proof}
The proof is essentially identical to that of Theorem~\ref{T:main}. The only
change is that we can no longer apply Lemma~\ref{L:r}. However, we claim that
$L(P,Q)$ only depends on $c[P,Q]$ and $M$ despite this, for which reason we can
still apply Lemma~\ref{L:Fisher}, as before, to complete the proof.

Fix $M$, the dimension of $\Theta$.

To prove that $L(P,Q)$ only depends on $c[P,Q]$,
let $(\boldsymbol{x},\boldsymbol{\theta})$ be an elementary, semi-continuous
estimation problem,
and let $L$ be any smooth loss function,
sensitive on its type, that satisfies all of IIA, IRP, IRO
and ISI. Furthermore, fix $\theta_1,\theta_2\in\Theta$, and let $P$ be the
distribution $P_{\theta_1}$ and $Q$ be the distribution $P_{\theta_2}$. We will
show that $L(P,Q)$ can only depend on $c[P,Q]$.

We note that we can freely assume $P\ne Q$, because
\[
c[P,Q]=\chi_{(0,1]} \Leftrightarrow P=Q \Leftrightarrow L(P,Q)=0,
\]
where the first equivalence follows from the definition of $c[P,Q]$ (for
discrete distributions $P$ and $Q$) and the second equivalence follows from
(1.1).

The assumption $P\ne Q$ also implies
that the two distributions are not linearly dependent.

We begin by describing a new estimation problem
$((\boldsymbol{x},\boldsymbol{y},\boldsymbol{z}),\boldsymbol{\theta})$,
where $\boldsymbol{y}$ is a random variable independent of $\boldsymbol{x}$ and
of $\boldsymbol{\theta}$, which is uniformly distributed in $\{1,\ldots,2^M\}$,
and where $\boldsymbol{z}=r_{P,Q}(\boldsymbol{x})$.

We know that
$L_{(\boldsymbol{x},\boldsymbol{\theta})}(\theta_1,\theta_2)
=L_{((\boldsymbol{x},\boldsymbol{y},\boldsymbol{z}),\boldsymbol{\theta})}(\theta_1,\theta_2)
=L_{((\boldsymbol{z},\boldsymbol{y},\boldsymbol{x}),\boldsymbol{\theta})}(\theta_1,\theta_2)$,
because the first equality stems from ISI and the second from IRO.

In principle, we now wish to apply ISI again, in order to show that
\begin{equation}\label{Eq:rISI}
L_{((\boldsymbol{z},\boldsymbol{y},\boldsymbol{x}),\boldsymbol{\theta})}(\theta_1,\theta_2)
=L_{((\boldsymbol{z},\boldsymbol{y}),\boldsymbol{\theta})}(\theta_1,\theta_2),
\end{equation}
because if we can establish \eqref{Eq:rISI} then we are done. The reason for
this is that from Lemma~\ref{L:likelihood} we know that the value of
$L_{((\boldsymbol{z},\boldsymbol{y}),\boldsymbol{\theta})}(\theta_1,\theta_2)$
can only depend on the conditional data distributions at $\theta_1$ and $\theta_2$,
and by construction these conditional data distributions depend only on the distribution of
$\boldsymbol{z}=r_{P,Q}(\boldsymbol{x})$ at $\boldsymbol{x}\sim P$ and at
$\boldsymbol{x}\sim Q$, in both of which the distribution of
$\boldsymbol{z}$ is fully constructible from knowledge of $c[P,Q]$.

Unfortunately, it is not possible for us to apply ISI directly to prove
\eqref{Eq:rISI}. The reason for this is that even though (as we will
demonstrate) $\boldsymbol{x}$ is independent of the choice of $\theta_1$ versus
$\theta_2$ given $(\boldsymbol{z},\boldsymbol{y})$, this
independence may not extend to all other parameter choices in $\Theta$.

We therefore first apply Lemma~\ref{L:likelihood} in order to show that
$L_{((\boldsymbol{z},\boldsymbol{y},\boldsymbol{x}),\boldsymbol{\theta})}(\theta_1,\theta_2)
=L_{((\boldsymbol{z},\boldsymbol{y},\boldsymbol{x}),\boldsymbol{\theta'})}(\theta'_1,\theta'_2)$,
for any elementary estimation problem
$((\boldsymbol{z},\boldsymbol{y},\boldsymbol{x}),\boldsymbol{\theta'})$ over
any parameter domain $\Theta'$ of the same dimension, as long as 
the distribution of $(\boldsymbol{z},\boldsymbol{y},\boldsymbol{x})$ given
$\boldsymbol{\theta'}=\theta'_1$ is the same as its distribution given
$\boldsymbol{\theta}=\theta_1$, and its distribution given
$\boldsymbol{\theta'}=\theta'_2$ is the same as given
$\boldsymbol{\theta}=\theta_2$.

We will construct such an estimation problem where $\boldsymbol{x}$ is
independent of $\boldsymbol{\theta'}$ given $(\boldsymbol{z},\boldsymbol{y})$
across the entire range of $\boldsymbol{\theta'}$.

To do this, let $\Theta'$ be the unit cube of dimension $M$, and let
$\theta'_1=(1,0,\ldots,0)$ and $\theta'_2=(0,\ldots,0)$. The conditional data distributions at
$\theta'_1$ and $\theta'_2$ will be the same as in $\theta_1$ and $\theta_2$
in the original problem, so as to enable us to apply Lemma~\ref{L:likelihood}.

We will retain the definition of $\boldsymbol{z}$ as $r_{P,Q}(\boldsymbol{x})$.
The variable $\boldsymbol{y}$ will remain independent of $\boldsymbol{x}$, but
in the new problem it will depend on $\boldsymbol{\theta'}$. Specifically, it
will remain uniformly distributed over $\{1,\ldots,2^M\}$ only for
$\boldsymbol{\theta'}$ values of the form $(\alpha,0,\ldots,0)$. In the
remaining $2^M-2$ corners of the unit cube, $\boldsymbol{y}$ will have other
distributions, arbitrarily chosen from all distributions over the same support,
such that all $2^M-1$ such distributions of $\boldsymbol{y}$ are linearly
independent (which is clearly possible to do).

The distribution of $\boldsymbol{x}$ will be $P$ at all corners of the unit
cube other than in the origin.

Elsewhere in the unit cube, away from its corners, the distribution of
$(\boldsymbol{z},\boldsymbol{y},\boldsymbol{x})$ will be the multilinear
continuation of their distributions at the corners. Because, by the
construction of $\boldsymbol{y}$, we know the corner distributions to all
be linearly independent other than at $\theta'_1$ and $\theta'_2$, the
constructed problem is well-defined (in that it
satisfies (1.1)), and is elementary.

We will now show that $\boldsymbol{x}$ is independent of $\boldsymbol{\theta'}$
given $(\boldsymbol{z},\boldsymbol{y})$. Consider any specific value $x$ of
$\boldsymbol{x}$ and $y$ of $\boldsymbol{y}$,
and let $z=r_{P,Q}(x)$ be the $\boldsymbol{z}$ value corresponding to
$\boldsymbol{x}=x$.
Furthermore, let $X_z$ be the set of all values of $\boldsymbol{x}$
that share the same $\boldsymbol{z}$ value. The variable $\boldsymbol{x}$
is independent of $\boldsymbol{\theta'}$ given
$(\boldsymbol{z},\boldsymbol{y})$ if for every
such $x$, the conditional probability
$\mathbf{P}(\boldsymbol{x}=x|(\boldsymbol{z},\boldsymbol{y})=(z,y), \boldsymbol{\theta'}=\theta')$
is not a function of $\theta'$. This conditional probability can be computed as
\begin{equation}\label{Eq:condprob}
\begin{aligned}
\mathbf{P}(\boldsymbol{x}=x|(\boldsymbol{z},\boldsymbol{y})=(z,y), \boldsymbol{\theta'}=\theta')
&=\frac{\mathbf{P}(\boldsymbol{x}=x,\boldsymbol{y}=y|\boldsymbol{\theta'}=\theta')}{\sum_{x'\in X_z} \mathbf{P}(\boldsymbol{x}=x',\boldsymbol{y}=y|\boldsymbol{\theta'}=\theta')} \\
&=\frac{\mathbf{P}(\boldsymbol{x}=x|\boldsymbol{\theta'}=\theta') \mathbf{P}(\boldsymbol{y}=y|\boldsymbol{\theta'}=\theta')}{\sum_{x'\in X_z} \mathbf{P}(\boldsymbol{x}=x'|\boldsymbol{\theta'}=\theta') \mathbf{P}(\boldsymbol{y}=y|\boldsymbol{\theta'}=\theta')} \\
&=\frac{\mathbf{P}(\boldsymbol{x}=x|\boldsymbol{\theta'}=\theta')}{\sum_{x'\in X_z} \mathbf{P}(\boldsymbol{x}=x'|\boldsymbol{\theta'}=\theta')}.
\end{aligned}
\end{equation}

Consider the distribution of $\boldsymbol{x}$ at
$\boldsymbol{\theta'}=\theta'$. By construction, it is
$\gamma P + (1-\gamma) Q$, for some $0\le\gamma\le 1$. Substituting this into
\eqref{Eq:condprob}, we get
\begin{align*}
\mathbf{P}(\boldsymbol{x}=x|(\boldsymbol{z},\boldsymbol{y})=(z,y), \boldsymbol{\theta'}=\theta')
&=\frac{\gamma \mathbf{P}_{\boldsymbol{x}\sim P}(x)+(1-\gamma) \mathbf{P}_{\boldsymbol{x}\sim Q}(x)}{\sum_{x'\in X_z} \gamma \mathbf{P}_{\boldsymbol{x}\sim P}(x')+(1-\gamma) \mathbf{P}_{\boldsymbol{x}\sim Q}(x')} \\
&=\frac{\gamma r_{P,Q}(x) \mathbf{P}_{\boldsymbol{x}\sim Q}(x)+(1-\gamma) \mathbf{P}_{\boldsymbol{x}\sim Q}(x)}{\sum_{x'\in X_z} \gamma r_{P,Q}(x') \mathbf{P}_{\boldsymbol{x}\sim Q}(x')+(1-\gamma) \mathbf{P}_{\boldsymbol{x}\sim Q}(x')}\\
&=\frac{\gamma z \mathbf{P}_{\boldsymbol{x}\sim Q}(x)+(1-\gamma) \mathbf{P}_{\boldsymbol{x}\sim Q}(x)}{\sum_{x'\in X_z} \gamma z \mathbf{P}_{\boldsymbol{x}\sim Q}(x')+(1-\gamma) \mathbf{P}_{\boldsymbol{x}\sim Q}(x')}\\
&=\frac{(\gamma z + (1-\gamma)) \mathbf{P}_{\boldsymbol{x}\sim Q}(x)}{\sum_{x'\in X_z} (\gamma z +(1-\gamma)) \mathbf{P}_{\boldsymbol{x}\sim Q}(x')}\\
&=\frac{\mathbf{P}_{\boldsymbol{x}\sim Q}(x)}{\sum_{x'\in X_z} \mathbf{P}_{\boldsymbol{x}\sim Q}(x')}.
\end{align*}

This is clearly not a function of $\theta'$, and hence the independence of
$\boldsymbol{x}$ is proved, and ISI can be employed to show
\[
L_{((\boldsymbol{z},\boldsymbol{y},\boldsymbol{x}),\boldsymbol{\theta'})}(\theta'_1,\theta'_2)
=L_{((\boldsymbol{z},\boldsymbol{y}),\boldsymbol{\theta'})}(\theta'_1,\theta'_2)
=L_{((\boldsymbol{z},\boldsymbol{y}),\boldsymbol{\theta})}(\theta_1,\theta_2),
\]
where the final equality is given, once again, by Lemma~\ref{L:likelihood}.

Thus, $L(P,Q)$ depends only on $c[P,Q]$ and $M$, and we can use
Lemma~\ref{L:Fisher}, as in Theorem~\ref{T:main},
to complete the proof.
\end{proof}

\subsection{Feasibility and necessity}\label{SS:feasibility}

\begin{theorem}[Theorem~4.3 of the main text]\label{T:feasibility}
Our system of axioms is feasible, in the sense that there exist
smooth and problem continuous
loss functions, $L$, sensitive on all $\boldsymbol{\theta}$-continuous
problem types, that satisfy all of IRP, IRO, IIA and ISI.
\end{theorem}

While the existence of such a loss function $L$ may not seem trivial, there are,
in fact, many loss functions that satisfy the necessary requirements.
We will show how to construct an infinite family of such $L$.

Our family of loss functions will be a subset of the
$f$-divergences~\citep{ali1966general}.

An $f$-divergence is a loss function $L$ that can be computed, most generally,
as
\[
L(P,Q)=\mathbf{E}_Q\left[F(r_{P,Q})\right].
\]
This is usually expressed as
\[
L(P,Q)=\int_X F(f_P(x)/f_Q(x)) f_Q(x) \text{d}x,
\]
which represents the case where $P$ and $Q$ are continuous distributions.

We call $F:\mathbb{R}^{>0}\to\mathbb{R}$ the $F$-function of the $f$-divergence
(refraining from using the more common term ``$f$-function'' so as to avoid
unnecessary confusion with our probability density functions).
It should be convex and satisfy $F(1)=0$.

\begin{lemma}\label{L:4axioms}
All $f$-divergences $L$ satisfy all of IRP, IRO, IIA and ISI.
\end{lemma}

\begin{proof}
As we have already established in Lemma~\ref{L:likelihood}, satisfying IRP
and IIA, our two parameter-space axioms, is tantamount to requiring that $L$
is a conditional-distribution-based loss function. The two observation-space axioms, IRO and
ISI, add that the loss must depend only on a sufficient statistic for the
observation. The former is proved directly from the definition
of $f$-divergences. The latter is among the most basic properties of
$f$-divergences (See, e.g., \cite{amari2000methods}).
\end{proof}

While all $f$-divergences satisfy IRP, IRO, IIA and ISI, not all of them
satisfy smoothness, sensitivity and problem continuity.

\begin{lemma}\label{L:wb}
Any $f$-divergence whose $F$-function satisfies the following conditions
\begin{enumerate}
\item $F$ has 3 continuous derivatives, and
\item $F''(1)>0$
\end{enumerate}
is smooth and is sensitive on all
$\boldsymbol{\theta}$-continuous problem types. If it also satisfies
\begin{enumerate}
\item $\lim_{x\to 0} F(x)<\infty$, and
\item $\lim_{x\to\infty} F'(x)<\infty$,
\end{enumerate}
it is also problem continuous.

Here, $F'$ and $F''$ are the first two derivatives of $F$,
which we assumed exist.
\end{lemma}

The main difficulty in Lemma~\ref{L:wb} is proving problem continuity.
We will build up to this using the following two lemmas.

\begin{lemma}\label{L:bound}
If $F:\mathbb{R}^{> 0}\to\mathbb{R}$ is continuously differentiable and
convex, and the limit $\lim_{x\to 0} F(x)$ is bounded, then
there is a value $x_0>0$, such that for every $x$ in $(0,x_0]$, $F'(x)>-1/x$,
where $F'$ is the derivative of $F$.
\end{lemma}

\begin{proof}
Because we know that $F$ is continuously differentiable,
we know that $F'$ exists and is continuous in this range. Furthermore, because
$F$ is convex, $F'$ is monotone increasing.

Let us assume, to the contrary, that no $x_0$ as in the claim exists. Then
there exists an infinite decreasing sequence $x_1,\ldots$, converging to zero,
for which $F'(x_i)\le -1/x_i$.

Because we know that $F'$ is monotone increasing, an upper bound for it in
the range $(x_{i+1},x_i]$ is $F'(x_i)$ and therefore also $-1/x_i$.

Integrating this, we see that a lower bound for $F(x_{i+1})$ is
\[
F(x_{i+1})\ge F(x_1)+\sum_{j=1}^{i} (x_j-x_{j+1}) (1/x_j).
\]
We will show, however, that as $i$ goes to infinity, this lower bound also
diverges to infinity, contradicting our assumption that
$\lim_{x\to 0} F(x)$ is bounded.

To see this, consider the series summand. This equals $1-x_{j+1}/x_j$. For
the series to converge, this summand must converge to $0$, so for all $j$
at least as large as some $j_0$ we have that there is some $r$ for which
$r\le x_{j+1}/x_{j}\le 1$. In this range, we know that
\[
\log(x_{j+1}/x_j)\ge(x_{j+1}/x_j-1) \frac{\log(r)}{r-1},
\]
by the concavity of the $\log$ function.

We can therefore conclude the following regarding the series partial sum.
\[
\sum_{j=j_0}^{\infty} 1-x_{j+1}/x_j
\ge -\frac{r-1}{\log(r)} \sum_{j=j_0}^{\infty} \log(x_{j+1}/x_j)
= -\frac{r-1}{\log(r)} \lim_{j\to\infty} \log(x_{j}/x_{j_0})
= \infty,
\]
proving the claim.
\end{proof}

\begin{lemma}\label{L:mcontinuity}
If $F:\mathbb{R}^{>0}\to\mathbb{R}$ is the $F$-function of an $f$-divergence,
$F$ is continuously differentiable, $\lim_{x\to 0} F(x)<\infty$ and
$\lim_{x\to\infty} F'(x)<\infty$, where $F'$ is the derivative of $F$, then
the $f$-divergence is a problem-continuous loss function.
\end{lemma}

\begin{proof}
Let $P$ and $Q$ be probability distributions over a ground set $X$, and let
$\{P_i\}$ and $\{Q_i\}$ be sequences of distributions over ground sets $X_i$,
converging in measure to $P$ and $Q$, respectively.

The condition of convergence in measure is that for any $\epsilon>0$ and any
$\delta>0$ there is an $i_0$, such that for every $i\ge i_0$,
\[
\mathbf{P}_{\boldsymbol{x}\sim P}\left(1-\epsilon\le\frac{\text{d} P_i}{\text{d} P}(\boldsymbol{x})\le 1+\epsilon\right)\ge 1-\delta,
\]
and
\[
\mathbf{P}_{\boldsymbol{x}\sim Q}\left(1-\epsilon\le\frac{\text{d} Q_i}{\text{d} Q}(\boldsymbol{x})\le 1+\epsilon\right)\ge 1-\delta.
\]

We wish to prove that if these conditions hold, then the $f$-divergences
between $P_i$ and $Q_i$ converge to the $f$-divergence between $P$ and $Q$,
where the $f$-divergence between $P$ and $Q$ is calculated as
\[
D(P,Q) = \mathbf{E}_{\boldsymbol{x}\sim P}\left[ F\left(\frac{\text{d} Q}{\text{d} P}(\boldsymbol{x})\right)\right].
\]

While these formulations, using the Radon-Nikodym derivative
$\text{d}Q/\text{d}P$, are the most general, for convenience of presentation,
we will consider all distributions as continuous, describable by their
probability density functions (pdfs). We will use $f_P$ to describe the pdf
of $P$. Our proof will remain entirely general, but the use of pdfs to
describe distributions will make their presentations easiest.

Using pdfs, the claims can be reworded as follows. We know that for all
$\epsilon>0$ and $\delta>0$ a sufficiently large $i$ will satisfy
\[
\mathbf{P}_{\boldsymbol{x}\sim P}\left(1-\epsilon\le\frac{f_{P_i}(\boldsymbol{x})}{f_P(\boldsymbol{x})}\le 1+\epsilon\right)\ge 1-\delta
\]
and
\[
\mathbf{P}_{\boldsymbol{x}\sim Q}\left(1-\epsilon\le\frac{f_{Q_i}(\boldsymbol{x})}{f_Q(\boldsymbol{x})}\le 1+\epsilon\right)\ge 1-\delta,
\]
and wish to prove that 
\begin{equation}\label{Eq:divergences}
\lim_{i\to\infty} \int_{X_i} F\left(\frac{f_{Q_i}(x)}{f_{P_i}(x)}\right) f_{P_i}(x) \text{d}x =  \int_X F\left(\frac{f_{Q}(x)}{f_{P}(x)}\right) f_{P}(x) \text{d}x.
\end{equation}

For a given $\epsilon$ and a given $i$, let $X^\epsilon_i$ be the subset
of $X$ wherein 
\begin{equation}\label{Eq:P_approx}
1-\epsilon\le\frac{f_{P_i}(x)}{f_P(x)}\le 1+\epsilon
\end{equation}
and
\begin{equation}\label{Eq:Q_approx}
1-\epsilon\le\frac{f_{Q_i}(x)}{f_Q(x)}\le 1+\epsilon
\end{equation}
are both satisfied.

To prove \eqref{Eq:divergences}, we will separate the difference between the
left-hand side and the right-hand side to four elements, then show that all
four must equal zero. The four sub-differences we will look at are:
\begin{equation}\label{Eq:diffhatXP}
\lim_{i\to\infty} \int_{X\setminus X^\epsilon_i} F\left(\frac{f_{Q}(x)}{f_{P}(x)}\right) f_{P}(x) \text{d}x,
\end{equation}

\begin{equation}\label{Eq:diffhatXPi}
\lim_{i\to\infty} \int_{X_i\setminus X^\epsilon_i} F\left(\frac{f_{Q_i}(x)}{f_{P_i}(x)}\right) f_{P_i}(x) \text{d}x,
\end{equation}

\begin{equation}\label{Eq:diffPPi}
\lim_{i\to\infty} \int_{X^\epsilon_i} \left|F\left(\frac{f_{Q_i}(x)}{f_{P_i}(x)}\right)\right| |f_{P_i}(x)-f_P(x)| \text{d}x
\end{equation}
and
\begin{equation}\label{Eq:diffF}
\lim_{i\to\infty} \int_{X^\epsilon_i} \left|F\left(\frac{f_{Q_i}(x)}{f_{P_i}(x)}\right) - F\left(\frac{f_{Q}(x)}{f_{P}(x)}\right)\right| f_{P}(x) \text{d}x.
\end{equation}

Regarding $X\setminus X^\epsilon_i$, we know that it has some measure
$\mu_P=\mu_P(\epsilon,i)$ in $P$ and some measure $\mu_Q=\mu_Q(\epsilon,i)$ in
$Q$, both tending to $0$ as $i$ goes to infinity, for any $\epsilon$.
Correspondingly, $X^\epsilon_i$ has measure $1-\mu_P$ in $P$
and $1-\mu_Q$ in $Q$.

Because within $X^\epsilon_i$ \eqref{Eq:P_approx} and \eqref{Eq:Q_approx} hold,
the measure of $X^\epsilon_i$ in $P_i$ is at least
$(1-\mu_P)(1-\epsilon)$ and its measure in $Q_i$ is at least
$(1-\mu_Q)(1-\epsilon)$. Correspondingly, the measure of
$X_i\setminus X^\epsilon_i$ in $P_i$ is at most $\mu_P+\epsilon-\mu_P\epsilon$,
and its measure in $Q_i$ is at most $\mu_Q+\epsilon-\mu_Q\epsilon$.

We have assumed that
$F:\mathbb{R}^{>0}\to\mathbb{R}$ is continuously differentiable and that it is
convex (because it is an $F$-function for an $f$-divergence). Under these
preconditions, the conditions assumed, $\lim_{x\to 0} F(x)<\infty$ and
$\lim_{x\to\infty} F'(x)<\infty$, can easily be shown to be equivalent to the
condition that there exists $A\ge 0$ and $B\ge 0$ such that for all $x$,
$|F(x)|\le Ax+B$.

Using $A$ and $B$, we can now bound \eqref{Eq:diffhatXP} as follows.
\begin{align*}
\lim_{i\to\infty} \int_{X\setminus X^\epsilon_i} F\left(\frac{f_{Q}(x)}{f_{P}(x)}\right) f_{P}(x) \text{d}x
&\le \lim_{i\to\infty} \int_{X\setminus X^\epsilon_i} \left(A\frac{f_{Q}(x)}{f_{P}(x)}+B\right) f_P(x) \text{d}x \\
&= \lim_{i\to\infty} \int_{X\setminus X^\epsilon_i} A f_Q(x) \text{d}x + \int_{X\setminus X^\epsilon_i} B f_P(x) \text{d}x \\
&= \lim_{i\to\infty} A\mu_Q + B\mu_P \\
&= 0.
\end{align*}
We bound \eqref{Eq:diffhatXPi} similarly.
\begin{align*}
\lim_{i\to\infty} \int_{X_i\setminus X^\epsilon_i} F\left(\frac{f_{Q_i}(x)}{f_{P_i}(x)}\right) f_{P_i}(x) \text{d}x
&\le \lim_{i\to\infty} \int_{X_i\setminus X^\epsilon_i} \left(A\frac{f_{Q_i}(x)}{f_{P_i}(x)}+B\right) f_{P_i}(x) \text{d}x \\
&= \lim_{i\to\infty} \int_{X_i\setminus X^\epsilon_i} A f_{Q_i}(x) \text{d}x + \int_{X_i\setminus X^\epsilon_i} B f_{P_i}(x) \text{d}x \\
& \le \lim_{i\to\infty} A(\mu_Q+\epsilon-\mu_Q\epsilon) + B(\mu_P+\epsilon-\mu_Q\epsilon) \\
&= (A+B)\epsilon.
\end{align*}

Regarding \eqref{Eq:diffPPi}, note that within $X^\epsilon_i$ we know that
$|f_{P_i}(x)-f_P(x)|\le f_{P_i}(x) \frac{\epsilon}{1-\epsilon}$. For this
reason, we can bound \eqref{Eq:diffPPi} as follows.
\begin{align*}
\lim_{i\to\infty} \int_{X^\epsilon_i} \left|F\left(\frac{f_{Q_i}(x)}{f_{P_i}(x)}\right)\right|& |f_{P_i}(x)-f_P(x)| \text{d}x \\
&\le \lim_{i\to\infty} \int_{X^\epsilon_i} \left(A\frac{f_{Q_i}(x)}{f_{P_i}(x)}+B\right)\left(\frac{\epsilon}{1-\epsilon}\right) f_{P_i}(x) \text{d}x \\
&\le \lim_{i\to\infty} \left(\frac{\epsilon}{1-\epsilon}\right) \left(\int_{X_i} A f_{Q_i}(x) \text{d}x + \int_{X_i} B f_{P_i}(x) \text{d}x\right) \\
&= \left(\frac{\epsilon}{1-\epsilon}\right)(A+B).
\end{align*}

It now only remains to bound \eqref{Eq:diffF}. This is once again an integral
over $X^\epsilon_i$, which is a domain where we know that
\[
\frac{1-\epsilon}{1+\epsilon}\left(\frac{f_{Q}(x)}{f_{P}(x)}\right)
\le \frac{f_{Q_i}(x)}{f_{P_i}(x)}
\le \frac{1+\epsilon}{1-\epsilon}\left(\frac{f_{Q}(x)}{f_{P}(x)}\right).
\]
Let us therefore define $\alpha(x)$ as the value in the range
\[
\frac{1-\epsilon}{1+\epsilon}\le 1+\alpha(x)\le \frac{1+\epsilon}{1-\epsilon}
\]
that maximises
\[
\left|F\left((1+\alpha(x))\frac{f_{Q}(x)}{f_{P}(x)}\right) - F\left(\frac{f_{Q}(x)}{f_{P}(x)}\right)\right|.
\]

By the Mean Value Theorem, for each $x$ there is a $\gamma(x)$,
\[
\frac{1-\epsilon}{1+\epsilon}\le 1+\gamma(x)\le \frac{1+\epsilon}{1-\epsilon},
\]
such that
\[
F\left((1+\alpha(x))\frac{f_{Q}(x)}{f_{P}(x)}\right) - F\left(\frac{f_{Q}(x)}{f_{P}(x)}\right) = F'\left((1+\gamma(x))\frac{f_{Q}(x)}{f_{P}(x)}\right)\alpha(x)\frac{f_{Q}(x)}{f_{P}(x)}.
\]

The value of \eqref{Eq:diffF} can therefore be bounded from above by

\begin{align*}
\lim_{i\to\infty} & \int_{X^\epsilon_i} \left|F\left(\frac{f_{Q_i}(x)}{f_{P_i}(x)}\right) - F\left(\frac{f_{Q}(x)}{f_{P}(x)}\right)\right| f_{P}(x) \text{d}x \\
&\le \lim_{i\to\infty} \int_{X^\epsilon_i} \left|F'\left((1+\gamma(x))\frac{f_{Q}(x)}{f_{P}(x)}\right)\right||\alpha(x)|\frac{f_{Q}(x)}{f_{P}(x)}f_{P}(x) \text{d}x.
\end{align*}

As these equations should be valid for all $\epsilon$, let us assume that
$\epsilon\le 1/4$, in which case $\alpha(x)$ and $\gamma(x)$ are both bounded
in the range $[-2/5, 2/3]$.

Because of our assumption that $\lim_{y\to 0} F(y)<\infty$, we can now use
Lemma~\ref{L:bound} to know that there is a $y_0$ such that if
$y\le y_0$, the value of $F'(y)$ must be greater than $-1/y$.

Let us now partition $X^\epsilon_i$ to those $x$ values for which
$(1+\gamma(x))\frac{f_{Q}(x)}{f_{P}(x)}\le y_0$ and
$F'\left((1+\gamma(x))\frac{f_{Q}(x)}{f_{P}(x)}\right)<0$, which we will refer
to as $L^\epsilon_i$, and the rest, which we will refer to as $H^\epsilon_i$.

Within the domain $L^\epsilon_i$, we have
\begin{align*}
\lim_{i\to\infty} \int_{L^\epsilon_i} \left|F'\left((1+\gamma(x))\frac{f_{Q}(x)}{f_{P}(x)}\right)\right|&|\alpha(x)|\frac{f_{Q}(x)}{f_{P}(x)}f_{P}(x) \text{d}x \\
&\le \lim_{i\to\infty} \int_{L^\epsilon_i} \frac{|\alpha(x)|}{1+\gamma(x)} \left(\frac{f_P(x)}{f_Q(x)}\right)\frac{f_{Q}(x)}{f_{P}(x)}f_{P}(x) \text{d}x \\
&\le \lim_{i\to\infty} 2\epsilon\frac{1+\epsilon}{(1-\epsilon)^2} \int_X f_P(x)\text{d}x \\
&\le \frac{40}{9}\epsilon,
\end{align*}
where the last equality plugs in our condition $\epsilon\le 1/4$.

Finally, within the domain $H^\epsilon_i$, we know that
$|F'(x)|$ is bounded. This is because of the following conditions. First,
because $F$ is convex, $F'$ is monotone increasing, so its supremum is
$\lim_{x\to\infty} F'(x)$, which we assumed to be bounded.
On the other hand, its values must either satisfy
$F'\left((1+\gamma(x))\frac{f_{Q}(x)}{f_{P}(x)}\right)\ge 0$, in which case
they are bounded from below by $0$, or
$(1+\gamma(x))\frac{f_{Q}(x)}{f_{P}(x)}> y_0$, in which case they are bounded
from below by $F'(y_0)$.

Let $A'\ge 0$ be an upper bound for $|F'(x)|$.

We now have
\begin{align*}
\lim_{i\to\infty} \int_{H^\epsilon_i} \left|F'\left((1+\gamma(x))\frac{f_{Q}(x)}{f_{P}(x)}\right)\right||\alpha(x)|\frac{f_{Q}(x)}{f_{P}(x)}f_{P}(x) \text{d}x
&\le \lim_{i\to\infty} A' \frac{2\epsilon}{1-\epsilon} \int_X f_Q(x) \text{d}x \\
&\le \epsilon\frac{8 A'}{3},
\end{align*}
where the last equality, as before, stems from our assumption
$\epsilon\le 1/4$.

In total, we've established that
\begin{align}\label{Eq:finaldiff}
&\left|\lim_{i\to\infty} \int_{X_i} F\left(\frac{f_{Q_i}(x)}{f_{P_i}(x)}\right) f_{P_i}(x) \text{d}x -  \int_X F\left(\frac{f_{Q}(x)}{f_{P}(x)}\right) f_{P}(x) \text{d}x\right| \\
&\quad\quad\quad\quad\le \left(\frac{7}{3}(A+B)+\frac{40}{9}+\frac{8 A'}{3}\right)\epsilon,
\end{align}
under our assumption $\epsilon\le 1/4$.

Because $A$, $B$ and $A'$ are nonnegative constants, independent of $\epsilon$,
and because this equality must hold for any $\epsilon$ in $(0,1/4]$, we
conclude that the true difference on the left-hand side of \eqref{Eq:finaldiff}
(which is independent of $\epsilon$) must be zero, so
\eqref{Eq:divergences} holds and the lemma is proved.
\end{proof}

\begin{proof}[Proof of Lemma~\ref{L:wb}]
Any $f$-divergence whose $F$-function has 3 continuous derivatives satisfies
the smoothness condition. This can be verified simply by explicitly computing
the required derivatives from the $f$-divergence's formula. In particular, this
computation shows that the Hessian equals $F''(1)$ times the Fisher information
matrix. Therefore, if $F''(1)>0$ the $f$-divergence also satisfies sensitivity
for all problem types.
Problem continuity, if relevant, is then proved by Lemma~\ref{L:mcontinuity}.
\end{proof}

\begin{proof}[Proof of Theorem~\ref{T:feasibility}]
From Lemma~\ref{L:4axioms} and Lemma~\ref{L:wb}, we know that it is enough to
show an example of a loss function $L$ that is an $f$-divergence satisfying
the extra conditions of Lemma~\ref{L:wb}.

A concrete example of a commonly-used $L$ function satisfying all criteria
is squared Hellinger distance~\citep{pollard2002user},
\[
H^2(p,q)=\frac{1}{2}\int_X \left(\sqrt{p(x)}-\sqrt{q(x)}\right)^2\text{d} x,
\]
which is the $f$-divergence whose $F$-function is $F(r)=1-\sqrt{r}$.
\end{proof}

\begin{theorem}[Theorem~4.4 of the main paper]\label{T:nec_disc}
All axioms used in Theorem~3.1 are necessary.
\end{theorem}

\begin{proof}
Much as the proof of feasibility was simply a proof by example, proving
necessity will be done by counterexample: we will show alternate
loss functions, $L$, leading to alternate estimations, that satisfy all axioms
but one.
For each such $L$, we will show that its $|H^\theta_L|$ is not proportional to
$|\mathcal{I}_\theta|$, and hence produces an estimator different to WF.

A well-known loss function that satisfies IRO and IIA but not IRP is
quadratic loss,
\begin{equation}\label{Eq:quadloss}
L(\theta_1,\theta_2)=|\theta_1-\theta_2|^2.
\end{equation}
As was demonstrated in the main paper,
an error intolerant estimator with this loss function yields the
continuous MAP estimate, different to the WF estimate.

A loss function satisfying IRP and IIA but not IRO is
\begin{equation}\label{Eq:noIROcont}
L(P,Q)=\int_X f_Q(x) (f_P(x)-f_Q(x))^2 \text{d}x,
\end{equation}
which is the expected
square difference between the probability densities at $\boldsymbol{x}\sim Q$.
Calculating $H_L^\theta$ we get
\[
H_L^\theta(i,j)=\mathbf{E}_{\boldsymbol{x}\sim f_\theta}\left[2\left(\frac{\partial f_\theta(\boldsymbol{x})}{\partial \theta(i)}\right)\left(\frac{\partial f_\theta(\boldsymbol{x})}{\partial \theta(j)}\right)\right],
\]
which is different to the Fisher information matrix, and defines an
estimator that is not WF.

To prove necessity of our third axiom, IIA,
we construct a loss function $L$ that satisfies IRP and IRO but not IIA as
follows.
Let $L_1$ and $L_2$ be two smooth and problem continuous
loss functions, sensitive on the relevant problem type, satisfying all axioms
(such as, for example, two $f$-divergences matching the criteria of
Lemma~\ref{L:wb}) and let $t$ be a threshold value.

Consider the function
\begin{equation}\label{Eq:noIIA_threshold}
P(\theta)=\mathbf{P}(L_1(\boldsymbol{\theta},\theta)\le t).
\end{equation}

By construction, this function is independent of representation.

Define
\begin{equation}\label{Eq:noIIA}
L(\theta_1,\theta_2)=P(\theta_2) L_2(\theta_1,\theta_2).
\end{equation}

The resulting $|H^\theta_L|$ equals
$P(\theta)^M |H_{L_2}^\theta|$,
where the equality stems from the fact that
$P(\theta_2)$ is independent of $\theta_1$ and therefore acts as a
constant multiplier in the calculation of the Hessian.

This new estimator is different to the Wallace-Freeman estimator in the fact
that it adds a weighing factor $P(\theta)^{M/2}$.
\end{proof}

\begin{theorem}[Theorem~4.5 of the main paper]\label{T:nec_disc2}
All axioms used in Theorem~3.6 are necessary.
\end{theorem}

\begin{proof}
This proof is a direct continuation of the proof of Theorem~\ref{T:nec_disc},
and re-uses the same techniques and some of the same examples.

To begin with, quadratic loss, given in \eqref{Eq:quadloss} as an example of
a loss function that satisfies IIA and IRO but not IRP also demonstrates that
IRP is necessary in the semi-continuous case because (by not depending on
the estimation problem at all) it also satisfies ISI.

Similarly, the example of \eqref{Eq:noIIA} can be re-used here to demonstrate
the necessity of IIA also in the semi-continuous case. We have already shown
that this loss function satisfies IRP and IRO. To show the remaining condition,
namely that it also satisfies ISI, one merely needs to choose two
loss functions $L_1$ and $L_2$ that satisfy the conditions of
Theorem~\ref{T:feasibility}. By assumption, these satisfy ISI,
and so by construction \eqref{Eq:noIIA_threshold} is independent of
superfluous information and \eqref{Eq:noIIA} must be, too, proving the claim.

To demonstrate that ISI is necessary, consider
\begin{equation}\label{Eq:noISI}
L_1(P,Q)=\sum_{x\in X} \mathbf{P}_{\mathbf{x}\sim Q}(\mathbf{x}=x) \left(\mathbf{P}_{\mathbf{x}\sim P}(\mathbf{x}=x)-\mathbf{P}_{\mathbf{x}\sim Q}(\mathbf{x}=x)\right)^2,
\end{equation}
which is the discrete analogue of \eqref{Eq:noIROcont}.

This loss function clearly satisfies IRP and IIA, just as \eqref{Eq:noIROcont}
does. Neither satisfies ISI, however: if we, for example, replace
$\mathbf{x}$ by $(\mathbf{x},\mathbf{y})$, where $\mathbf{y}$ is a Bernoulli
random variable with $p=1/2$, losses will not be preserved. (All losses will,
in fact, be uniformly scaled down by a factor of $4$.)

In the proof of Theorem~\ref{T:nec_disc} we use \eqref{Eq:noIROcont} as an
example of a loss function that does not satisfy IRO. In the semi-continuous
domain, however, \eqref{Eq:noISI} does satisfy IRO. This is because the impact
of applying piecewise-diffeomorphic deformations on discrete domains is quite
different to applying them on continuous domains. In the continuous domain,
such functions can stretch and contract the probability space, causing
probability densities, such as those used in \eqref{Eq:noIROcont}, to change.
In the discrete domain, no such effects are possible: the probabilities used
in \eqref{Eq:noISI} do not change as a result of deformations of the
observation space, and so IRO is met.

However, as before, we can use the Hessian, which in this case is
\[
H_L^\theta(i,j)=\mathbf{E}_{\boldsymbol{x}\sim P_\theta}\left[2\left(\frac{\partial P_\theta(\boldsymbol{x})}{\partial \theta(i)}\right)\left(\frac{\partial P_\theta(\boldsymbol{x})}{\partial \theta(j)}\right)\right],
\]
to show that the resulting estimator is different to WF, proving the necessity
of ISI.

Lastly, we want to prove that IRO is necessary. For this, let
$x(1:k)$ be the value of $x$'s first $k$ dimensions, let $x(k)$ be the value of
its $k$'th dimension alone, and for a distribution $P$, if $x\sim P$,
let $P^y_k$ be the
distribution of $x(k)$ given that $x(1:k-1)=y$. Consider, now, a function
$L_N(P,Q)$ calculated in the following way.
\[
L_N(P,Q)=\sum_{k=1}^{N} \mathbf{E}_{\boldsymbol{x}\sim Q}\left[L_1\left(P^{\boldsymbol{x}(1:k-1)}_k,Q^{\boldsymbol{x}(1:k-1)}_k\right)\right],
\]
where we reuse the function $L_1$ from \eqref{Eq:noISI}, but use it only to
compare between distributions over a one-dimensional observation space.

By construction, $L_N$ satisfies ISI for any $L_1$, because any dimension
that does not add information also does not add to the value of $L_N(P,Q)$.
Therefore, it satisfies all of IRP, IIA and ISI.

However, as was already demonstrated regarding $L_1$ in the one-dimensional
case, it does not lead to WF, proving the necessity of IRO.
\end{proof}

\begin{theorem}[Theorem~4.7 of the main text]\label{T:axiom0}
MLE satisfies IRP, IRO, IIA and ISI.
\end{theorem}

\begin{proof}
Consider loss functions $L$ of the form
\[
L(\theta_1, \theta_2)=\gamma(\theta_2)L'(\theta_1, \theta_2),
\]
where $L'$ is a smooth and problem-continuous loss function, sensitive on
the relevant problem type.

In the proof to Theorem~2.8 of the main text, we showed that any PMLE
estimator can be described in this way, with $L'$ being quadratic loss.
Substituting the identity function for $g$ in (2.7), we get that MLE
can be described by
\[
\gamma(\theta_2) = f(\theta_2)^{2/M}.
\]

This $L$ satisfies IRO, IIA and ISI, but not IRP. We conclude from it that
MLE satisfies IRO, IIA and ISI.

To prove that MLE also satisfies IRP, let $L'$ be squared Hellinger distance
and let
\[
\gamma(\theta_2) = \left(\frac{f(\theta_2)}{\sqrt{|\mathcal{I}_{\theta_2}|}}\right)^{2/M}.
\]
This $L$, too, leads to MLE, because
$H_{L'}^{\theta_2}$ equals $\mathcal{I}_{\theta_2}$.

The new $L$ satisfies IRP because both $L'$ and $\gamma$ are invariant to the
representation of $\theta$.
\end{proof}

We note that even
though the theorem is solely about MLE, our construction is equally applicable
to any PMLE estimator whose penalty function $g$ satisfies the four expected
invariance properties.

\clearpage

\section{Supplementary discussion}\label{S:discussion}

\subsection{Comparison with MML}\label{SS:MMLcomp}

In the main text we discuss axioms that give rise to the
Wallace-Freeman point estimator (WF). This estimator was originally developed
in the Minimum Message Length (MML) literature, where it is most commonly
referred to as the Wallace-Freeman approximation.

In this section we discuss in greater detail the relationship between
Wallace-Freeman estimation as it appears in the MML literature, and how it
appears here, in the context of Error Intolerant Estimation.

\subsubsection{Introduction to MML}
MML is a general name for any member of the family of Bayesian statistical
inference methods based on the minimum message length principle. They are
closely related to the family of minimum description length (MDL) estimators
\citep{grunwald2007minimum,rissanen1987stochastic,rissanen1999hypothesis},
but predate them.

The minimum message length principle
was first introduced in \citet{WallaceBoulton1968}, and the
estimator that follows the principle directly, which was first described in
\citet{wallace1975invariant}, is known as Strict MML (SMML).

SMML can be defined as follows.
Given an estimation problem $(\boldsymbol{x},\boldsymbol{\theta})$ and given
an observation, $x$, we wish
to choose $\hat{\theta}(x)$ so that it optimally
trades off two ideals. First, it must be a good approximation to $\theta$,
the true value of $\boldsymbol{\theta}$, in that the distribution of
$\boldsymbol{x}$ given $\boldsymbol{\theta}=\hat{\theta}(x)$ is a good
approximation to its distribution given $\boldsymbol{\theta}=\theta$.
Second, the choice of $\hat{\theta}$ must be ``simple'', in the sense that
it upholds the ideal of Occam's razor. Both these ideals can be formulated
in information-theoretic terms.

Namely, suppose we wish to communicate both our estimate $\hat{\theta}(x)$
and the observation $x$, and suppose, further, that we choose for this the
following protocol known as the MML \emph{two-part message}.
First, we communicate $\hat{\theta}(x)$ via the optimal
communication protocol. (The choice of this optimal protocol is different,
depending on what the function $\hat{\theta}$ is.)
Second, we communicate $x$ via a protocol that is optimal only under the
assumption that $\boldsymbol{\theta}=\hat{\theta}(x)$.

The expected length of the first part of the message is the
Shannon entropy \citep{shannon1948} of $\hat{\theta}(\boldsymbol{x})$.
Thus, a ``simpler'' function
$\hat{\theta}:X\to\mathbb{R}^M$ will result in a shorter first part message
in expectation.
On the other hand, the second part of the message introduces inefficiencies
due to its approximation that $\boldsymbol{\theta}=\hat{\theta}(x)$. Thus,
the better the approximation the shorter the expected length of the
second part of the message becomes.

Strict MML chooses the function $\hat{\theta}$ so that the overall length
of the two-part message is minimised in expectation, creating, according to
MML theory, an optimal trade-off between simplicity and accuracy.

Notably, this construction only works if $X$ is countable (because otherwise
the second part of the message becomes infinite) and always results in a
$\hat{\theta}$ function whose range is only a countable subset of $\Theta$,
even if $\Theta$ itself is continuous (because otherwise the first part
of the message becomes infinite).

SMML is, however, used also for continuous $X$. This is done by extending
the construction above in one of several equivalent ways. For example,
instead of optimising the message length, it is possible to optimise the
\emph{excess message length}, which is the difference between the expected
message length and the information-theoretical optimum to deliver the
information without restrictions on the message format. This optimum is
the Shannon entropy of $\boldsymbol{x}$ and is
independent of the choice of $\hat{\theta}$. By subtracting this entropy,
the excess message length can be represented as a
Kullback-Leibler divergence. Because this divergence is computable also
on problems with a continuous $X$, it provides a natural way to extend the
definition of SMML to the continuous domain.

Nevertheless, even in those cases, SMML will advocate a $\hat{\theta}$ whose
range is countable.

\subsubsection{Wallace-Freeman Estimation}

Strict MML, as a representation-invariant Bayesian point-estimation method, has
much theoretical appeal, but it cannot be used in practice because it is
computationally and analytically intractable in all but a select few
single-parameter cases \citep{dowty2015smml}.
It was proved to be NP-Hard to compute even in estimating the parameters of
a trinomial distribution \citep{farrwallace2002}.

Even beyond the computational problems, SMML is a difficult estimate to
work with. For example, it provides piecewise-constant estimates and can be
asymmetrical even when working on an estimation problem exhibiting symmetry.

\citet{WallaceFreeman1987} addressed these problems by developing
the Wallace-Freeman estimator (WF) as a computationally convenient
approximation to SMML. The approximation itself makes many assumptions
regarding the underlying estimation problem, and, furthermore, makes no
guarantees regarding the inaccuracy incurred by using it. The main
guarantees given are that WF, like SMML, is representation-invariant, and
that WF is in general not identical to SMML, because WF's estimates are a
continuous function of $x$ for problems with a smooth
$(\boldsymbol{x},\boldsymbol{\theta})$, whereas SMML is piecewise constant.

\subsubsection{The evolution of MML justifications}

Interestingly, though the MML criterion, as described above, is at heart an
information-theoretic criterion, when \citet{wallace1975invariant} describe
their motivation, it is not an information-theoretic one. The paper sets out
with a motivation remarkably similar to ours here: it aims to extend
discrete MAP into the continuous domain, while preserving certain good
properties of discrete MAP that are not met by continuous MAP.
\citeauthor{wallace1975invariant} reference, specifically, the idea of
representation invariance. They do not, however, provide any theoretical
rationale regarding why one should start with the discrete MAP estimate,
other than the following:

\begin{quote}
We do not here advance any argument in favour of [discrete MAP], save to note
that it in some sense yields the most plausible, or least improbable account
of what has been observed.
\end{quote}

\citeauthor{wallace1975invariant} derive Strict MML (SMML)
as a representation invariant extension of discrete MAP by
several approximations, chief of which is the discretisation of the
observation space. This particular approximation is a choice that makes a
significant difference to the meaning of representation invariance, and
the approximate nature of the resulting solution can be observed directly,
e.g.\ when
applying it on a discrete problem, where discrete MAP can also be used:
the SMML estimate and the discrete MAP estimate are not the same.

At the conclusion of \citet{wallace1975invariant}, however, the paper points
out that
the SMML estimate can be given a different interpretation, this time from an
information-theoretic perspective rather than a Bayesian perspective, and
under this alternate description SMML becomes an exact solution.

According to \citet{Dowe2008a}, this dual justification is not a fluke, but a
reflection of the different authors' perspectives:
\begin{quote}
Chris [Wallace] was already a Bayesian from his mid-20s in the 1950s while
David Boulton was clearly talking here in the spirit of (en)coding. Story has
it that they had their separate approaches, went away and did their mathematics
separately, re-convened about 6 weeks later and found out that they were doing
essentially the same thing.
\end{quote}

The paper then goes on to quote Wallace himself, from a talk given 2003,
describing the birth of SMML as a synthesis of these two ideas:
\begin{quote}
\verb|[|W\verb|]|e came together again and I looked at his maths and he looked
at my maths
and we discovered that we had ended up at the same place. There really wasn't
any difference between doing Bayesian analyses the way I thought one ought
to do Bayesian analyses and doing data compression the way he thought you
ought to do data compression. Great light dawned\ldots
\end{quote}

By the time of the writing of the other seminal paper of MML,
 \citet{WallaceFreeman1987contribution,WallaceFreeman1987}, where the
Wallace-Freeman estimator was first introduced, the information-theoretic view
of MML was already more deeply entrenched, but the authors nevertheless
use justifications that align remarkably well with the justifications for
error intolerant estimation.

They note, for example, as a weakness of the Bayes estimator approach that it
can yield an estimate outside of the given parameter space, and stress the
usefulness of their method in settling on a single hypothesis rather than on
a weighted mixture of possibilities, even bringing up, as we do here, the
need for this in scientific inference. They write:
\begin{quote}
\verb|[|W\verb|]|ould we be happy with a scientist who proposed a Bayesian
mixture of a
countably infinite set of incompatible models for electromagnetic fields?
\end{quote}

Thus, the early motivations for MML align remarkably well with our motivations
in deriving error intolerant estimation. However, they rely on a cascade of
approximations, first from Discrete MAP to SMML, and then from SMML to WF.
Regarding the second of these, \citeauthor{WallaceFreeman1987} write:
\begin{quote}
Even if the approximations used in this paper cannot be justified, the general
concepts can still be applied.
\end{quote}

In other words, the authors acknowledge the lack of theoretical justification
for their approximations, leave room for the possibility that better
approximations can be found, and stress that the importance is in the
underlying concepts that these results embody.

Throughout the years, other good properties were alleged for SMML and WF,
which would have provided additional justification for them. For example, a
general form of consistency for MML results was argued by
\citet{dowe1997resolving}. However, these good properties have since been
refuted for both WF and Strict MML itself
\citep{brand2019mml}.

By 2005, when \citet{Wallace2005}, the canonical text of MML literature, was
published, the information-theoretic view of MML was the orthodox view. Thus,
\citet{Wallace2005} describes Strict MML as an \emph{exact} solution, not as
a derivative of Discrete MAP, and WF as an approximation thereof.

In the same year, \citet{comley200511} advocated a simple recipe for the MML
practitioner to follow: if one's estimation problem has continuous parameters,
one should use WF; if it is discrete---use Discrete MAP. Thus, while
originally (working with a Bayesian motivation) Discrete MAP was the ideal
starting point, and SMML---a derivative thereof, now, from an
information-theoretic lens, SMML was seen as the ideal, and both Discrete MAP
and WF were viewed as its approximations.

Figure~\ref{F:comparison} summarises this discussion, providing a graphical
depiction of the relationships between Error Intolerant Estimation and MML,
in their
view of the Wallace-Freeman estimator: while MML sees the connections between
Discrete MAP and WF as approximative and going through SMML as an
intermediary, whether in the original Bayesian view in which these results
were first presented or in the information-theoretic view that was later
adopted, Error Intolerant Estimation derives exactly both Discrete MAP and the
Wallace-Freeman estimator (which is not an ``approximation'' in this context),
and does so using purely Bayesian reasoning and the presented axioms.

\begin{figure}[htb!]
\begin{center}
\begin{tikzpicture}[scale=0.5]
\node[draw,rectangle, fill=gray!10, minimum width=2.5cm, minimum height=1cm](DMAP) at (0,12) {Discrete MAP};
\node[draw,rectangle, fill=gray!10, minimum width=2.5cm, minimum height=1cm](SMML) at (17,12) {Strict MML};
\node[draw,ellipse, fill=black, text=white, minimum width=3cm, minimum height=1.8cm, align=center](RA) at (0,0) {Error Intolerant\\ Estimation};
\node[draw,rectangle, fill=gray!10, minimum width=2.5cm, minimum height=1cm](WF) at (17,0) {Wallace-Freeman};

\draw[-latex, ultra thick,decorate,decoration={snake,post=lineto,post length=8pt},gray!50] ([yshift=3mm]DMAP.east) -- ([yshift=3mm]SMML.west) node [midway, above] {\citep{wallace1975invariant}};
\draw[-latex, ultra thick,decorate,decoration={snake,post=lineto,post length=8pt},gray!80] ([yshift=-3mm]SMML.west) -- ([yshift=-3mm]DMAP.east) node [midway,below] {\citep{comley200511}};
\draw[-latex, ultra thick,decorate,decoration={snake,post=lineto,post length=8pt},gray!50] ([xshift=7.5mm]SMML.south) -- ([xshift=7.5mm]WF.north) node [midway,sloped,above] {\citep{WallaceFreeman1987}};
\draw[-latex, ultra thick,decorate,decoration={snake,post=lineto,post length=8pt},gray!80] ([xshift=-7.5mm]SMML.south) -- ([xshift=-7.5mm]WF.north) node [midway,sloped,below] {\citep{Wallace2005}};
\draw[-latex, line width=2mm] ([yshift=-0.2mm]RA.north) -- (DMAP);
\draw[-latex, line width=2mm] ([xshift=-1mm]RA.east)  -- (WF);
\end{tikzpicture}
\end{center}
\caption{A graphical comparison of Error Intolerant Estimation's derivation of
the Wallace-Freeman estimator with MML's. Squiggly lines represent
approximations and straight lines represent exact derivations. In light grey:
When Strict MML was first introduced in \citet{wallace1975invariant},
Discrete MAP was taken as the ideal estimator, and SMML was derived as an
approximation. Similarly, when WF was first introduced in
\citet{WallaceFreeman1987}, it was a further, secondary approximation, but
still in keeping with the view that these estimators formalise the ideal of
finding a ``most plausible'' estimate. In dark grey: More modern MML texts
\citep{comley200511,Wallace2005} consider Strict MML to
be the ideal estimator, and consider both Discrete MAP and WF its
approximations. In black: Error Intolerant Estimation, on the other hand,
derives both Discrete MAP and WF directly and independently from its axiomatic
principles, and requires for this neither approximations nor
intermediaries.}\label{F:comparison}
\end{figure}
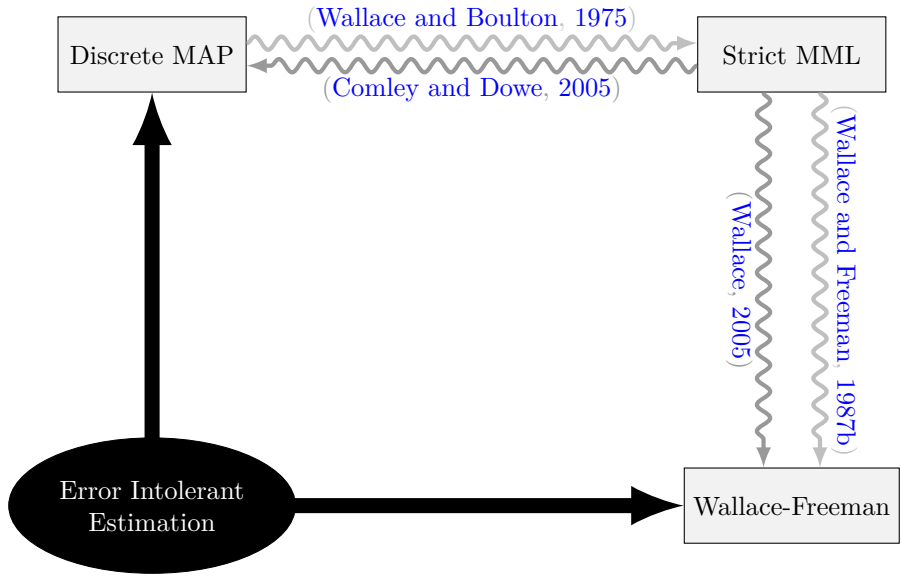

We believe this places the methodology of working with Discrete MAP in the
discrete setting and with Wallace-Freeman estimation in a continuous setting
on much stronger theoretical footing.

It not only explains the many successes of this recipe over the past 30
years, of which recent examples include
\citet{bregu2021online, bregu2021mixture,hlavavckova2020heterogeneous,bourouis2021color,sumanaweera2018bits,schmidt2016minimum,saikrishna2016statistical,jin2015two,kasarapu2014representing},
it also explains the recipe's failings, because it provides robust reasoning
regarding why the method works, and therefore when we should expect it to
underperform.

For example, in \citet{webb2016multiple} the WF results are less competitive,
but this is clear when analysed from an Error Intolerant Estimation perspective:
the paper utilises WF in a context where inferences have to be made piecewise,
serially. The success criterion in such a study is ``best-fit for later
inferences'', not an error intolerant choice of a most-plausible ``truth''.
Because of this, the error intolerance criterion is ill-fitting for the task,
and the WF results, being overly conservative, were outperformed by other
methods. MML reasoning alone does not predict this.

\FloatBarrier

\subsection{Error intolerance and trade-off based methods}\label{SS:tradeoff}

In the main paper, we describe error intolerant estimation as a ``bet-the-ranch''
type situation, where every mistake is equally and prohibitively bad, and
mention that trade-off based methods are an ill-fit for such a scenario.

In this supplementary section, we expand on this description in greater
detail, and provide intuition regarding why setting all mistakes as equally
bad avoids such trade-offs.

\subsubsection{A trade-off example}\label{SS:motivation}

Consider a physicist measuring extremely low electrical charges as part of an
experiment. The physicist is trying to determine whether only one electron
was discharged as part of the experiment ($\theta=1$), or whether it was two
electrons ($\theta=2$).
For simplicity, let us assume these are the only choices, and that there
is no reason to consider either option more likely a priori.

If the physicist's equipment measures the electric charge discharged in the
experiment as equivalent to $1.5$ electrons (which we'll take to mean that
the probability of either event is the same) and if the costs of either
possible type of mistake is the same, most estimation methods will have
difficulty deciding which way to go.

\citet{lamport2012buridan} named this problem \emph{Buridan's principle}, in
reference to the dilemma of Buridan's ass (named after the fourteenth century
French philosopher Jean Buridan) where an ass placed equidistant between two
bales of hay must starve to death because it has no reason to choose one bale
over the other.

When using Bayes estimators, however, the problem becomes much worse.
Consider \emph{any} value, $x$, reported by the measuring equipment.
Consider \emph{any} non-trivial joint probability distribution on
$(\boldsymbol{x}, \boldsymbol{\theta})$. For example, suppose
that the posterior distribution given the data is
$\mathbf{P}(\boldsymbol{\theta}=1|x)=0.9$ and
$\mathbf{P}(\boldsymbol{\theta}=2|x)=0.1$.
Consider, further, any choice of a loss function (as long as loss increases
super-linearly with error magnitude). For example, let us assume quadratic
loss:
\[
L(\theta_1,\theta_2)=|\theta_1-\theta_2|^2.
\]

Recall that a Bayes estimator is defined as a minimiser of expected loss:
\[
\hat{\theta}^L_{\text{Bayes}}(x)\defeq\argmin_{\theta\in\Theta} \mathbf{E}[L(\boldsymbol{\theta},\theta)|\boldsymbol{x}=x].
\]
So, for example, the expected loss for the choice $\hat{\theta}(x)=1$ is
\[
0.9\times L(1,1) + 0.1\times L(2,1) = 0.9 \times 0 + 0.1 \times 1 = 0.1,
\]
which is clearly superior to the expected loss for the choice $\hat{\theta}(x)=2$,
\[
0.9\times L(1,2) + 0.1\times L(2,2) = 0.9 \times 1 + 0.1 \times 0 = 0.9.
\]
Given that our chosen $L$ is defined not only over $\Theta\times\Theta$ but
rather over $\mathbb{R}^M\times\mathbb{R}^M$, let us, however, consider what
would happen if we extend the definition of our Bayes estimator to
\[
\argmin_{\theta\in\mathbb{R}^M} \mathbf{E}[L(\boldsymbol{\theta},\theta)|\boldsymbol{x}=x]
\]
(noting that in
general Bayes inference, the space of actions, which is here the possible
estimation values, is unrelated to the hypothesis space, $\Theta$.)

When all $\mathbb{R}^M$ is considered, the optimal choice is
$\hat{\theta}(x) = 1.1$, for which the expected loss is only
\[
0.9\times L(1, 1.1) + 0.1\times L(2, 1.1) = 0.9 \times 0.01 + 0.1 \times 0.81 = 0.09.
\]
So, in this case, such extended Bayes estimation would have advocated a value
in the open interval $(1,2)$, a value that clearly cannot be the true value of
$\theta$ because it lies outside $\Theta$.

One can repeat this same example with any choice of observations, any choice of
positive posterior probabilities, and any choice of loss function that is
super-linear with distance, and the estimate will never be within
$\Theta$.

This is a general problem with trade-off-based estimation methods, and while
it is particularly jarring in cases of discrete estimation or of continuous
estimation over a non-convex $\Theta$, where the estimation can fall outside
of $\Theta$ altogether, the issue exists in all estimation problems, and does
not require any extension to the definition of Bayes estimation. Consider
exactly the same situation as before, except this time let us define
$\Theta=[1,2]$, and choose our prior so that the probability that $\theta$
is in the open interval $(1,2)$ is negligible.\footnote{This could model, e.g.,
a situation in which the scientist, being a scientist, allows some tiny
probability to the possibility that everything we think we know about the
discrete nature of electric charges will ultimately be proven false.}
The Bayes estimate (this time, without any extension to its definition)
will be the same as in the extended Bayes one: $1.1$, in our example.

Thus, a Bayes estimator is able
to choose a highly unlikely $\theta$, simply because it is a convenient
trade-off between other $\theta$ values, each of which is far more likely
than the one actually chosen.
In error intolerant estimation we specifically aim to avoid such
trade-off based decisions.

\subsubsection{How error intolerance limits trade-offs}

In the proofs of both Theorem~2.6 and Theorem~\ref{T:tcontinuous}, we
transform the risk attitude functions, $T_\epsilon$, into new functions,
$A_\epsilon$, that map from objective loss to subjective utility.

One way to interpret this rescaling is that we use a loss function, $L$, to
determine how similar or different $\theta$ values are to each other, and then
use an attenuation function, $A_\epsilon$, to translate this divergence into a
\emph{similarity measure}: a $1$ indicates an exact match and a $0$ that
the two $\theta$ are not materially similar.

This rescaling gives a better intuitive understanding of why lowering
one's error tolerance leads to decisions more appropriate to the
error intolerance scenario: for each $\theta_1$, the neighbourhood in which
$A_\epsilon(L(\theta_1, \theta_2))$ is positive is the neighbourhood for which
mistaking $\theta_1$ for $\theta_2$ is at all an acceptable error. Beyond that,
the utility of choosing a given $\theta_2$ is zero.
Lowering the error tolerance contracts this neighbourhood of partial similarity;
at the limit, anything that is not ``essentially identical'' to $\theta_1$
according to the loss function contributes zero utility, the minimal possible
value. The optimal decision rule in such a scenario is therefore one that
strives, as much as possible, to avoid this possibility, and therefore
maximises the probability that its decision is, essentially, ``exactly right''.

There is, in particular, no reason for such an estimator to ever
produce an estimate $\hat{\theta}(x)$ outside $\Theta$ even when such values
are considered, because the
neighbourhood of $\theta$ values for which
$A_\epsilon(L(\theta, \hat{\theta}(x)))$ is positive
will ultimately shrink to exclude all $\theta\notin\Theta$, thus yielding for
such a decision a utility of zero.

It is interesting to note that while the loss function, $L$, retains in
error intolerant estimation its basic meaning from Bayes estimation,
it changes how it impacts the estimator. Unlike in standard Bayesian
risk minimisation, in our scenario the only parts of the loss function that are
of interest are those within an $\epsilon$-neighbourhood of the diagonal. One
can think of this as using the loss function to convey the local topology of
the hypothesis space. This is because error intolerant estimation does not concern
itself with errors larger than any $\epsilon$: these are all considered equally
``wrong''.

\subsection{Rationale of the axioms on loss}\label{SS:rationale}

In the main paper we introduce four axioms on the choice of a loss function
to be used in error intolerant estimation.
In this section, we provide a deeper analysis of the rationale and
incontrovertibility of each.

\subsubsection{IRP}

Our first axiom is Invariance to Representation of Parameter Space (IRP).
Informally, it states that our loss evaluation (and therefore
our choice of estimate) should not depend on how we name our hypotheses.

While the IRP axiom may seem highly intuitive, it is actually the one axiom
among our four for which it is least obvious why it is ``incontrovertible''.
After all, common loss functions in use, such as quadratic loss, do not satisfy
it.

Not only that, but the very fact that common loss functions do not satisfy
IRP is often put to good use. When estimating, say, the size of a cohort
it makes a big difference whether one is interested in getting the answer
right in terms of the minimum error in number of people or in terms of
minimum error in percents. One favours the use of a linear representation,
the other---a logarithmic one.

As another example, consider estimation of a probability. When estimating
the probability that a person is likely to respond to a marketing message,
the error needs to be measured on a linear scale, because the underlying
analysis is about maximising the expected size of the total cohort that
will respond to the message. On the other hand, when estimating the
probability that a dangerous experiment has crossed its safety bounds, the
difference between a probability of $0.001\%$ and $0.0001\%$ is far more
substantial than between $4\%$ and $5\%$. Again, a logarithmic scale
becomes more appropriate, in order to accentuate the difference between
near-zero figures. In yet other domains, both near-zero and near-one
probability ranges will need to be magnified, and a log-odds metric may
become appropriate.

In light of this, what is the justification for IRP?

The important element to note is that in all these examples the choice of
estimation method was led by a desire to manage errors in particular ways.
This is not the case in error intolerant estimation. Here, we are in a
bet-the-house situation. Everything rests on the question of whether our
estimate is \emph{correct}, not how well it approximates the correct answer.

In such a situation, a person faced with, for example, the need to estimate
the side length of a cube and estimating it to be, say, $2$ metres in length
will undoubtedly answer $8$ cubic metres if asked to estimate the volume of
said cube. Both answers must be informed by that person's view of what the
most likely ``truth'' is regarding the dimensions of the cube, which is why
this truth must always be the same, regardless of how the person is asked
to express it.

We see, therefore, that while IRP is not universally incontrovertible, and
is, in fact, frequently violated, that is not the case within the world
of error intolerant estimation. Here, there is no reason to doubt its veracity.

\subsubsection{IRO}

Our second axiom, Invariance to Representation of Observation Space (IRO),
is also regarding invariance to representation, but
this time regarding the representation of observation space.
Informally, IRO guarantees that the estimate should remain the same
regardless of how the observations are presented to us.

Is IRO incontrovertible?

Statisticians analysing real-world data spend much of their time wrangling this
data. This wrangling, bringing the raw input data into a usable format, is
a process that only works under the assumption that the raw formatting of the
input does not matter: regardless of how the input is initially presented, one
spends the necessary effort to bring the information into the format most
appropriate for the analysis of the problem.

This is precisely the statement of IRO: the choice of the raw input format
should not influence the insights derived from the data.

\subsubsection{IIA}

The third axiom, Invariance to Irrelevant Alternatives (IIA), takes its name
from \citet{nash1950bargaining}, a seminal work in game theory, where it
was introduced in the following formulation: if, when faced with a choice
among alternatives in a set $\Theta$, a certain rational player would choose
$\theta\in\Theta$, then, with all other conditions being equal, the same
rational player would make the same choice if the alternatives
were narrowed down to some $\Theta'\subset\Theta$, as long as
$\theta\in\Theta'$. For example, if, given any evidence $x$, we estimate
that the best place to dig for gold in Australia is, say, in Melbourne,
then IIA states that we should
reach the same conclusion, given the same evidence, also if the question asked
is where to dig for gold in only \emph{mainland} Australia.

Our IIA axiom merely reformulates
this principle so that it addresses the underlying loss function
rather than the estimate itself.

Informally, our formulation of IIA states that the loss experienced when
choosing hypothesis $\theta_2$ in cases where the true value is $\theta_1$
should only relate to the nature of these two hypotheses, not any other
property of the estimation problem. This follows directly from the semantics
of what loss function values are meant to represent.

The IIA axiom has been a hotly debated axiom ever since its introduction
in 1950. In our context, however, it is a direct statement regarding the nature
of the error intolerant scenario: we aim for an estimate that is ``best''
according to its own merits, not based on its relationship with other potential
estimates (as is the case with trade-off-based methods). This was already
stipulated by AIA in the context of the estimator itself, as parameterised
by a loss function, and IIA merely extends this requirement also to the
underlying loss function, and thus to the estimator as a whole.

\subsubsection{ISI}

What ISI states is that merely adding data bits to the observation that do
not convey any information about $\boldsymbol{\theta}$ should not change
our estimate of it.

Like our other axioms, ISI, too, seems nearly
tautological, as one would be hard put to justify why such extra,
information-free data should influence a point estimate in any way.

\bibliographystyle{ba}
\bibliography{bibwf_SI}